\documentclass[Unicode]{cedram-alco}

\usepackage{graphicx}
\usepackage{setspace}

\usepackage{color}
\usepackage{booktabs}
\usepackage{tabularx}

\usepackage{tikz}
\usetikzlibrary{positioning,automata,arrows,shapes,matrix, arrows, decorations.pathmorphing}
\tikzset{main node/.style={circle,draw,minimum size=0.3em,inner
sep=0.5pt}}
\tikzset{state node/.style={circle,draw,minimum size=2em,fill=blue!20,inner
sep=0pt}}
\tikzset{small node/.style={circle,draw,minimum size=0.5em,inner
sep=2pt,font=\sffamily\bfseries}}

\setlist[enumerate,1]{label=\textup{(\arabic*)}}

\equalenv{corollary}{coro}
\equalenv{definition}{defi}
\equalenv{remark}{rema}
\equalenv{example}{exam}
\equalenv{notation}{nota}

\newcommand{\Z}{\mathbb{Z}}

\newcommand{\cala}{\mathcal{A}}
\newcommand{\calc}{\mathcal{C}}

\newcommand{\cali}{\mathcal{I}}
\newcommand{\calf}{\mathcal{F}}
\newcommand{\call}{\mathcal{L}}
\newcommand{\calr}{\mathcal{R}}
\newcommand{\cals}{\mathcal{S}}
\newcommand{\calt}{\mathcal{T}}
\newcommand{\la}{\mathbf{a}}

\newcommand{\kl}{Kazhdan--Lusztig }
\newcommand{\sx}{\mathcal{S}(W)}
\newcommand{\sw}{\mathcal{S}(W)}
\newcommand{\ul}{\underline}
\newcommand{\tul}{\textup{(}}
\newcommand{\tur}{\textup{)}}
\newcommand{\twostub}{$\la(2)$-stub }
\newcommand{\twostubs}{$\la(2)$-stubs }

\newcommand{\triples}{\mathrm{Tri}(W) }
\newcommand{\cores}{\mathrm{Cor}}
\newcommand{\ntatf}{nontrivially $\la(2)$-finite }

\newcommand{\ra}{\rightarrow}
\newcommand{\se}{\subseteq}
\newcommand{\ip}[1]{\langle#1\rangle}

\newcommand{\inverse}{^{-1}}

\newcommand{\abs}[1]{\lvert #1 \rvert}
\newcommand{\Abs}[1]{\bigg| #1 \bigg|}

\DeclareMathOperator{\supp}{\mathrm{Supp}}
\DeclareMathOperator{\fc}{\mathrm{FC}}

\author[\initial{R.} \middlename{M.} Green]{\firstname{R.} \middlename{M.} \lastname{Green}}

\address{Department of Mathematics\\
University of Colorado Boulder, Campus Box 395\\
Boulder, Colorado\\
USA, 80309
}
\email{rmg@colorado.edu}

\author[\initial{T.} Xu]{\firstname{Tianyuan} \lastname{Xu}}
\address{Department of Mathematics and Statistics\\
    Haverford College\\
    Haverford, Pennsylvania\\
USA, 19041
}
\email{txu2@haverford.edu}


\begin{document}

\title[Kazhdan--Lusztig cells of $\la$-value 2]{Kazhdan--Lusztig cells of $\mathbf{a}$-value 2 in $\mathbf{a}(2)$-finite Coxeter
systems}

\begin{abstract}
    A Coxeter group is said to be \emph{$\mathbf{a}(2)$-finite} if it has finitely many elements of
    $\mathbf{a}$-value 2 in the sense of Lusztig. In this paper, we give explicit combinatorial
    descriptions of the left, right, and two-sided Kazhdan--Lusztig cells of $\mathbf{a}$-value 2 in
    an irreducible $\mathbf{a}(2)$-finite Coxeter group. In particular, we introduce elements we
    call \emph{stubs} to parameterize the one-sided cells and we characterize the one-sided cells
    via both star operations and weak Bruhat orders.  We also compute the cardinalities of all the
    one-sided and two-sided cells of $\mathbf{a}$-value 2 in irreducible $\mathbf{a}(2)$-finite
Coxeter groups.\end{abstract} \maketitle

\section{Introduction}
Let $(W,S)$ be an arbitrary Coxeter system and let $H$ be the associated Hecke algebra.  In the
landmark paper \cite{KL}, Kazhdan and Lusztig introduced the Kazhdan--Lusztig basis of $H$ and used
it to partition the Coxeter group $W$ into left, right and two-sided cells. Each two-sided cell is a
union of left cells as well as a union of right cells. These {Kazhdan--Lusztig cells} have
connections with numerous objects from representation theory, such as primitive ideals of universal
enveloping algebras of Lie algebras \cite{primitiveIdeals}, characters of reductive groups over
finite fields \cite{LusztigReductive}, and unipotent conjugacy classes of reductive groups
\cite{LusztigCellIV}. Consequently, the determination of cells has been an important problem in
representation theory. The goal of this paper is to parameterize the one-sided cells of $\la$-value
2 in irreducible {$\la(2)$-finite} Coxeter systems, as well as to count all one-sided and two-sided
cells of $\la$-value 2 in such Coxeter systems. 

Before explaining the meaning of the ``$\la$-value'' and ``$\la(2)$-finite'', let us briefly
summarize for context some known results about Kazhdan--Lusztig cells. For Weyl groups of types $A$,
$B$, $D$ and affine $A$, the cells may be obtained via various combinatorial models involving Young
tableaux, the Robinson--Schensted correspondence, and their suitable generalizations; see
\cite{KL,Ariki} for type $A$, \cite{primitiveIdeals,Garfinkle,TypeBCells} for types $B$ and $D$, and
\cite{ShiCellBook,MatrixBall} for type affine $A$.  Cells of some other specific Coxeter systems
(especially those with small ranks or complete Coxeter diagrams) have also been computed, often  on
a case-by-case basis via carefully designed algorithms.  For example, the paper \cite{TakahashiF4}
treats cells in the Coxeter group $F_4$, \cite{ShiAffineC4} treats $\tilde C_4$, \cite{AlvisH4}
treats $H_3$ and $H_4$, \cite{ChenE8} and \cite{GeckHallsE8} treat $E_8$; the paper
\cite{GuilhotRank2} studies affine Weyl groups of rank 2, and the papers \cite{XiComplete, Xie} deal
with Coxeter systems with complete Coxeter diagrams. For a more comprehensive summary of the cell
literature, see \cite[\S 26]{BonnafeCells}. 

Given that most of the above references focus on Coxeter systems of particular types, it is perhaps
natural to wonder whether one can systematically study Kazhdan--Lusztig cells of an arbitrary
Coxeter system $(W,S)$. The \emph{$\la$-function}, a certain function $\la: W\ra \Z_{\ge 0}$, offers
an important tool in this study.  The $\la$-function is due to Lusztig, who first defined it for
Weyl and affine Weyl groups in \cite{LusztigCellI, LusztigCellII} and then extended the definition
to general Coxeter groups in \cite{LusztigHecke}.  Lusztig showed that under a certain boundedness
condition on $(W,S)$, the $\la$-function takes a constant value on every two-sided cell of $W$ and
hence on every left or right cell. Lusztig conjectures that the boundedness condition is satisified
by every Coxeter system (see \cite[\S13.4]{LusztigHecke}), and we shall assume that the conjecture
holds throughout the paper (as is common in the literature). Under this assumption, it makes sense
to speak of the $\la$-value of a cell, and the set $W_n:=\{w\in W:\la(w)=n\}$ is a union of
two-sided cells for all $n\in \Z_{\ge 0}$.  In principle, one may also hope to organize and study
cells by $\la$-values.  Indeed, cells of $\la$-values 0 and 1 in a general Coxeter system $(W,S)$
are well understood. For example, if we assume $(W,S)$ is {irreducible} (see
Remark \ref{rmk:irreducible}), then both $W_0$ and $W_1$ contain a single two-sided cell, we have
$W_0=\{1\}$, and $W_1$ consists of all non-identity elements in $W$ with a unique reduced word.
Moreover, Coxeter systems where $W_1$ is finite have been classified, and the cardinality of $W_1$
is known for such systems. For more details on the cell $W_1$, see Proposition \ref{prop:facts01},
\cite{LusztigSubregular}, \cite{Hart}, \cite[Chapter 13]{BonnafeCells} and \cite{XuSubregular}.

Motivated by the desire to understand elements and cells of  $\la$-value 2, we considered the set
$W_2=\{w\in W:\la(w)=2\}$ in \cite{GreenXu} and classified all Coxeter systems for which $W_2$ is
finite. We called these systems \emph{$\la(2)$-finite}, and our classification states that an
irreducible Coxeter system is $\la(2)$-finite if and only if it has a complete Coxeter diagram or is
of Coxeter types $A_n, B_n, \tilde C_n, E_{q,r}, F_n, H_n$ and $I_n$. The Coxeter groups of type
$E_{q,r}$ encompass all Weyl groups of type $D$, $E$ and affine $E$; see Remark \ref{rmk:labels}. We will
recall the classification as Part (4) of Proposition \ref{prop:facts01} in \autoref{sec:HeckeAlgebras}.

In this paper, we determine and count the left, right and two-sided cells within $W_2$ for all
$\la(2)$-finite Coxeter systems.  Our results center around the notion of \emph{stubs}. To define
them, recall that every pair of noncommuting generators $I=\{i,j\}\se S$ of a Coxeter group $W$
gives rise to four partially defined functions on $W$, namely, the left upper, left lower, right
upper and right lower star operations (see \autoref{sec:StarOperations}). These operations
generalize the star operations introduced in \cite{KL}. We define a \emph{left stub} in $W$ to be an
element that admits no right lower star operation with respect to any pair of noncommuting
generators; similarly, we define a \emph{right stub} to be an element admitting no left lower star
operation. We equip the set of left stubs in $W_2$ with two equivalence relations, one called
\emph{slide equivalence} and the other called \emph{simple slide equivalence}
(Definition \ref{def:equivalences}), and we will show that stubs and these two relations ``control
everything'' about the cells of $W_2$ in $\la(2)$-finite Coxeter systems. 

To be more precise, denote the set of left stubs in $W_2$ by $\sw$, denote slide equivalence by
$\approx$, and denote simple slide equivalence by $\sim$.  Call a two-sided cell a \emph{2-cell}, a
left or right cell a \emph{1-cell}, and call the intersection of a left cell with a right cell in
the same 2-cell a \emph{0-cell}. For each stub $w\in \sw$, define the \emph{right upper star
closure} to be the set $R_w$ of all elements that can be obtained from $w$ by a sequence of right
upper star operations.  Then we will describe the 1-cells and 2-cells in terms of stubs as follows:

\begin{thm}
    \label{thm:main1} Let $(W,S)$ be any irreducible $\la(2)$-finite Coxeter system. 
    \begin{enumerate}
        \item\textup{Stubs parameterize 1-cells (Theorem \ref{thm:cellParameterization} \&
            Theorem \ref{thm:CellViaDecomp}):}\\ 
            There is a bijection from $\sw$ to the set of right cells in $W_2$ that sends each stub
            $w\in \sw$ to its right upper star closure $R_w$. Moreover, the cell $R_w$ coincides with
            the set of the elements of $W_2$ that are stronger than $w$ in the right weak Bruhat order.
            \textup{(}The last fact may be viewed as an analog of Proposition \ref{prop:facts01}.\tul
            3\tur.\textup{)} 

            Similar results hold for left cells. In fact, each left cell in $W_2$ is of the form
            $R_w\inverse=\{x\inverse: x\in R_w\}$ for some $w\in \sw$. 

        \item\textup{Slide equivalence determines 2-cells (Theorem \ref{thm:twoSidedCells}):}\\ 
            Let $\mathcal{C}$ be the set of $\approx$-classes in $\sw$ and let $\mathcal{E}$ be the set of
            2-cells in $W_2$. Then there is a bijection $\Phi: \mathcal{C}\ra \mathcal{E}$ given by
            $\Phi(C)=\cup_{w\in C}R_w$ for all $C\in \mathcal{C}$. If $(W,S)$ is not of type $E_{1,r}$ for any
            $r\ge 1$, then $\sw$ contains a single $\sim$-class and $W_2$ is itself a 2-cell. If $(W,S)$ is of
            type $E_{1,r}=D_{r+3}$ \tul see Remark \ref{rmk:labels}\tur, then we can explicitly describe the sets
            $\mathcal{C}$ and $\mathcal{E}$ as well. In particular, there are three 2-cells in $W_2$ if $r=1$
            and two 2-cells in $W_2$ if $r>1$.
    \end{enumerate}
\end{thm}

We will also count all left, right and two-sided cells in $W_2$.  As explained in
Corollary \ref{coro:intersection}, these enumeration problems can be reduced to the enumeration of 0-cells by
general theory. This motivates us to examine 0-cells, which are in fact also interesting in their
own right as we will soon explain.  By Theorem \ref{thm:main1}.(1), the 0-cells in $W_2$ in an
$\la(2)$-finite Coxeter system are all of the form $I(x,y)=R_x\cap (R_y)\inverse$ where $x,y\in
\sw$. We have the following results on 0-cells and on the enumeration of all cells in $W_2$. The
relations $\approx$ and $\sim$ play crucial roles in these results.

\begin{thm}
    \label{thm:main2}
    Maintain the setting and notation of Theorem \ref{thm:main1}.
    \begin{enumerate}
    \item \textup{Elements of $W_2$ decompose nicely (Theorem \ref{thm:anatomy} \&
         Corollary \ref{coro:intersectionsViaCores}):}\\ Each element in $W_2$ can be canonically
         decomposed into an ``$\la(2)$-triple'' \tul Definition \ref{def:anatomy}\tur. The decomposition
         gives rise to a bijection between the set of $\la(2)$-triples and $W_2$, and we can use
         it to characterize each 0-cell $I(x,y)$ where $x,y\in\sw$ via a set of ``core elements
         compatible with $x$ and $y$'' \tul Definition \ref{def:anatomy}\tur.
    \item \textup{Different 0-cells can be related (Remark \ref{rmk:relating0cells} \&
         Proposition \ref{prop:invariance}):}\\ For all $x,y,x',y'$, we can use Part \tul 1\tur\ and the
         relation $\approx$ to obtain the 0-cells $I(x,y)$ and $I(x',y')$ from each other
         whenever they are in the same 2-cell. In addition, we have
         $\abs{I(x,y)}=\abs{I(x',y')}$ whenever $x\sim x'$ and $y\sim y'$.
    \item \textup{We can count all 0-cells (Theorem \ref{thm:oneI} \& Proposition \ref{prop:intersectionSizes}):}\\ We
         can compute one particular 0-cell $I(x,y)$ of each 2-cell in $W_2$ and then use Part \tul
         2\tur\, to compute and count all 0-cells in $W_2$. 
    \item \textup{We can count all 1-cells and 2-cells (Theorem \ref{thm:cellSizes}):}\\ We can use Part
         \tul 3\tur\, to compute the cardinalities of all left, right and two-sided cells in $W_2$.
         The cardinalities are given by Table \ref{tab:cellSizes}.
    \end{enumerate}
\end{thm} 
\noindent

The results from Theorems \ref{thm:main1} and \ref{thm:main2} have applications for {distinguished
involutions} and {cell modules}.  Here, distinguished involutions are certain special involutions in
Coxeter groups, and each 1-cell contains a unique distinguished involution (see \cite[Chapter
14]{LusztigHecke}).  The canonical decompositions mentioned in Theorem \ref{thm:main1} lead to a natural
description of all distinguished involutions in $W_2$, as we will explain in
Remark \ref{rmk:distinguishedInvolutions}. Cell modules are modules of Hecke algebras associated to
Kazhdan--Lusztig cells. In the simply-laced $\la(2)$-finite Coxeter types, namely types $A_n$ and
$E_{q,r}$, we will show that each 0-cell contains exactly one element. In a forthcoming paper, we
will exploit this fact to interpret the left cell module afforded by $L$ in terms of the {reflection
representation} of the Coxeter group. Doing so realizes the cell module without needing input from
the Kazhdan--Lusztig basis used to define it. 

Our results are also helpful for understanding the \emph{$J$-rings} and their categorical analogs
attached to Coxeter systems. Introduced by Lusztig \cite{LusztigCellII, LusztigHecke}, the $J$-ring
of a Coxeter system decomposes into a direct sum of subrings corresponding to the two-sided cells of
the group, and the structures of these subrings can sometimes be deduced from knowledge of 0-cells.
For example, since the 0-cells in $W_2$ are all singletons in types $A_n$ and $E_{q,r}$ as we just
mentioned, the subrings of the $J$-ring corresponding to the two-sided cells in $W_2$ must be matrix
rings by general theory.  We hope to discuss similar subrings of $J$-rings in other $\la$(2)-finite
Coxeter types in future work.  In \cite{Mazorchuk1, Mazorchuk2, Mazorchuk3}, Mazorchuk and his
coauthors studied 2-representations of categorical analogs of the $J$-ring and of its subring
corresponding to the cell $W_1$.  Facts about $W_1$, including the classification of Coxeter groups
where the cell is finite and combinatorial descriptions of the 1-cells in $W_1$, proved useful in
these studies. We hope our investigation of the cells in $W_2$ would be useful in a similar way.  

Let us comment on the key ingredients and methods used in our study of cells in $W_2$. First, we
recall our observation from \cite{GreenXu} that elements in $W_2$ are {fully commutative} in the
sense of Stembridge \cite{FC}.  Each fully commutative element has a {Cartier--Foata canonical form}
and an associated poset called a {heap}, both of which are essential to our analysis of elements in
$W_2$. In particular, antichains in heaps and the {widths} of heaps (see \autoref{sec:FC}) play
important roles in our proofs. A result of Ernst on fully commutative elements in the Coxeter groups
$\tilde C_{n-1}$ is also key to our determination of those particular 0-cells mentioned in
Theorem \ref{thm:main2}.(3); see Lemma \ref{lemm:Dana} and \autoref{sec:proofs}. 

Next, we remark that the determination and enumeration of cells in $W_2$ can often be done in
multiple ways.  For example, it is possible to use the {Robinson--Schensted correspondence} or
{Temperley--Lieb diagrams} to find and count cells of $W_2$ in type $A$.  Various {generalized
Temperley--Lieb diagrams} also exist in types $B, \tilde C, D, E$ and $H$ (see, for example,
\cite{ErnstTL,GreenTangles}) and can be indirectly applied to study cells of $W_2$, although in type
$F$ such diagrams have not been defined to our knowledge.  After obtaining the stub
characterizations of 1-cells in each $\la(2)$-finite Coxeter system, one can also count the cells
with brute-force computation by considering Cartier--Foata forms or heaps. Finally, since we are
interested in the case where $W_2$ is a finite set, we were able to automate the computation and
counting of cells with a computer, even in cases where $W$ is infinite.  We do not elaborate on
these alternative methods, however, since our approach via stubs seems the most general, conceptual
and convenient overall.

Last but not least, we should mention that some of our results may seem reminiscent of those in the
paper \cite{FanStructure} by Fan, where the author studied certain cells defined via the so-called
{monomial bases} of the {generalized Temperley--Lieb algebras} of certain Coxeter systems.  These
Coxeter systems do not include those of type $\tilde C_n$ or $E_{q,r}$ (for general values of $q,r$)
that are studied in this paper. We also note that, as pointed out in \cite[\S 4.1]{FanStructure} and
explained in \cite{BremkeFan, GreenLosonczy, GL}, monomial cells are different from Kazhdan--Lusztig
cells in general, even for general finite Coxeter groups. For example, we will see in Examples
\ref{eg:b4} and \ref{eg:0cellsB4} that the left Kazhdan--Lusztig cells of $\la$-value 2 in the
Coxeter group of type $B_4$ have sizes 8 and 10, but the left monomial cells of $\la$-value 2 in the
same group have sizes 2, 4 and 6 by \cite[\S 7.1]{FanStructure}. In view of the above facts, we do
not attempt to use the results of \cite{FanStructure} in any form in this paper. 

The rest of the paper is organized as follows. In \autoref{sec:preliminaries} we recall the necessary
background for this paper on Coxeter systems, Hecke algebras, fully commutative elements and star
operations. We also recall the classification of $\la(2)$-finite Coxeter systems in
\autoref{sec:HeckeAlgebras}. In \autoref{KLcells}, we define stubs, describe them, and explain how
stubs parameterize 1-cells in $W_2$. We also introduce the equivalence relations $\approx, \sim$ on
stubs and use the former relation to find the two-sided cells of $W_2$ in \autoref{sec:2cells}.
Section \ref{sec:0cells} investigates 0-cells in detail and explains how to deduce the cardinalities of all
cells in $W_2$ via the 0-cells. The starting point of the deduction is Theorem \ref{thm:oneI}, which is
stated without proof in \autoref{sec:cellData}. The entire \autoref{sec:rep0cells} is dedicated to the
proof of this theorem. We prepare a set of technical lemmas on heaps in \autoref{sec:heapLemmas} and
use them to complete the proof in \autoref{sec:proofs}.

\section{Background}
\label{sec:preliminaries}

We review some preliminary facts on Coxeter groups and Hecke algebras relevant to the paper,
including the definition of Kazhdan--Lusztig cells, the definition of the $\la$-function, and the
classification of $\la(2)$-finite Coxeter systems. We then recall the notions of fully commutative
elements and generalized star operations, both of which will be essential to the study of
Kazhdan--Lusztig cells of $\la$-value 2.  

\subsection{Coxeter systems}
\label{sec:CoxeterSystems}
Throughout the paper, $(W,S)$ stands for a Coxeter system with a finite generating set $S$, Coxeter
group $W$ and Coxeter matrix $M=[m(s,t)]_{s,t\in S}$. Thus, we have $m(s,s)=1$ for all $s\in S$,
$m(s,t)=m(t,s)\in \Z_{\ge 2}\cup\{\infty\}$ for all distinct generators $s,t\in S$, and $W$ is the
group generated by $S$ subject to the relations $\{(ss')^{m(s,s')}=1, s,s'\in S, m(s,s')<\infty\}$.
It is well known that $st$ has order $m(s,t)$ for all distinct $s,t\in S$; in particular, $s$ and
$t$ commute if and only if $m(s,t)=2$. 

The \emph{Coxeter diagram} of $(W,S)$ is the undirected diagram $G$ on vertex set $S$ where two
vertices $s,t$ are connected by an edge $\{s,t\}$ if and only if $m(s,t)\ge 3$. It follows that two
distinct generators in $S$ commute in $W$ if and only if they are not adjacent in $G$. For each edge
$\{s,t\}$ in $G$, we think of $m(s,t)$ as the \emph{weight} of the edge, call the edge \emph{simple}
if $m(s,t)=3$, and call the edge \emph{heavy} otherwise.  When drawing $G$, we label all heavy edges
by their weights, so that the defining data of $(W,S)$ can be fully recovered from the drawing.  The
system $(W,S)$ is \emph{irreducible} if the Coxeter diagram $G$ is connected and \emph{reducible}
otherwise. 

\begin{remark}
\label{rmk:irreducible}
To simplify statements, we shall assume all Coxeter systems to be irreducible for the rest of the
paper.  The assumption is well justified as far as the $\la$-function is concerned, because the
$\la$-function behaves additively across connected components of the Coxeter graph. As a
consequence, Kazhdan--Lusztig cells of a particular $\la$-value in a reducible Coxeter system can be
easily obtained via cells of the same or lower $\la$-values in suitable irreducible Coxeter systems.
Since we are only interested in cells of $\la$-value 2, and cells of $\la$-values 1 or 0 in
irreducible Coxeter systems are well understood as explained in the introduction, it suffices to
consider only irreducible Coxeter systems as we study cells of $\la$-value 2 in this paper; see also
\cite[Theorem 1.3]{GreenXu}.
\end{remark}

For every subset $J$ of $S$, the subgroup $W_J:=\ip{s:s\in J}$ of $W$ is called the \emph{parabolic
subgroup generated by $J$}. Note that if for another Coxeter system $(W',S')$ there is a bijection
$f: S'\ra J$ with $m(s,t)=m(f(s),f(t))$ for all $s,t\in S'$ (in other words, if the Coxeter diagram
of $(W',S')$ is isomorphic to the subgraph of the Coxeter diagram of $(W,S)$ induced by $J$), then
this bijection naturally extends to a group isomorphism from $W'$ to $W_J$. For example, in the
notation of Proposition \ref{prop:facts01}.(4) the Coxeter group $A_{n-1}$ naturally embeds into $B_{n}$ for
all $n\ge 2$ in the sense that $A_{n-1}$ is isomorphic to the parabolic subgroup generated by the
set $J=\{2,\dots,n\}$. 

Let $S^*$ be the free monoid on $S$, viewed as the set of words on the alphabet $S$. To distinguish
words in $S^*$ from elements of $W$, we denote each word with an underline.  For example, the words
$\ul u=st, \ul v=ts\in S^*$ represent the same element $w=st=ts$ in $W$ if $s,t$ are commuting
distinct generators in $S$.  For each $w\in W$, the words $\ul w\in S^*$ that express $w$ with a
minimum number of letters are called the \emph{reduced words} of $w$; that minimum number is called
the \emph{length} of $w$ and written $l(w)$. More generally, we call a factorization $w=w_1w_2\dots
w_k$ in $w$ \emph{reduced} if $l(w)=\sum_{i=1}^k l(w_i)$. In this paper, the notation $w=w_1\cdot
w_2\cdot \dots \cdot w_k$ will always indicate that the factorization $w=w_1w_2\dots w_k$ is
reduced.  Whenever we have $z=x\cdot y$ for elements $x,y,z\in W$, we write $x\le^R z$ and $y\le^L
z$. The relations $\le^R$ and $\le^L$ define two partial orders called the \emph{right weak Bruhat
order} and \emph{left weak Bruhat order} on $W$, respectively.

For each $w\in W$, we define a \emph{left descent} of $w$ to be a generator $s\in S$ such that
$l(sw)<l(w)$, and we denote the set of left descents of $w$ by $\call(w)$. Similarly, a \emph{right
descent} of $w$ is a generator $s\in S$ such that $l(ws)<l(w)$, and we denote the set of right
descents of $w$ by $\calr(w)$. Descents can be described via reduced words: it is well known that
for all $s\in S$ and $w\in W$, we have $s\in \call(w)$ if and only if some reduced word of $w$
begins with the letter $s$, and $s\in \calr(w)$ if and only if some reduced word of $w$ ends with
the letter $s$.  

We end the subsection by recalling the Matsumoto--Tits Theorem. Define a \emph{braid relation} to be
a relation of the form $sts\cdots=tst\cdots$ where $s,t\in S, 2\le m(s,t)<\infty$, and both sides
contain $m(s,t)$ factors. Then the theorem states that for all $w\in W$, every pair of reduced words
of $w$ can be obtained from each other by applying a finite sequence of braid relations. It follows
that every two reduced words of an element $w$ contain the same set of letters. We call the set the
\emph{support} of $w$ and denote it by $\supp(w)$.

\subsection{Cells and the \texorpdfstring{$\la$}{a}-function}
\label{sec:HeckeAlgebras}
Let $\cala=\Z[v,v\inverse]$ and let $H$ be the \emph{Hecke algebra} of a Coxeter system $(W,S)$.  We
recall that $H$ is a unital, associative $\cala$-algebra given by a certain presentation and that
$H$ has an $\cala$-linear basis $\{C_w:w\in W\}$ called the \emph{Kazhdan--Lusztig} basis; see
\cite[\S 2.2]{XuSubregular}. The Kazhdan--Lusztig basis gives rise to Kazhdan--Lusztig cells and the
$\la$-function as follows.

Let $x,y\in W$. For all $z\in W$, let $D_z: H\ra\cala$ be the unique linear map such that
$D_z(C_w)=\delta_{z,w}$ for all $w\in W$, where $\delta$ is the Kronecker delta symbol.  Write
$x\prec_L y$ or $x\prec_R y$ if $D_x(C_sC_y)\neq 0$ or $D_{x}(C_yC_s)\neq 0$ for some $s\in S$,
respectively, and write $x\prec_{LR}y$ if either $x\prec_L y$ or $x\prec_R y$.  The preorders
$\le_L,$ $\le_R,$ $\le_{LR}$ defined as the reflexive, transitive closures of the relations
$\prec_L,\prec_R, \prec_{LR}$ naturally generate equivalence relations $\sim_L, \sim_R, \sim_{LR}$,
and we call the resulting equivalence classes the  \emph{left, right and two-sided Kazhdan--Lusztig
cells} of $W$, respectively.  For example, we have $x\sim_L y$, i.e. $x$ and $y$ are in the same
left Kazhdan--Lusztig cell, if and only if $x\le_L y$ and $y\le_L x$, which happens if and only if
there are sequences $x=w_0,w_1,\dots, w_n=y$ and $y=z_0,z_1,\dots, z_k=x$ where $n,k\ge 0$ and
$w_i\prec_L w_{i+1}, z_j\prec_L z_{j+1}$ for all $0\le i<n, 0\le j<k$.  For brevity, from now on we
will refer to Kazhdan--Lusztig cells simply as ``cells'', and we call both left and right cells
\emph{one-sided} cells. By the definition of $\prec_{LR}$, each two-sided cell of $W$ is a union of
left cells, as well as a union of right cells. 

To define the $\la$-function from the basis $\{C_w:w\in W\}$, consider the structure constants
$h_{x,y,z}\in \cala \, (x,y,z\in W)$ such that 
\begin{equation}
    \label{eq:h} C_x C_y =\sum_{z\in W} h_{x,y,z} C_z 
\end{equation} for all $x,y\in W$. By \S 15.2 of \cite{LusztigHecke}, for each $z\in W$ there exists
    a unique integer $\la(z)\ge 0$ such that 
\begin{enumerate}
    \item[(a)] $h_{x,y,z}\in v^{\la(z)}\Z[v\inverse]$ for all $x,y\in W$;
    \item[(b)] $h_{x,y,z}\notin v^{\la(z)-1}\Z[v\inverse]$ for some $x,y\in W$. 
\end{enumerate} The assignment $z\mapsto \la(z)$ defines the function $\la: W\ra \Z_{\ge 0}$.

In the sequel, we will make extensive use of the following standard results on cells and the
$\la$-function. We note that the proofs of Parts (5)--(9) of the proposition either directly or
indirectly rely on the use of Lusztig's boundedness conjecture \cite[\S 13.4]{LusztigHecke}. The
conjecture states that for the standard basis $\{T_w:w\in W\}$ of the Hecke algebra (see \cite[\S
2.2]{XuSubregular}) and the structure constants $f_{x,y,z}$ such that $T_xT_y=\sum_{z\in
W}f_{x,y,z}T_z$ for all $x,y\in W$, there is a nonnegative integer $N$ such that $v^{-N}f_{x,y,z}$
is an element of the ring $\Z[v\inverse]$ for all $x,y,z\in W$.  In the following proposition, Parts
(5)--(9) contain statements from conjectures P9--P12 and P14 in \cite[\S 14]{LusztigHecke}, and the
proofs of these five conjectures in turn rely on Lusztig's boundedness conjecture (despite the fact
    that conjectures P1--P6 from \cite[\S 14]{LusztigHecke} are known to hold in the setting of our
current paper, where the Hecke algebra has ``equal parameters'', without reliance on the boundedness
conjecture); see also \cite[\S 15.1]{LusztigHecke}, \cite[\S 14.3]{BonnafeCells} and \cite[Theorem
15.2.5]{BonnafeCells}.

\begin{prop}
     \label{prop:Facts}
     Let $(W,S)$ be a Coxeter system and let $x,y\in W$.
     \begin{enumerate}
         \item If $x\le^L y$ then $y\le_L x$; if $x\le^R y $ then $y\le_R x$.
         \item If $x$ is a product of $k$ pairwise commuting generators in $S$, then $\la(x)=k$.
         \item We have $\call(x)=\call(y)$ if $x\sim_R y$, and $\calr(x)=\calr(y)$ if $x\sim_L y$;
         \item For each left cell $L$ in $W$, the set $L\inverse=\{z\inverse: z\in L\}$ is a right cell; for each
             right cell $R$ in $W$, the set $R\inverse=\{z\inverse:z\in R\}$ is a left cell. 
         \item If $x\le_{LR} y$, then we have $\la(x)\ge \la(y)$. In particular, if $x\sim_{LR}y$, then
             we have $\la(x)=\la(y)$. Moreover, if $x\le_L y$ but $x\not\sim_L y$, or if $x\le_R y$ but
             $x\not\sim_R y$, then we have $\la(x)>\la(y)$.
         \item We have $z\sim_{LR} z\inverse$ for all $z\in W$.
         \item Let $J\se S$ and consider the corresponding parabolic subgroup $W_J$ of $W$.
             Let $\la_J$ denote the $\la$-function associated to the Coxeter system $(W_J,J)$. If 
             $x\in W_J$, then
             the $\la$-value $\la_J(x)$ computed in the system $(W_J, J)$ equals the value
             $\la$-value $\la(x)$ computed in the system $(W,S)$.
         \item Let $L_1,L_2$ be two left cells in the same two-sided
             cell. Then the intersection $L_1\cap (L_2\inverse)$ is
             nonempty.
         \item If $x\sim_{LR}y$, then $W$ contains an element $z$ such that $z\sim_L x$ and $z\sim_R y$. 
     \end{enumerate}
        \end{prop}
        \begin{proof}
Parts (1)--(7) are due to Lusztig and can be found in \cite{LusztigHecke}: Parts (1)--(2) follow
immediately from Theorem 6.6 and Proposition 2.6; Parts (3)--(4) are proved in Section 8; the
statements in Parts (5)--(7) appear in Sections 14 as conjectures P9--P12 and P14.  Part (8) is
Lemma 2.4 of \cite{BO}. (The lemma relies on the construction of the so-called $J$-ring of $(W,S)$
from \cite[\S 18]{LusztigHecke}, which in turn depends on Lusztig's boundedness conjecture.)  To see
Part (9), note that if $x\sim_{LR} y$ then $x\sim_{LR} y\inverse$ by Part (6); therefore the left
cell of $x$ intersects the inverse of the left cell of $y\inverse$ nontrivially by Part (8).  This
inverse is precisely the right cell of $y$ by Part (4), so any element $z$ in the intersection
satisfies $z\sim_L x$ and $z\sim_R y$.  
    \end{proof}

\begin{definition}
    For each Coxeter system $(W,S)$ and each integer $n\ge 0$, we let $W_n=\{w\in W:\la(w)=n\}$, and we say that
    $(W,S)$ is \emph{$\la(n)$-finite} if $W_n$ is finite. 
    \label{def:an-finite}
\end{definition}
 
By Proposition \ref{prop:Facts}.(5), for each integer $n\ge 0$ the set $W_n$ defined above is a union of
two-sided cells, of left cells, and of right cells.  We are interested in the cells in $W_2$ within
$\la(2)$-finite Coxeter systems.  Our motivation partly comes from the following facts about $W_0$
and $W_1$. 

\begin{prop}
    \label{prop:facts01}
    Let $(W,S)$ be a Coxeter system with Coxeter diagram $G$. \tul Recall from Remark \ref{rmk:irreducible} that
    we assume $(W,S)$ is irreducible.\tur
    \begin{enumerate}
        \item We have $W_0=\{1\}$. In particular, $W_0$ is finite and simultaneously
            the unique left, right and two-sided cell of $\la$-value 0. 
        \item We have  $W_1=\{w\in W:w\neq 1 \text{ and $w$ has a unique reduced word}\}$. The
            set $W_1$ is finite \textup{(}i.e. the
                system $(W,S)$
            is $\la(1)$-finite\textup{)} if and only if $G$ is acyclic, contains no edge with infinite
            weight, and contains at most one heavy edge with finite weight.  
        \item The set $W_1$ is itself a two-sided cell and hence the unique two-sided cell of $\la$-value 1 in
            $W_1$. The left cells in $W_1$ are in bijection with $S$ and given by the sets \[ L_{s}:=\{z\in
            W_1:s\le^L z\} \] while the right cells in $W_1$ are in bijection with $S$ and given by the sets \[
            R_{s}:=\{z\in W_1:s\le^R z\} \] where $s\in S$. 
        \item If $G$ contains a cycle, then $W$ is $\la(2)$-finite if and only if $G$ is complete. If $G$ is
            acyclic, then $W$ is $\la(2)$-finite if and only if $G$ is one of the graphs in Figure \ref{fig:finite
            E}, where $n$ equals the number of vertices in the Coxeter diagram and there are $q$ and $r$ vertices
            strictly to the left and the right of the trivalent vertex in $E_{q,r}$. 
            \begin{figure}
                \begin{tikzpicture}

                    \node(1) {$A_n$};
                    \node[main node] (11) [right=0.5cm of 1] {};
                    \node[main node] (12) [right=1cm of 11] {};
                    \node[main node] (13) [right=1cm of 12] {};
                    \node[main node] (15) [right=1.5cm of 13] {};
                    \node[main node] (16) [right=1cm of 15] {};
                    \node (17) [right=0.5cm of 16] {$(n\ge 1)$}; 

                    \node (a1) [below =0.1cm of 11] {\tiny{$1$}};
                    \node (a2) [below =0.1cm of 12] {\tiny{$2$}};
                    \node (a3) [below =0.1cm of 13] {\tiny{$3$}};
                    \node (a4) [below =0.05cm of 15] {\tiny{$(n-1)$}};
                    \node (a5) [below =0.1cm of 16] {\tiny{$n$}};

                    \path[draw]
                    (11)--(12)--(13)
                    (15)--(16);

                    \path[draw,dashed]
                    (13)--(15);

                    \node(2) [below=1cm of 1] {$B_n$};

                    \node[main node] (21) [right=0.5cm of 2] {};
                    \node[main node] (22) [right=1cm of 21] {};
                    \node[main node] (23) [right=1cm of 22] {};
                    \node[main node] (25) [right=1.5cm of 23] {};
                    \node[main node] (26) [right=1cm of 25] {};
                    \node (27) [right=0.5cm of 26] {$(n\ge 2)$}; 

                    \node (b1) [below =0.1cm of 21] {\tiny{$1$}};
                    \node (b2) [below =0.1cm of 22] {\tiny{$2$}};
                    \node (b3) [below =0.1cm of 23] {\tiny{$3$}};
                    \node (b4) [below =0.05cm of 25] {\tiny{$(n-1)$}};
                    \node (b5) [below =0.1cm of 26] {\tiny{$n$}};

                    \path[draw]
                    (21) edge node [above] {\tiny{$4$}} (22)
                    (22)--(23)
                    (25)--(26);

                    \path[draw,dashed]
                    (23)--(25);

                    \node (c) [below=1cm of 2] {$\tilde C_{n-1}$};
                    \node[main node] (c1) [right=0.3cm of c] {};
                    \node[main node] (c2) [right=1cm of c1] {};
                    \node[main node] (c3) [right=1cm of c2] {};
                    \node[main node] (c5) [right=1.5cm of c3] {};
                    \node[main node] (c6) [right=1cm of c5] {};
                    \node (c7) [right=0.5cm of c6] {$(n\ge 5)$}; 

                    \node (b11) [below =0.1cm of c1] {\tiny{$1$}};
                    \node (b22) [below =0.1cm of c2] {\tiny{$2$}};
                    \node (b33) [below =0.1cm of c3] {\tiny{$3$}};
                    \node (b44) [below =0.05cm of c5] {\tiny{$(n-1)$}};
                    \node (b55) [below =0.1cm of c6] {\tiny{$n$}};

                    \path[draw]
                    (c1) edge node [above] {\tiny{$4$}} (c2)
                    (c2)--(c3)
                    (c5) edge node [above] {\tiny{$4$}} (c6);

                    \path[draw,dashed]
                    (c3)--(c5);

                    \node (3) [below=1.5cm of c] {$E_{q,r}$};

                    \node[main node] (31) [right=0.4cm of 3] {};
                    \node[main node] (32) [right=1cm of 31] {};
                    \node[main node] (34) [right=1.5cm of 32] {};
                    \node[main node] (38) [right=1.5cm of 34] {};
                    \node[main node] (39) [right=1cm of 38] {};
                    \node (40) [right=0.5cm of 39] {$(r\ge q\ge 1)$}; 
                    \node[main node] (a) [above=0.8cm of 34] {};

                    \path[draw]
                    (34)--(a)
                    (31)--(32)
                    (38)--(39);
                    \path[draw,dashed]
                    (32)--(34)--(38);

                    \node (c11) [below =0.1cm of 34] {\tiny{$0$}};
                    \node (c22) [below =0.1cm of 31] {\tiny{$-q$}};
                    \node (c33) [below =0.05cm of 32] {\tiny{$-(q-1)$}};
                    \node (c44) [below =0.05cm of 38] {\tiny{$(r-1)$}};
                    \node (c55) [below =0.1cm of 39] {\tiny{$r$}};
                    \node (c66) [right =0.02cm of a] {\tiny{$v$}};

                    \node(4) [below=1cm of 3] {${F}_n$};

                    \node[main node] (41) [right=0.5cm of 4] {};
                    \node[main node] (42) [right=1cm of 41] {};
                    \node[main node] (43) [right=1cm of 42] {};
                    \node[main node] (45) [right=1.5cm of 43] {};
                    \node[main node] (46) [right=1cm of 45] {};
                    \node (47) [right=0.5cm of 46] {$(n\ge 4)$}; 

                    \node (d11) [below =0.1cm of 41] {\tiny{$1$}};
                    \node (d22) [below =0.1cm of 42] {\tiny{$2$}};
                    \node (d33) [below =0.1cm of 43] {\tiny{$3$}};
                    \node (d44) [below =0.05cm of 45] {\tiny{$(n-1)$}};
                    \node (d55) [below =0.1cm of 46] {\tiny{$n$}};

                    \path[draw]
                    (41)--(42)
                    (42) edge node [above] {\tiny{$4$}} (43)
                    (45)--(46);
                    \path[draw,dashed]
                    (43)--(45);

                    \node(h) [below=1cm of 4] {$H_n$};

                    \node[main node] (h1) [right=0.45cm of h] {};
                    \node[main node] (h2) [right=1cm of h1] {};
                    \node[main node] (h3) [right=1cm of h2] {};
                    \node[main node] (h5) [right=1.5cm of h3] {};
                    \node[main node] (h6) [right=1cm of h5] {};
                    \node (h7) [right=0.5cm of h6] {$(n\ge 3)$}; 

                    \node (bb1) [below =0.1cm of h1] {\tiny{$1$}};
                    \node (bb2) [below =0.1cm of h2] {\tiny{$2$}};
                    \node (bb3) [below =0.1cm of h3] {\tiny{$3$}};
                    \node (bb4) [below =0.05cm of h5] {\tiny{$(n-1)$}};
                    \node (bb5) [below =0.1cm of h6] {\tiny{$n$}};

                    \path[draw]
                    (h1) edge node [above] {\tiny{$5$}} (h2)
                    (h2)--(h3)
                    (h5)--(h6);
                    \path[draw,dashed]
                    (h3)--(h5);

                    \node(i) [below=1cm of h] {$I_2(m)$};

                    \node[main node] (i1) [right=0.25cm of i] {};
                    \node[main node] (i2) [right=1cm of i1] {};
                    \node (i3) [right=0.5cm of i2] {$(5\le m\le \infty)$};

                    \node (bb1) [below =0.1cm of i1] {\tiny{$1$}};
                    \node (bb2) [below =0.1cm of i2] {\tiny{$2$}};

                    \path[draw]
                    (i1) edge node [above] {\tiny{$m$}} (i2);
                \end{tikzpicture}
                \caption{Irreducible $\la(2)$-finite Coxeter groups with acyclic diagrams.}
                \label{fig:finite E}
            \end{figure}
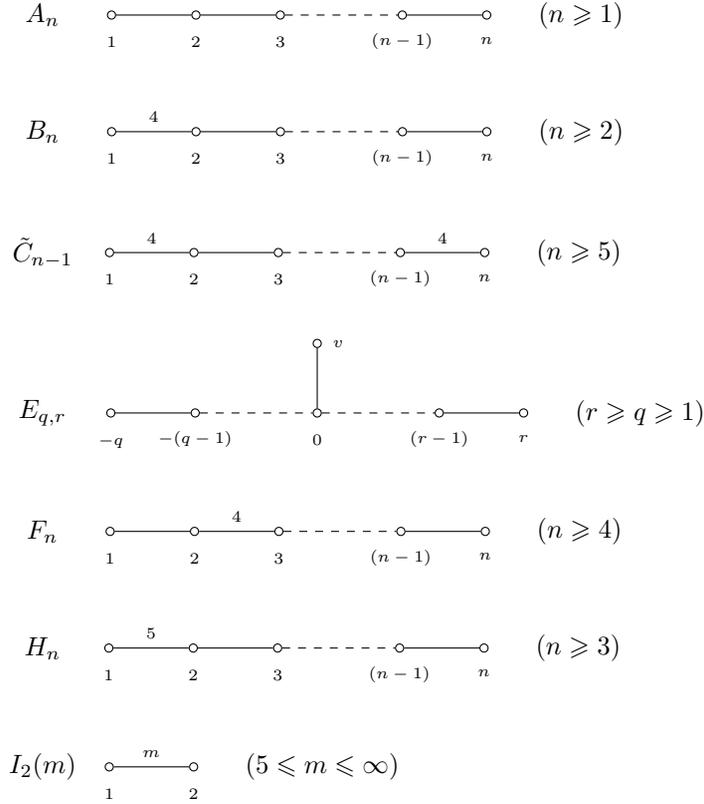   
    \end{enumerate}
\end{prop}
\begin{proof} Part (1) is proved in \cite[\S 13]{LusztigHecke}.  Let 
\[
C= \{w\in W:w\neq 1 \text{ and
      $w$ has a unique reduced word}\}.\] Then the first assertion in Part (2), the assertion that
      $W_1=C$, is proved in \cite[Corollary 3.1]{XuSubregular}. Given the equality $W_1=C$, the
      second assertion in Part (2) follows immediately from \cite[Proposition
      3.8(h)]{LusztigSubregular}.  Part (3) follows from  \cite[Proposition
      3.8(c)]{LusztigSubregular} and Proposition \ref{prop:Facts}(4).  Part (4) is the main result of
      \cite{GreenXu}. 
\end{proof}

\begin{remark}[Labeling conventions]
    \label{rmk:labels}
    For each $\la(2)$-finite Coxeter system $(W,S)$, we will adhere to the labeling of the set $S$
    from Figure \ref{fig:finite E} for the rest of the paper. For example, the generators in $B_n$ will
    always be labeled $1,2,\dots,n$, with $\{1,2\}$ forming the unique heavy edge, and generators in
    $E_{q,r}$ will always be labeled by the numbers $-q,-q+1,\dots, 0,\dots, r$ and the letter $v$
    under the assumption that $r\ge q$, which causes no loss of generality. Note that in the cases
    where $q=1$, $(q,r)=(2,2), (q,r)=(2,3), (q,r)=(2,4), (q,r)=(3,3)$ and $(q,r)=(2,5)$, the Coxeter
    group of type $E_{q,r}$ coincides with the Weyl or affine Weyl group of type $D_{r+3}, E_6, E_7,
    E_8, \tilde E_7$ and $\tilde E_8$, respectively; the group $F_5$ coincides with the affine Weyl
    group of type $\tilde F_4$.
\end{remark}

As we study cells within the set $W_2$ for $\la(2)$-finite Coxeter groups, Part (4) of
Proposition \ref{prop:Facts} allows us to consider only right cells because each left cell $L$ is the setwise
inverse of a right cell $R:=L^{-1}$.  Furthermore, Part (6) of the same proposition implies that the
set $R$ is not only a right cell but also a right cell in the same two-sided cell as $L$. The
following corollary, which we will use to count cells in $W_2$, is now immediate. 

\begin{corollary}
\label{coro:intersection} Let $E$ be a two-sided cell in a Coxeter group $W$. Let $\call\calc$ and
      $\calr\calc$ be the collections of left cells and right cells in $E$, respectively. Then
      $\call\calc=\{R\inverse: R\in \calr\calc\}$. In particular, for each $R\in \calr\calc$ we have
      \[ R=\bigsqcup_{L\in\call\calc}R\cap L=\bigsqcup_{R'\in \calr\calc}R\cap (R')\inverse \] and
      hence \[ \abs{R}=\sum_{L\in\call\calc}\Abs{R\cap L}=\sum_{R'\in \calr\calc}\Abs{R\cap
      (R')\inverse}.  \]
\end{corollary}

The above corollary suggests that intersections of left and right cells in a two-sided cell are
important. This motivates the following definition.
\begin{definition} 
    We define a \emph{zero-sided cell} in a Coxeter system $(W,S)$ to be a set of the form $R\cap L$
    where $R$ is a right cell and $L$ is a left cell in the same two-sided cell as $R$.  For
    brevity, we also refer to two-sided, one-sided and zero-sided cells of $W$ as \emph{2-cells},
    \emph{1-cells} and \emph{0-cells}.
    \label{def:0cells}
\end{definition}

\begin{example}
    \label{eg:b4}
    Let $(W,S)$ be the Coxeter system of type $B_4$, which is $\la(2)$-finite by
    Proposition \ref{prop:Facts}.(10). Later we will see that $W_2$ is the unique two-sided cell of $\la$-value
    2 in $W$; moreover, the cell $W_2$ can be partitioned into 6 right cells as well as into 6 left
    cells. Under a certain labeling of the right and left cells respectively as $R(i)$ and $L(j)$
    for $1\le i,j\le 6$, the sizes of the 0-cells in $W_2$ are given by Table \ref{tab:b4example}, where
    the entry in the $R(i)$-row, $L(j)$-column equals the size of the 0-cell $R(i)\cap L(j)$. By
    Corollary \ref{coro:intersection}, we may count each 1-cell $R(i)$ by summing the entries in the
    corresponding row, and the we may count the 2-cell $W_2$ by summing all entries in the table.
    It follows that $\abs{R(i)}=10$ if $1\le i\le 4$, that $\abs{R(i)}=8$ if $5\le i\le 6$, and that
    $\abs{W_2}=56$.
    \begin{table} \centering
        \begin{tabular}{|c|c|c|c|c|c|c|} \hline & $L(1)$ & $L(2)$ & $L(3)$ & $L(4)$ & $L(5)$ &
              $L(6)$\\\hline $R(1)$ & 2& 2&  2 & 2 &  1&  1\\\hline $R(2)$ & 2& 2&  2 & 2 &  1&
              1\\\hline $R(3)$ & 2& 2&  2 & 2 & 1&  1\\\hline $R(4)$ & 2& 2&  2 & 2 &  1&  1\\\hline
              $R(5)$ & 1& 1&  1 & 1 &  2&  2\\\hline $R(6)$ & 1& 1&  1 & 1 &  2&  2\\\hline
        \end{tabular} \vspace{1em} \caption{Sizes of 0-cells of $\la$-value 2 in type $B_4$}
        \label{tab:b4example}
    \end{table}
\end{example}

\begin{remark}[Symmetry]
    \label{rmk:symmetry}
    As is evident in Propositions \ref{prop:Facts} and \ref{prop:facts01}, there is a large amount
    of ``left-right symmetry'' in the properties of Coxeter group elements and Kazhdan--Lusztig
    cells.  The rest of the paper contains many more pairs of ``left-right symmetric'' notions and
    assertions, such as the two statements in Proposition \ref{prop:StarOpAndCells}. For such pairs, we will
    often formulate one notion or assertion carefully and leave the formulation of the other to the
    reader (as we actually already did in Corollary \ref{coro:intersection} and Example \ref{eg:b4}), or prove one
    of the assertions and invoke symmetry as justification for the other.  The reader may assume
    that the obvious symmetry always works as expected. 
\end{remark}

\subsection{Fully commutative elements}
\label{sec:FC}

Recall that a braid relation in a Coxeter system is a relation of the form $st\dots =ts\dots$ where
$s,t\in S, 2\le m(s,t)<\infty$ and both sides contain $m(s,t)$ factors. We call the relation a
\emph{commutation relation} if $m(s,t)=2$ and call each side of the relation a \emph{long braid} if
$m(s,t)\ge 3$. An element $w\in W$ is called \emph{fully commutative}, or \emph{FC}, if every pair
of reduced words of $w$ can be connected by a finite sequence of commutation relations. Proposition
2.1 of \cite{FC} gives a well-known ``word criterion for full commutativity'': we have $w$ if FC if
and only if no reduced word of $w$ contains a long braid as a subword. Here and henceforth, by a
subword of a word $s_1\dots s_q$ we always mean a contiguous subword, i.e. a word of the form
$s_is_{i+1}\dots s_{j-1}s_{j}$ for some $1\le i\le j\le q$.

We denote the set of all FC elements in $W$ by $\fc(W)$. We shall study cells of $\la$-value 2 via
FC elements because of the following fact:

\begin{prop}[{\cite[Proposition 3.9]{GreenXu}}] \label{prop:a2impliesFC} 
    Let $w\in W$. If $\la(w)=2$, then $w$ is FC.  
\end{prop} 
 
We review two important tools for studying FC elements.  The first is a canonical form called the
\emph{Cartier--Foata form}. For an FC element $w$, this refers to the reduced word $w=w_p\cdot \dots
\cdot w_2\cdot w_1$ satisfying the following properties:
\begin{enumerate}
    \item  For all $1\le i\le p$, the element $w_i$ is a product of pairwise commuting generators in
         $S$;
    \item For all $1<i\le p$, every generator $t\in \supp(w_i)$ fails to commute with some generator
         $s\in \supp(w_{i-1})$.
\end{enumerate} 
The factors $w_1,w_2,w_3, \dots$ can be obtained inductively as the product of the right descents of
the elements $x_1=w, x_2=x_1w_1, x_3=x_2w_2, \dots$, respectively, and the factorization $w=w_p\dots w_1$
is unique up to the re-ordering of the generators appearing in $w_i$ for each $i$; see
\cite{GreenStarReducible}.  We call each $w_i$ the \emph{$i$-th layer} of $w$. When the generators
in $S$ are labeled by integers, we shall insist that within each layer we order the generators in
increasing order from left to right. For example, in the group $W=A_5$ the element $w=1532\in
\fc(W)$ has Cartier--Foata form $w=w_2w_1$ where $w_1=25$ and $w_2=13$.  For more on the
Cartier--Foata form, see \cite{GreenStarReducible}.

Our second tool for studying an FC element $w$ is a labeled poset called a {heap}. To define it, we
first define the \emph{heap of a word} $\underline w=s_1\dots s_q\in S^*$ to be the labeled poset
$H(\underline w):=([q], \preceq)$ where the underlying set is $[q]=\{1,2,\dots,q\}$, the partial
order $\preceq$ is the reflexive, transitive closure of the relation $\prec$ defined by \[ i\prec j
\text{\quad if $i< j$ and $m(s_i,s_j)\neq 2$}, \] and the label of the element $i$ is $s_i$ for each
$i\in [q]$.  The heaps of two words related by a commutation relation are isomorphic as labeled
posets in the sense that there exists a poset isomorphism $f:H({\ul w})\ra H({\ul w'})$ such that
$f(i)$ and $i$ have the same label for all $i\in H({\ul w})$; see \cite[\S 2.2]{FC}.  Thus, it makes
sense to define the \emph{heap of an element} $w\in W$, up to isomorphism, to be the heap of any
reduced word of $w$; we denote this heap by $H(w)$. The following ``heap criterion for full
commutativity'' is a well-known analog of the word criterion mentioned earlier. 

\begin{prop}[{\cite[Proposition 3.3]{FC}}]
    \label{prop:FCCriterion}
    Let  $\underline w= s_1s_2\cdots s_q\in S^*$. Then $\underline w$ is the reduced word of a fully commutative
    element in $W$ if and only if the heap $H(\underline w)$ satisfies the following conditions: 
    \begin{enumerate}
        \item there is no covering relation $i\prec j$ such that $s_i=s_j$;
        \item there is no convex chain $i_1\prec i_2\prec \ldots \prec i_m$ in $H(\underline w)$ with
             $s_{i_1}=s_{i_3}=\cdots = s$ and $s_{i_2}=s_{i_4}=\cdots=t$ where $s,t\in S$ and $m = m(s,t)\ge 3$.
    \end{enumerate} 
\end{prop}

\begin{remark}
    Let $w\in \fc(W)$.  By the  definition of $H(w)$, the left and right descents of $w$ are exactly
    the labels of the minimal and maximal elements in $H(w)$, respectively. In particular, the
    support of the first layer in the Cartier--Foata form of $w$ coincides with the set of the
    labels of the maximal elements in $H(w)$, because they both equal $\calr(w)$.
    \label{rmk:heapsAndWords}
\end{remark}

\begin{remark}
    \label{rmk:chainsAndWalks}
    By the definition of heaps, for every chain of coverings $i_1\prec i_2\dots \prec i_k$ in the
    heap of an FC element $w=s_1\dots s_q$, the generators $s_{i_j}$ and $s_{i_{j+1}}$  must be
    distinct and adjacent in the Coxeter diagram. In other words, the sequence of generators
    $s_{i_1},\dots, s_{i_k}$ forms a \emph{walk} from $s_{i_1}$ to $s_{i_k}$ on the Coxeter diagram
    in the usual graph theoretical sense. We also recall, for later use, that such a walk has
    \emph{length} $(k-1)$ and that a \emph{path} from a vertex $v$ to a vertex $u$ on a graph is a
    walk from $v$ to $u$ of minimal length. On an acyclic and connected graph, there is a unique
    path from $u$ to $v$ for any two vertices $u,v$.  
\end{remark}

There is an intuitive way to visualize the heap of a word $\underline w=s_1\dots s_q\in S^*$ in the
lattice $S\times \Z_{\ge 0}$. Here, we think of the elements of $S$ as \emph{columns} and the
elements of $\Z_{\ge 0}$ as \emph{levels}, and we say two columns $s,t$ \emph{commute} if they
commute as generators, i.e. if $m(s,t)= 2$. To embed $H(\underline w)$ into the lattice, we read
the letters in $\underline w$ from left to right, and for the $j$-th letter we drop a vertex $p_j$
in the column $s_j$ to the lowest level possible subject to the condition that $p_j$ should fall
above every vertex $p_i$ that has been dropped into a column that does not commute with $s_j$, i.e.
above every vertex $p_i$ for which $i\prec j$. In addition, if a column $s$ does not commute with
$s_j$ and contains at least one vertex $p_i$ with $i<j$, i.e. if $i\prec j$, then we draw an edge
to connect $p_j$ to the highest such vertex, i.e. to the vertex $p_i$ in column $s$ with $i$
maximal. It follows that the vertices $p_1,\dots, p_q$ and the edges of the form $\{p_i,p_j\}$ where
$i\prec j$ form exactly the Hasse diagram of the heap $H(\underline w)$. 

When drawing the embedding of $H(\underline w)$, it is customary to not draw the columns or levels
and to label each point $p_i$ simply by $s_i$. This ensures that the isomorphic heaps arising from
the reduced words of $w$ are given by the same graph, the Hasse diagram of the heap $H(w)$. For
example, in Figure \ref{fig:HeapExample} the picture on the right shows the heap $H(w)$ of the
element $w = abcadb$ in the Coxeter group whose Coxeter diagram is drawn on the left. 
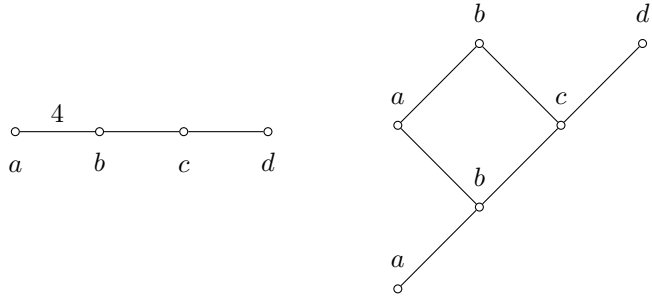
\begin{figure} \centering
    \begin{tikzpicture} \node[main node] (a) {}; \node[main node] (b) [right=1cm of a] {};
          \node[main node] (c) [right=1cm of b] {}; \node[main node] (d) [right=1cm of c] {};

        \node (aa) [below=0.2cm of a] {$a$};
        \node (bb) [below=0.1cm of b] {$b$};
        \node (cc) [below=0.2cm of c] {$c$};
        \node (dd) [below=0.1cm of d] {$d$};

        \path[draw]
        (b)--(c)--(d)
        (a) edge node [above] {$4$} (b);

        \node[main node] (1) [below right=2cm and 5cm of a] {};
        \node[main node] (3) [above right=1cm and 1cm of 1] {};
        \node[main node] (4) [above left=1cm and 1cm of 3] {};
        \node[main node] (5) [above right=1cm and 1cm of 3] {};
        \node[main node] (6) [above left=1cm and 1cm of 5] {};
        \node[main node] (7) [above right=1cm and 1cm of 5] {};

        \node (11) [above=0.1cm of 1] {$a$};
        \node (33) [above=0.1cm of 3] {$b$};
        \node (44) [above=0.1cm of 4] {$a$};
        \node (55) [above=0.1cm of 5] {$c$};
        \node (66) [above=0.1cm of 6] {$b$};
        \node (77) [above=0.1cm of 7] {$d$};

        \path[draw]
        (1)--(3)--(4)--(6)--(5)--(7)
        (3)--(5);
    \end{tikzpicture}
    \caption{Lattice embedding of a heap.}
    \label{fig:HeapExample}
\end{figure}

Heaps can help compute $\la$-values. Let $n(w)$ be the  \emph{width} of the poset $H(w)$, i.e. let 
\[
    n(w)= \max\{\abs{A}:A \text{\; is an antichain in $H(w)$}\}.
\]
Then $n(w)$ bounds and may equal $\la(w)$ (see also Remark \ref{rmk:a.vs.n}):

\begin{prop}[{\cite{Shi}}]\label{prop:a=n}
    Let $w\in \fc(W)$. Then
    \begin{enumerate}
        \item we have $\la(w)\ge n(w)$;
        \item we have $\la(w)=n(w)$ if $W$ is a Weyl or affine Weyl group.
    \end{enumerate}
\end{prop}

\begin{proof}
    The equality in (2) is the main result, Theorem 3.1, of \cite{Shi}. The inequality in (1) essentially follows
    the results in \S 1.3 of the same paper, but let us give a short proof in our notation: the definition of
    heaps implies that $w$ admits a reduced factorization of the form $w=x\cdot y\cdot z$ where $y$ is a product of
    $n(w)$ pairwise commuting generators, whence Parts (1), (2) and (5) of Proposition \ref{prop:Facts} imply that
    $\la(w)\ge\la(y)=n(w)$.  
\end{proof}

\begin{corollary}
    \label{coro:layerSizeBound}
    Suppose that $\la(w)=2$ \textup{(}so $w$ is FC by Proposition \ref{prop:a2impliesFC}\textup{)}. 
    \begin{enumerate}
        \item An antichain in $H(w)$ has at most two elements,  and every antichain in $H(w)$ with two elements is
             maximal.
        \item Every layer in the Cartier--Foata form of $w$ has at most two generators.
        \item We have $\abs{\calr(w)}\le 2$.
    \end{enumerate}
\end{corollary}
\begin{proof} 
    We must have $n(w)\le 2$ by Proposition \ref{prop:a=n}.(1). This implies (1). By the definition of heaps, each layer in
    the Cartier--Foata form of $w$ corresponds to an antichain that has as many elements as the support of the
    layer, so (1) implies (2). Finally, the set $\calr(w)$ equals the
    support of the first layer of $w$ by Remark \ref{rmk:heapsAndWords}; therefore (3) follows from (2).
\end{proof}

Antichains of heaps appear frequently in this paper. We highlight one property of maximal antichains now. Recall
that an \emph{ideal} in a poset $P$ is a set $\cali$ such that if $y\in \cali$ and $x<y$ for $x,y\in P$ then $x\in
\cali$ ; dually, a \emph{filter} in $P$ is a set $\calf$ such that if $y\in \calf$ and $x>y$ for $x,y\in P$ then
$x\in \calf$. Any set $A\se P$ \emph{generates} an ideal 
\[ \cali_A:=\{i\in H:i\le j\text{ for some $j\in A$}\} \]
and a filter 
\[ \calf_A:=\{i\in H:i\ge j\text{ for some $j\in A$}\}.\] 
If $A$ is an antichain, then any element in the set $P\setminus (\cali_A\cup\calf_A)$, if it exists, can be
adjoined to $A$ to form a larger antichain. This implies the following:

\begin{lemma}
    \label{lemm:maxAntichain} If $A$ is a maximal antichain in $P$, then $P=\cali_A\cup \calf_A$.
\end{lemma}

\subsection{Star operations} 
\label{sec:StarOperations}

In this subsection we discuss star operations on $W$. We emphasize the so-called right star operations and omit
details on the analogous left star operations in the spirit of Remark \ref{rmk:symmetry}.  

A right star operation is a function from $W$ to $W$ partially defined with respect to a pair of noncommuting
generators $I=\{s,t\}\se S$. Let $m=m(s,t)$ and consider the parabolic subgroup $W_I=\ip{s,t}$. Then each element
$w$ in $W$ admits a unique reduced factorization $w=w^I \cdot w_I$ where $\calr{(w^I)}\cap I = \emptyset$ and
$w_I\in W_I$.  This factorization is the \emph{left coset decomposition} of $w$ with respect to $I$; see \cite[\S
2.4]{BB}. Depending on the value of $w_I$, one of the following mutually exclusive conditions must hold; the
sequences in (2) and (3) are both infinite if $m=\infty$.
\begin{enumerate} 
    \item $w_I=1$, or $w_I$ is the longest element $sts\dots$ of length $m(s,t)$ in $W_I$; 
    \item $w$ is one of the $(m-1)$ elements \[ x_1:= w^I\cdot s, \quad x_2:= w^I\cdot st,\quad   x_3:= w^I\cdot
         sts,\quad \dots \] where $\call(w_I)=\{s\}$ and $l(w_I)<m$; 
    \item $w$ is one of the $(m-1)$ elements \[ y_1:= w^I\cdot t,\quad y_2:= w^I\cdot ts,\quad  y_3:= w^I\cdot tst,
         \quad\dots \] where $\call(w_I)=\{t\}$ and $l(w_I)<m$.
\end{enumerate} 

\noindent
By definition, we may apply a \emph{right upper star operation with respect to $I$} on $w$ if and only if $w$ is an
element of the form $x_i$ or $y_i$ for some $1\le i\le m-2$; in these cases, the result of the operation is denoted
by $w^*$ and defined to be $x_{i+1}$ or $y_{i+1}$, respectively. Similarly, a \emph{right lower star operation with
respect to $I$} is defined on $w$ precisely when $w$ is an element of the form $x_i$ or $y_i$ for some $2\le i\le
m-1$, whence the result of the operation is the element $w_*:=x_{i-1}$ or $w_*:=y_{i-1}$, respectively.  Left star
operations are defined similarly, and we denote the result of applying a left upper or lower star operation on $w$
by ${}^* w$ or ${}_* w$, respectively. 

When $m=m(s,t)=3$, the (generalized) star operations defined above recover the original star operations introduced
by Kazhdan and Lusztig in \cite{KL}. In this case at most one of $w^*$ and $w_*$ is defined depending on the value
of $w_I$; therefore we may unambiguously speak of \emph{the right star operation with respect to $I$}, without
indicating if the operation is upper or lower. More precisely, both $w^*$ and $w_*$ are undefined if $w_I\in
\{1,sts=tst\}$, $w^*$ is defined while $w_*$ is not if $w_I\in\{s,t\}$, and $w_*$ is defined while $w^*$ is not if
$w_I\in \{st,ts\}$. In the last two cases, we denote whichever one of $w^*$ and $w_*$ is
defined by $w*$,
placing the star sign in the middle. We call the operation $w\mapsto w*$ a \emph{simple star operation}; it is an
involution in the sense that $(w*)*=w$ whenever $w*$ is defined. Similarly, the pair $I=\{s,t\}$ gives rise to a
partially defined involution $w\mapsto *w$ on $W$ that we call the \emph{simple left star operation with respect to
$I$}.

\begin{example}
    \label{eg:StarOperation} 
    Let $(W,S)$ be a Coxeter system with $S=\{a,b,c\}$, $m(a,b)=3$, $m(b,c)=4$ and $m(a,c)=2$. Let $I=\{a,b\},
    J=\{b,c\}$, and let $w=abcab$. Then with respect to $I$, the coset decompositions of $w$ are given by \[ w=
    w_I\cdot {}^I w= aba\cdot cb, \quad w= w^I \cdot w_I = abc\cdot ab, \] so $w_*=abc\cdot a$ while $w^*,
    {}_* w, {}^* w$ are not defined.  Moreover, since $m(a,b)=3$, we have $w*=w_*=abca$ while $*w$ is undefined.
    With respect to $J$, we have \[ w=w_J\cdot {}^J w=b\cdot abcb, \quad w=w^J\cdot w_J=ba\cdot bcb,
    \] so
    ${}^{*}w=cb\cdot abcb, w_*=ba\cdot bc,$ while ${}_*w$ and $w^*$ are not defined.  
\end{example}

Star operations are intimately related to \kl cells. First, left and right cells are closed under left and right
star operations, respectively: 
\begin{prop}[{\cite[Proposition 3.3]{GreenXu}}]
    \label{prop:StarOpAndCells} 
    Let $(W,S)$ be a Coxeter system. Then the following holds with respect to every pair of noncommuting generators
    $\{s,t\}$ in  $S$:
    \begin{enumerate}
        \item  If $w\in W$ and ${}_* w$ is defined, then $w\sim_L {}_* w$;
        \item  If $w\in W$ and $w_*$ is defined, then $w\sim_R w_*$.
    \end{enumerate}
\end{prop}

\noindent Second, simple star operations preserve cell equivalence in the following way: 
\begin{prop}[{\cite[Corollary 4.3]{KL}}]
    \label{prop:SimpleStarOpAndCells} 
    Let $(W,S)$ be a Coxeter system, let $y,w\in W$, and suppose $m(s,t)=3$ for some generators $s,t\in S$.  Then
    the following hold with respect to the pair $\{s,t\}$:
    \begin{enumerate}
        \item if  $y\sim_L w$ and $y*,w*$ are defined, then $y*\sim_L w*$;
        \item if $y\sim_R w$ and $*y,*w$ are defined, then $*y\sim_R *w$. 
    \end{enumerate}
\end{prop}

\noindent 
Later, we will often combine Propositions \ref{prop:StarOpAndCells} and \ref{prop:SimpleStarOpAndCells} with
Proposition \ref{prop:Facts} to show two elements are in the same cell or in distinct cells. Note that
Proposition \ref{prop:StarOpAndCells} implies that star operations preserve $\la$-values:

\begin{corollary}
    Let $x,y\in W$. If $y$ is obtained from $x$ via any star operation, then $\la(y)=\la(x)$.
    \label{coro:starA}
\end{corollary}
\begin{proof}
    This is immediate from Proposition \ref{prop:StarOpAndCells} and Proposition \ref{prop:Facts}.(5).
\end{proof}

\noindent
Our discussion of star operations so far applies to arbitrary elements of $W$, rather than only FC elements. On the
other hand, if an element $w\in W$ is FC, then the heap of $w$ provides a simple characterization of what star
operations can be applied to $w$. We explain why this is true for right lower star operations below. Let
$I=\{s,t\}$ be a pair of noncommuting generators in $S$ as before. By definition, the element $w_*$ with respect to
$I$ is defined and given by $w_*=ws$ if and only if in the coset decomposition $w=w^I\cdot w_I$ we have $2 \le
l(w_I)< m(s,t)$ and the reduced word of $w_I$ is of the form  $w=\dots ts$. The last condition holds if and only if
$s \in\calr(w)$ and $t\in \calr(ws)$, in which case $w$ has a reduced word ending in $s$ and removal of that $s$
results in a reduced word of $w_*$. Remark \ref{rmk:heapsAndWords} now implies the following result:
\begin{prop}
    \label{prop:Reducibility} Let $w\in W$ be FC and suppose $3\le m(s,t)$ for some $s,t\in S$.  Then $w$ admits a
          right lower star operation with respect to $\{s,t\}$ which results in $ws$ if and only if the conditions
          below all hold:
    \begin{enumerate}
        \item the heap $H(w)$ contains a maximal element $i$ labeled by $s$;
        \item the element $i$ covers an element $j\in H(w)$ labeled by $t$;
        \item the element $i$ is the unique maximal element covering $j$ in $H(w)$.
    \end{enumerate}
\end{prop}

\noindent 
Note that Condition (3) is equivalent to the condition that $j$ becomes maximal upon removal of $i$ from $H(w)$.
The proposition allows us to tell what lower star operations applies to an FC element from a glance at its heap.
For example, given the element $w=abcadb$ and the embedding of $H(w)$ considered in Figure \ref{fig:HeapExample},
we note that the two highest vertices in the figure are the only maximal elements in $H(w)$ and that only the
removal of top vertex labeled by $b$ results in a new maximal vertex, namely, the higher of the two vertices
labeled by $a$. It follows that only one right lower star operation is defined on $w$, namely, the operation with
respect to $\{a,b\}$ that removes the right descent $b\in \calr(w)$. 

\section{Kazhdan--Lusztig Cells via Stubs}
\label{KLcells}

Recall the following definition of {stubs} from the introduction:
\begin{definition} 
\label{def:stubs} Let $(W,S)$ be a Coxeter system and let $w\in W$. We call $w$ a \emph{left stub} if no right
      lower star operation can be applied to $w$. Similarly, we call $w$ a \emph{right stub} if $w$ admits no left
      lower star operation. We denote the set of left stubs of $\la$-value 2 in $W$ by $\sw$, and denote the set of
      right stubs of $\la$-value 2 in $W$ by $\cals'(W)$.
\end{definition} 

In this section, we assume that $(W,S)$ is an irreducible $\la(2)$-finite Coxeter system, describe
the stubs in $\sw$, and use stubs to find the left, right and two-sided cells in $W_2$. Our main
results on cells are Theorems \ref{thm:stubDescription},  \ref{thm:cellParameterization},
\ref{thm:CellViaDecomp} and \ref{thm:twoSidedCells}. As we study 2-cells in \autoref{sec:2cells}, we
also introduce two equivalence relations on $\sw$ that will play a crucial role as we study 0-cells
in \autoref{sec:0cells}.  

We note that while some results in the section apply to arbitrary Coxeter systems, stubs seem to
control cells in $W_2$ in the nicest ways only for $\la$(2)-finite systems; see
Remark \ref{rmk:a2finiteAssumption}.  We also note that since $W_2$ certainly contains no cell when it is
empty, we shall often assume that $(W,S)$ is \emph{nontrivially $\la(2)$-finite} in the sense that
$W_2$ is finite but nonempty.  By Proposition \ref{prop:facts01}.(4) and the results of \cite{GreenXu}, the
nontrivially $\la(2)$-finite systems are exactly those of types $A_n (n\ge 3), B_n (n\ge 3), \tilde
C_{n-1} (n\ge 5), E_{q,r} (r\ge q\ge 1), F_n (n\ge 4)$ and $H_n (n\ge 3)$. Our description of stubs
and cells will therefore often be given type by type. 

We will often use the prefix ``$\la(2)$-'' as an adjective to indicate an object has $\la$-value 2.
For example, the set $\sw$ is exactly the set of left $\la(2)$-stubs in $W$, and an $\la(2)$-cell
means a cell of $\la$-value 2.

\subsection{Description of stubs}
\label{sec:StubDescriptions} 
We study {$\la(2)$-stubs} in this subsection. As $\la(2)$-elements are FC, we start with the
following observation:

\begin{prop}
    \label{prop:stubCharacterization}
    Let $(W,S)$ be an arbitrary Coxeter system and let $w$ be an FC element in $W$. Let $w=w_p\dots w_2w_1$ be the
    Cartier--Foata form of $w$. Then $w$ is a left stub if and only if every generator $t\in \supp(w_2)$ fails to
    commute with at least two generators $s,s'\in \supp(w_1)$. 
\end{prop}

\begin{remark}
    Note that the above condition on $t$ is vacuously true when $w=w_1$, i.e. when $w$ is a product of commuting
    generators.
\end{remark}

\begin{proof}[Proof of Proposition \ref{prop:stubCharacterization}]
    Let $t\in \supp(w_2)$ and $s\in \supp(w_1)$. Then $s\neq t$ by
    Proposition \ref{prop:FCCriterion}.(1), so $t$ is covered by $s$ in $H(w)$ if and only if $m(s,t)\ge 3$.  The
    proposition then follows from Remark \ref{rmk:heapsAndWords} and Proposition \ref{prop:Reducibility}.
\end{proof}
\noindent 
The proposition shows that whether an FC element is a left stub is determined ``locally'', by only the first two
layers of its Cartier--Foata form. 

\begin{corollary}
    \label{coro:firstTwoLayers}
    Let $(W,S)$ be an arbitrary Coxeter system.  Let $G$ be the Coxeter diagram of $(W,S)$.  Then an
    $\la(2)$-element $w\in W$ is a left stub if and only if in the Cartier--Foata form $w=w_p\dots w_2w_1$ of $w$,
    we have 
    \begin{enumerate} 
        \item  $w_1=ss'$ for two commuting generators $s, s'$ in $S$;
        \item Every generator $t\in \supp(w_2)$ is adjacent to both $s$ and $s'$ in $G$.  
    \end{enumerate}
\end{corollary}

\begin{proof}
    By Proposition \ref{prop:stubCharacterization}, it suffices to show that $\supp(w_1)$ contains exactly two
    generators.  By Corollary \ref{coro:layerSizeBound}, it further suffices to show that
    $\supp(w_1)$ is not empty or a singleton.  This follows from Proposition \ref{prop:facts01}: if $w_1$ is
    empty, then $w=1$ and $\la(w)=\la(1_W)=0$; if $\supp(w_1)$ contains only one generator $s$, then
    $w_2$ must be empty by Proposition \ref{prop:stubCharacterization}, which forces $w=w_1=s$ and
    $\la(w)=\la(s)=1$.    
\end{proof}

For nontrivially $\la(2)$-finite systems, Corollary \ref{coro:firstTwoLayers} turns out to impose strong
restrictions on the first two layers of $\la(2)$-stubs.  The following observation imposes further
restrictions on the deeper layers. 

\begin{lemma}
    \label{lemm:deeperLayers}
    Let $(W,S)$ be an arbitrary Coxeter system. Let $G$ be the Coxeter diagram
    of $(W,S)$. Let $w\in \fc(W)$, let $w=w_p\dots w_1$ be the Cartier--Foata form of $w$, and let $2<i\le p$. 
    \begin{enumerate}
        \item Every element $t\in \supp(w_i)$ is adjacent in $G$ to some element $s\in \supp(w_{i-1})$.  
        \item If $s,t\in S$ are generators connected by a simple edge in $G$ such that $s$ appears in $w_{i-2}$ and
             $w_{i-1}=t$, then $s$ does not appear in $w_i$.  
    \end{enumerate}
\end{lemma}

\begin{proof}
    Part (1) simply paraphrases Condition (2) from the characterization of the Cartier--Foata form
    in \autoref{sec:FC}. To prove (2), note that if $s$ appears in $w_{i}$ then we may commute the
    $s$ in $w_{i-2}$ to the left and the $s$ in $w_i$ to the right, if necessary, until they are
    both next to the $t$ in $w_{i-1}$. This results in a reduced factorization $w=x \cdot sts\cdot
    y$ of $w$, contradicting the word criterion for full commutativity. 
\end{proof}

The following example shows in detail how Conditions (1)--(2) from Corollary \ref{coro:firstTwoLayers} and
Lemma \ref{lemm:deeperLayers} strongly restrict the Cartier--Foata forms of left $\la(2)$-stubs in a
nontrivially $\la(2)$-finite Coxeter system.  Roughly speaking, the restrictions are strong because
the Coxeter diagram of such a system always contains few pairs of vertices that share a neighbor as
well as few heavy edges.

\begin{example}
    \label{eg:layerAnalysis} 
    We find all left $\la(2)$-stubs in the Coxeter system $(W,S)$ of type $B_n$ in this example.
    Table \ref{tab:B4stubs} contains the heaps of the six stubs in $\cals(B_4)$, which will serve as an example. 

    Let $w\in\sw$ and let $w=w_p\dots w_2w_1$ be the Cartier--Foata form of $w$.  By Condition (1)
    of Corollary \ref{coro:firstTwoLayers}, we must have either $w_1=ij$ for some $1\le i,j\le n$ where
    $j>i+2$ or $w_1=i(i+2)$ for some $1\le i\in n-1$. In the former case---which holds for only the
    stub in the third column in Table \ref{tab:B4stubs} if $n=4$---the second layer of $w$ cannot exist
    by Condition (2) of Corollary \ref{coro:firstTwoLayers} since $i,j$ share no neighbor in $G$.  In the
    latter case---which corresponds to the first two columns of Table \ref{tab:B4stubs}---either $w_2$
    does not exist and $w=w_1=i(i+2)$ (first row), or $w_2$ exists and is given by $w_2=(i+1)$, the
    only common neighbor of $i$ and $(i+2)$ in $G$ (second row). In the latter subcase, if $i>1$
    then we have $m(i+1,i)=m(i+1,i+2)=3$, hence Lemma \ref{lemm:deeperLayers} rules out the possibility
    of a third layer in $w$ and forces $w=w_2w_1=(i+1)\cdot i(i+2)$; if $i=1$, then
    Lemma \ref{lemm:deeperLayers} implies that we have either $w=w_{2}w_1=2\cdot (13)$ or
    $w=w_3w_2w_1=1\cdot 2\cdot (13)$. In particular, if $w$ contains at least three layers then we
    have to have $w_1=13, w_2=2$ and $w_3=1$, whence $w$ cannot have a fourth layer $w_4$: an
    element $t\in \supp(w_2)$ has to be adjacent to the only generator $1$ in $w_3$ and thus has to
    be $t=2$, so if $w$ contains a fourth layer then $w=\dots\cdot 2\cdot 1\cdot 2\cdot
    (13)=\dots\cdot (2121)\cdot 3$, contradicting the fact that $w$ is FC. The stub $1\cdot 2\cdot
    (13)$ is shown in the third row and first column in Table \ref{tab:B4stubs}.

    In summary, we can find all left $\la(2)$-stubs in $B_n$, and among them there is a unique stub
    with more than two layers, namely $1\cdot 2\cdot (13)$. The problem of finding left
    $\la(2)$-stubs in the Coxeter system $A_n$ is similar but easier.
\end{example}

\begin{table}
    \begin{tabular}{|c|c|c|}
        \hline
        \begin{tikzpicture}
            \node[main node] (1) {};
            \node[main node] (2) [right=1cm of 1] {};
            \node (11) [above = 0.1cm of 1] {\tiny{$1$}};
            \node (22) [above = 0.1cm of 2] {\tiny{$3$}};
        \end{tikzpicture}
        &
        \begin{tikzpicture}
            \node[main node] (1) {};
            \node[main node] (2) [right=1cm of 1] {};
            \node (11) [above = 0.1cm of 1] {\tiny{$2$}};
            \node (22) [above = 0.1cm of 2] {\tiny{$4$}};
        \end{tikzpicture} 
        &
        \begin{tikzpicture}
            \node[main node] (1) {};
            \node[main node] (2) [right=1cm of 1] {};
            \node (11) [above = 0.1cm of 1] {\tiny{$1$}};
            \node (22) [above = 0.1cm of 2] {\tiny{$4$}};
        \end{tikzpicture}
        \\
        \hline
        \begin{tikzpicture}
            \node[main node] (1) {};
            \node[main node] (2) [right=1cm of 1] {};
            \node[main node] (3) [below right = 0.5cm and 0.5cm of 1] {};
            \node (11) [above = 0.1cm of 1] {\tiny{$1$}};
            \node (22) [above = 0.1cm of 2] {\tiny{$3$}};
            \node (33) [below = 0.1cm of 3] {\tiny{$2$}};
            \path[draw]
            (1)--(3)--(2); 
        \end{tikzpicture}
        &
        \begin{tikzpicture}
            \node[main node] (1) {};
            \node[main node] (2) [right=1cm of 1] {};
            \node[main node] (3) [below right = 0.5cm and 0.5cm of 1] {};
            \node (11) [above = 0.1cm of 1] {\tiny{$2$}};
            \node (22) [above = 0.1cm of 2] {\tiny{$4$}};
            \node (33) [below = 0.1cm of 3] {\tiny{$3$}};
            \path[draw]
            (1)--(3)--(2);  
        \end{tikzpicture}
        &        
        \multicolumn{1}{c}{}\\
        \cline{1-2}
        \begin{tikzpicture} 
            \node[main node] (1) {};
            \node[main node] (2) [right=1cm of 1] {};
            \node[main node] (3) [below right = 0.5cm and 0.5cm of 1] {};
            \node[main node] (4) [below = 1cm of 1] {};
            \node (11) [above = 0.1cm of 1] {\tiny{$1$}};
            \node (22) [above = 0.1cm of 2] {\tiny{$3$}};
            \node (33) [below = 0.1cm of 3] {\tiny{$2$}};
            \node (11) [below = 0.1cm of 4] {\tiny{$1$}};
            \path[draw]
            (1)--(3)--(2)
            (3)--(4);
        \end{tikzpicture}            
        &
        \multicolumn{2}{c}{}\\
        \cline{1-1}
    \end{tabular}
    \vspace{1em} 
    \caption{The six stubs of $B_4$}
    \label{tab:B4stubs}
\end{table}

We classify the left $\la(2)$-stubs of all nontrivially $\la(2)$-finite Coxeter systems in the
following theorem.  In the theorem, we express every element in its Cartier--Foata form and use
$\cdot$ to separate the different layers, and the symbol $\sqcup$ denotes disjoint union. We will
draw the heaps of typical stubs immediately after the proof of the theorem.  The classification lays
the foundation for the subsequent study of $\la(2)$-cells. 

\begin{thm}
    \label{thm:stubDescription}
    Let $(W,S)$ be an irreducible and nontrivially $\la(2)$-finite Coxeter system. Suppose that $(W,S)$ is of type
    $X$, and let $G$ be the Coxeter diagram of $(W,S)$, labeled in the convention explained in Remark \ref{rmk:labels}.
    Let $\sx$ be the set of all left stubs of $\la$-value 2 in $W$.  Then $\sx$ can be described as follows. 
    \begin{enumerate}
        \item If $X=A_n$ for some $n\ge 3$, then $\sx=\mathcal{S}_1\sqcup\mathcal{S}_2$ where 
            \[
            \mathcal{S}_1=\{x_{ij}:=ij\,\vert\, 1\le i,j\le n,j>i+1\}\]
            and
            \[
                \mathcal{S}_2=\{y_i:=i\cdot
                (i-1)(i+1)\,\vert\, 1< i<n\}.
            \]
        \item If $X=B_n$ for some $n\ge 3$, then $\sx=\mathcal{S}_1\sqcup\mathcal{S}_2\sqcup \cals_3$ where
            $\mathcal{S}_1,\mathcal{S}_2$ are the same \textup{(}see Remark
            \ref{rmk:stubRemarks1}.\textup{(}1\textup{)}\textup{)} as in \textup{(}1\textup{)}
            and 
            \[
                \cals_3=\{z_1:=1\cdot 2\cdot 13\}.
            \] 
        \item If $X=\tilde C_{n-1}$ for some $n\ge 5$, then $\sx =\cals_1\sqcup \cals_2\sqcup
             \cals_3\sqcup\cals_4$ where $\cals_1,\cals_2,\cals_3$ are the same as in
             \textup{(}2\textup{)} and 
            \[
                \cals_4=\{z_n:=n\cdot (n-1)\cdot (n-2)n\}.
            \]
        \item If $X=E_{q,r}$ for some $r\ge q\ge 1$, then $\sx=\cals_1\sqcup\cals_2\sqcup
             \mathcal{T}_2\sqcup\cals_3$ for the sets 
            \[
            \cals_1=\{x_{st}\,\vert\,s,t\in S, m(s,t)=2\},\]
            \[
                \cals_2=\{y_{i}:= i\cdot (i-1)(i+1)\,\vert\, -q< i<r\},
            \]
            \[
                \mathcal{T}_2=\{y'_0:=0\cdot (-1)v\}\sqcup\{y''_0:=0\cdot 1v\},
            \]
            and $\cals_3=\{z_s: s\in S, s\neq 0\}$ where 
            \[
                z_s= 
                \begin{cases}
                    s\cdot (s-1)\cdot \ldots \cdot 0\cdot (-1)v &\text{if $s\neq v$ and $1\le s\le
                    r$},\\
                    s\cdot (s+1)\cdot \ldots \cdot 0\cdot (1)v & \text{if $s\neq v$ and $-q\le s\le -1$},\\
                    v\cdot 0 \cdot (-1)1&\text{if $s=v$}.
                \end{cases}
            \]
        \item If $X=F_n$ for some $n\ge 4$, then $\sx=\cals_1\sqcup\cals_2\sqcup\cals_3$
            where  $\cals_1,\cals_2$ are the same as in \textup{(}1\textup{)} and 
            \[
                \cals_3=\{
                    z_1:=1\cdot 2\cdot 3\cdot 24, z_2:=2\cdot 3\cdot
                24\}\sqcup\{z_i\,\vert\, 3\le i\le n\}, 
            \]
            where $z_i=i\cdot (i-1)\cdot \ldots\cdot 3\cdot 2\cdot 13$ for each $3\le i\le n$.
        \item If $X=H_n$ for some $n\ge 3$, then $\sx=\cals_1\sqcup\cals_2\sqcup\cals_3\sqcup\cals_4$ where
            $\cals_1,\cals_2,\cals_3$ are the same as in \textup{(}2\textup{)} and 
            \[
                \cals_4=\{z_i:=i\cdot (i-1)\cdot \ldots\cdot 1\cdot 2\cdot 13\,\vert\,
                2\le i\le n\}.
            \]
        \item In each of the above, we listed each stub of $\la$-value 2 exactly once. {\textup{(}}For example, two
            $x_{ij},x_{i'j'}$ from $ \cals_1$ are equal only if $i=i',j=j'$.{\textup{)}} The cardinality of $\sx$ 
            in $W$ is given by 
            \[ 
                \abs{\sx}=
                \begin{cases}\binom{n}{2} -1 & \text{if $X=A_n$ \textup{(}$n\ge 3$\textup{)}};\\ 
                    \binom{n}{2} & \text{if
                      $X=B_n$ \textup{(}$n\ge 3$\textup{)}};\\ 
                      \binom{n}{2} +1 & \text{if $X=\tilde C_{n-1}$
                      \textup{(}$n\ge 5$\textup{)}};\\ 
                      \binom{n+1}{2}-1 & \text{if $X= E_{q,r}$ \textup{(}$q,r\ge 1, n=q+r+2=\abs{S}$\textup{)}};\\ 
                      \binom{n+1}{2}-1 & \text{if $X= F_n$ \textup{(}$n\ge 4$\textup{)}};\\
                      \binom{n+1}{2}-1 & \text{if $X= H_n$ \textup{(}$n\ge 3$\textup{)}}.
                \end{cases} 
            \]
    \end{enumerate}
\end{thm}

\begin{remark}
    \label{rmk:stubRemarks1}
    \begin{enumerate}
        \item In Theorem \ref{thm:stubDescription}, by saying that a certain set $\cals_i$ in some part  is the same as
             the set $\cals_i$ given from an earlier part, we mean that the former set consists of the elements
             given by the same reduced words as those of the elements in the latter set. 
        \item By symmetry, an element $w\in W$ is right stub if and only if $w\inverse$ is a left stub, so
             $\cals'(W)=\{w\inverse: w\in \cals(W)\}$ for each Coxeter system $(W,S)$ appearing in Theorem
             \ref{thm:stubDescription}.
    \end{enumerate} 
\end{remark}

\begin{proof}[Proof of Theorem \ref{thm:stubDescription}]
    Part (7) follows from Parts (1)--(6) by inspection (to see that every element in $\sw$ is indeed
    listed only once) and by easy counting arguments, so it suffices to show that in each of Cases
    (1)--(6), an element $w\in W$ is a left \twostub if and only if $w\in \sx$. The ``only if''
    implication follows from analysis of the layers of left $\la(2)$-stubs:
    Corollary \ref{coro:firstTwoLayers} implies that the first two layers of such a stub $w$ must satisify
    Conditions (1)--(2) given in the corollary, whence Lemma \ref{lemm:deeperLayers} forces $w$ to be an
    element in $\sw$ by arguments similar to those we used for type $B_n$ in
    Example \ref{eg:layerAnalysis}. The ``if'' implication, i.e. the fact that every element $w\in \sw$ is
    a left $\la(2)$-stub, can be proven as follows: first, draw the heap of the specified reduced
    word of $w$ we gave (such as $x_{13}=13\in A_n$) and use Proposition \ref{prop:FCCriterion} to see that $w$
    is indeed FC; second, use Proposition \ref{prop:Reducibility} to see that $w$ is a stub;
    finally, note that we can transform $w$ to its first layer $w_1$ by a sequence of left lower
    star operations by the analog of Proposition \ref{prop:Reducibility} for left star operations,
    and then invoke Corollary \ref{coro:starA} to conclude that $\la(w)=\la(w_1)=2$.  The proof is complete.
\end{proof}

Let us describe and draw the heaps of stubs described in Theorem \ref{thm:stubDescription}. Call a stub
$w\in \sw$ \emph{short} if $l(w)=2$, \emph{medium} if $l(w)=3$, and \emph{long} if $l(w)>3$. Then
the short, medium, and long stubs in $\sw$ are precisely those labeled by ``$x$'', ``$y$''
(including $y'_0$ and $y''_0$ in type $E_{q,r}$) and ``$z$'' in Theorem \ref{thm:stubDescription}; they are
also precisely those with one, two, and more than two layers, respectively.  The heaps of the short
and medium stubs take the forms shown in Table \ref{tab:xyStubs}.  Long stubs exist in all
$\la(2)$-finite Coxeter systems except $A_n$.  In type $B,\tilde C, F$ and $H$, every long stub
involves a ``turn'' at the second layer, that is, there is always a heavy edge $\{s,t\}$ in the
Coxeter diagram such that $s$ appears in the first and third layers of the stub while $t$ appears in
the second layer. The stub $1213\in B_4$ pictured in Table \ref{tab:B4stubs} is such an example.
Table \ref{tab:zBCFH} contains more typical such stubs.  Finally, in type $E_{q,r}$, each long stub $w$
with Cartier--Foata form $w=w_p\dots w_2w_1$ satisfies the conditions $\supp(w_2)=\{0\},
\supp(w_1)=\{s,t\}, \supp(w_2)=\{u\}$ where $\{s,t,u\}=\{-1,1,v\}$; the heaps of the long stubs are
shown in Table \ref{tab:zE}.

    \begin{table}
        \begin{tabularx}{0.8\textwidth}{|X|X|}
            \hline
            \raisebox{-1.2em}{
    \begin{tikzpicture}
        \node (0) {short stubs:};
        \node [main node] (1) [right = 0.3cm of 0] {};
        \node [main node] (2) [right = 1cm of 1] {};
\end{tikzpicture}
}
    &
    \raisebox{-2em}{
    \begin{tikzpicture}
        \node (00) {medium stubs:};
        \node [main node] (4) [above right = 0.05cm and 0.3cm of 00] {};
        \node [main node] (5) [below right = 0.5cm and 0.5cm of 4] {};
        \node [main node] (6) [above right = 0.5cm and 0.5cm of 5] {};
        \node (x) [below =1.5em of 6] {};
        \path[draw]
        (4)--(5)--(6);
    \end{tikzpicture}
}
    \\
    \hline
\end{tabularx}
\vspace{1em} 
\caption{The short and medium stubs of $\sw$}
    \label{tab:xyStubs}
    \end{table}

    \begin{table}
        \centering
        \begin{tabular}{|c|c|}
            \hline
            \raisebox{0.8em}{
                \begin{tikzpicture}
                    \node (0) {$z_{n}\in \tilde C_{n-1}$:};
                    \node [main node] (1) [right = 0.8cm of 0] {};
                    \node [main node] (2) [above left = 0.5cm and 0.5cm of 1] {};
                    \node [main node] (3) [above right = 0.5cm and 0.5cm of 1] {};
                    \node [main node] (4) [below right = 0.5cm and 0.5cm of 1] {};

                    \node (11) [right = 0.01cm of 1] {\tiny{$(n-1)$}};
                    \node (22) [above = 0.01cm of 2] {\tiny{$(n-2)$}};
                    \node (33) [above = 0.01cm of 3] {\tiny{$n$}};
                    \node (44) [below = 0.01cm of 4] {\tiny{$n$}};
                    \path[draw]
                    (2)--(1)--(3)
                    (1)--(4);
            \end{tikzpicture}}
            & 
            \raisebox{-0.1em}{
                \begin{tikzpicture}
                    \node (x0) {$z_{1}\in F_{n}$:};
                    \node [main node] (x1) [right = 1.3cm of x0] {};
                    \node [main node] (x2) [above left = 0.5cm and 0.5cm of x1] {};
                    \node [main node] (x3) [above right = 0.5cm and 0.5cm of x1] {};
                    \node [main node] (x4) [below left = 0.5cm and 0.5cm of x1] {};
                    \node [main node] (x5) [below left = 0.5cm and 0.5cm of x4] {};

                    \node (x11) [below = 0.01cm of x1] {\tiny{$3$}};
                    \node (x22) [above = 0.01cm of x2] {\tiny{$2$}};
                    \node (x33) [above = 0.01cm of x3] {\tiny{$4$}};
                    \node (x44) [below = 0.01cm of x4] {\tiny{$2$}};
                    \node (x55) [below = 0.01cm of x5] {\tiny{$1$}};
                    \path[draw]
                    (x2)--(x1)--(x3)
                    (x1)--(x4)--(x5);
                \end{tikzpicture} 
            }
            \\ \hline

            \raisebox{0.2em}{
                \begin{tikzpicture}
                    \node (y0) {$z_i\in F_n$:};
                    \node (b) [below = 0.01cm of y0] {\tiny{$(3\le i\le n)$}};
                    \node [main node] (y1) [above right = 0.1cm and 0.8cm of y0] {};
                    \node [main node] (y2) [above left = 0.5cm and 0.5cm of y1] {};
                    \node [main node] (y3) [above right = 0.5cm and 0.5cm of y1] {};
                    \node [main node] (y4) [below right = 0.5cm and 0.5cm of y1] {};
                    \node [main node] (y5) [below right = 0.8cm and 0.8cm of y4] {};
                    \node [main node] (y6) [below right = 0.5cm and 0.5cm of y5] {};

                    \node (y11) [right = 0.01cm of y1] {\tiny{$2$}};
                    \node (y22) [above = 0.01cm of y2] {\tiny{$1$}};
                    \node (y33) [above = 0.01cm of y3] {\tiny{$3$}};
                    \node (y44) [right = 0.01cm of y4] {\tiny{$3$}};
                    \node (y44) [right = 0.01cm of y5] {\tiny{$(i-1)$}};
                    \node (y44) [right = 0.01cm of y6] {\tiny{$i$}};
                    \path[draw]
                    (y2)--(y1)--(y3)
                    (y1)--(y4)
                    (y5)--(y6);
                    \path[draw,dashed]
                    (y4)--(y5);
            \end{tikzpicture}}
            & 
            \begin{tikzpicture}
                \node (yx0) {$z_{i}\in H_{n}$:};
                \node (b) [below = 0.01cm of yx0] {\tiny{$(1\le i\le n)$}};
                \node [main node] (yx1) [right =  0.8cm of yx0] {};
                \node [main node] (yx2) [above left = 0.5cm and 0.5cm of yx1] {};
                \node [main node] (yx3) [above right = 0.5cm and 0.5cm of yx1] {};
                \node [main node] (yx4) [below left = 0.5cm and 0.5cm of yx1] {};
                \node [main node] (yx5) [below right = 0.5cm and 0.5cm of yx4] {};
                \node [main node] (yx6) [below right = 0.8cm and 0.8cm of yx5] {};
                \node [main node] (yx7) [below right = 0.5cm and 0.5cm of yx6] {};

                \node (yx11) [right = 0.01cm of yx1] {\tiny{$2$}};
                \node (yx22) [above = 0.01cm of yx2] {\tiny{$1$}};
                \node (yx33) [above = 0.01cm of yx3] {\tiny{$3$}};
                \node (yx44) [right = 0.01cm of yx4] {\tiny{$1$}};
                \node (yx55) [right = 0.01cm of yx5] {\tiny{$2$}};
                \node (yx66) [right = 0.01cm of yx6] {\tiny{$(i-1)$}};
                \node (yx77) [right = 0.01cm of yx7] {\tiny{$i$}};
                \path[draw]
                (yx2)--(yx1)--(yx3)
                (yx1)--(yx4)--(yx5)
                (yx6)--(yx7);
                \path[draw,dashed]
                (yx5)--(yx6);
            \end{tikzpicture}\\ \hline
        \end{tabular}
        \vspace{1em} 
        \caption{Typical long stubs of $\sw$ in $B_n, \tilde C_{n-1}, F_n$ and $H_n$}
        \label{tab:zBCFH}
    \end{table}

    \begin{table}
        \begin{tabularx}{\textwidth}{|X|c|X|}
            \hline
            \begin{tikzpicture}
                \node (y0) {$z_i$:};
                \node (b) [below = 0.01cm of y0] {\tiny{$(i\in \Z_{<0})$}};
                \node [main node] (y1) [right = 1.8cm of y0] {};
                \node [main node] (y2) [above left = 0.5cm and 0.5cm of y1] {};
                \node [main node] (y3) [above right = 0.5cm and 0.5cm of y1] {};
                \node [main node] (y4) [below left = 0.5cm and 0.5cm of y1] {};
                \node [main node] (y5) [below left = 0.8cm and 0.8cm of y4] {};
                \node [main node] (y6) [below left = 0.5cm and 0.5cm of y5] {};

                \node (y11) [right = 0.01cm of y1] {\tiny{$0$}};
                \node (y22) [above = 0.01cm of y2] {\tiny{$1$}};
                \node (y33) [above = 0.01cm of y3] {\tiny{$v$}};
                \node (y44) [right = 0.01cm of y4] {\tiny{$-1$}};
                \node (y44) [right = 0.01cm of y5] {\tiny{$(i+1)$}};
                \node (y44) [right = 0.01cm of y6] {\tiny{$i$}};
                \path[draw]
                (y2)--(y1)--(y3)
                (y1)--(y4)
                (y5)--(y6);
                \path[draw,dashed]
                (y4)--(y5);
            \end{tikzpicture}
            &
            \raisebox{1.5em}{
                \begin{tikzpicture}
                    \node (0) {$z_{v}$:};
                    \node [main node] (1) [right = 0.8cm of 0] {};
                    \node [main node] (2) [above left = 0.5cm and 0.5cm of 1] {};
                    \node [main node] (3) [above right = 0.5cm and 0.5cm of 1] {};
                    \node [main node] (4) [below  = 0.5cm  of 1] {};

                    \node (11) [right = 0.01cm of 1] {\tiny{$0$}};
                    \node (22) [above = 0.005cm of 2] {\tiny{$-1$}};
                    \node (33) [above = 0.01cm of 3] {\tiny{$1$}};
                    \node (44) [below = 0.01cm of 4] {\tiny{$v$}};
                    \path[draw]
                    (2)--(1)--(3)
                    (1)--(4);
                \end{tikzpicture}
            }
            & 
            \begin{tikzpicture}
                \node (y0) {$z_i$:};
                \node (b) [below = 0.01cm of y0] {\tiny{$(i\in \Z_{>0})$}};
                \node [main node] (y1) [right = 0.8cm of y0] {};
                \node [main node] (y2) [above left = 0.5cm and 0.5cm of y1] {};
                \node [main node] (y3) [above right = 0.5cm and 0.5cm of y1] {};
                \node [main node] (y4) [below right = 0.5cm and 0.5cm of y1] {};
                \node [main node] (y5) [below right = 0.8cm and 0.8cm of y4] {};
                \node [main node] (y6) [below right = 0.5cm and 0.5cm of y5] {};

                \node (y11) [right = 0.01cm of y1] {\tiny{$0$}};
                \node (y22) [above = 0.005cm of y2] {\tiny{$-1$}};
                \node (y33) [above = 0.01cm of y3] {\tiny{$v$}};
                \node (y44) [right = 0.01cm of y4] {\tiny{$1$}};
                \node (y44) [right = 0.01cm of y5] {\tiny{$(i-1)$}};
                \node (y44) [right = 0.01cm of y6] {\tiny{$i$}};
                \path[draw]
                (y2)--(y1)--(y3)
                (y1)--(y4)
                (y5)--(y6);
                \path[draw,dashed]
                (y4)--(y5);
            \end{tikzpicture}\\\hline
        \end{tabularx}
        \vspace{1em} 
        \caption{The long stubs of $\cals(E_{q,r})$}
        \label{tab:zE}
    \end{table}

We record a few features of the stubs in Theorem \ref{thm:stubDescription} for future use:

\begin{lemma}
    \label{lemm:stubRemarks}
    Let $(W,S)$ be an arbitrary $\la(2)$-finite Coxeter group from Theorem \ref{thm:stubDescription} and keep the notation
    of the theorem.  Then the stubs in $\sw$ have the following properties.  
    \begin{enumerate}
        \item The left descent set of each short stub  $x_{tt'}$, medium stub $y_t$ \tul including
             $y'_0,y''_0$ in type $E_{q,r}$\tur and long stub $z_t$ equals $\{t,t'\}, \{t\}$ and
             $\{t\}$, respectively. In particular, a stub $w\in \sw$ has two left descents if and
             only if it is a short stub, and in this case we can recover $w$ from $\call(w)$: if
             $\call(w)=\{t,t'\}$ then $w=x_{tt'}$.
        \item Let $w\in \sw$ be a medium or long stub and let $w=w_p\dots w_1$ be its Cartier--Foata
             factorization.  Then $l(w_1)=2$ and  $l(w_j)=1$ for all $2\le j\le p$, so we may pick a
             reduced word $ \ul{w}= s_1\dots s_{k}(s_{k+1}s_{k+2}) $ of $w$ such that
             $w_1=s_{k+1}s_{k+2}$, $w_2=s_{k}, w_3=s_{k-1}$, etc.
        \item Let $w$ and $\ul w$ be as in \textup{(2)}. Then in the heap $H(\ul w)$, every element
             $i<k+1$ is comparable both to $k+1$ and to $k+2$ via the convex chains of coverings
             $i\prec i+1\prec \dots k\prec k+1$ and $i\prec i+1\prec \dots k\prec k+2$. In
             particular, the element $k$ is both the only element covered by $k+1$ and the only
             element covered by $k+2$ in $H(\ul w)$. 
        \item Let $w$ and $\ul w$ be as in \textup{(2)}. Then we may start with $w$ and successively
             apply left lower star operations with respect to the noncommuting pairs of generators
             $\{s_1,s_2\}, \{s_2,s_3\},\dots, \{s_{k},s_{k+1}\}$ to obtain the sequence $w,
             s_2s_3\dots s_k(s_{k+1}s_{k+2}), s_3\dots s_{k}(s_{k+1}s_{k+2}), \dots,$
             $s_k(s_{k+1}s_{k+2}), w_1$. All elements in this sequence lie in $\sw$ and share the
             same left cell \tul and hence the same right descents\tur\,as $w$. In particular, $w$ is
             in the same left cell and has the same right descents as its first layer $w_1 =
             s_{k+1}s_{k+2}$, which is itself a short stub.
            \end{enumerate}
    \end{lemma}
    \begin{proof} All the claims are readily confirmed by inspection of Theorem \ref{thm:stubDescription}
          and the pictures in Tables \ref{tab:xyStubs}--\ref{tab:zE}.     
    \end{proof}

    \begin{remark}
      \label{rmk:a.vs.n}
      The question of whether the $n$-value (as defined in the paragraph before Proposition \ref{prop:a=n}) and
      $\la$-value of an FC element always agree in a general Coxeter system is open. The results of
      this subsection allow us to enhance Proposition \ref{prop:a=n}.(1) and give a partial answer to this
      question in the following way: 

\begin{prop} 
\label{prop:a.vs.n}
Let $(W,S)$ be an arbitrary Coxeter system and let $w\in \fc(W)$.  \begin{enumerate} \item We
   have $\la(w)=0$ if and only if $n(w)=0$.  \item We have $\la(w)=1$ if and only if
   $n(w)=1$.
\item If $(W,S)$ is $\la(2)$-finite, then $\la(w)=2$ if and only if $n(w)=2$.  \end{enumerate}
  \end{prop} 
\begin{proof} (1) By Proposition \ref{prop:facts01}.(1), the equations $\la(w)=0$ and $n(w)=0$ both hold
     exactly when $w$ is the identity element and has an empty reduced word, so they are
     equivalent.

(2) If $\la(w)=1$, then $n(w)\neq 0$ by Part (1) and $n(w)\le 1$ by Proposition \ref{prop:a=n}.(1), and
therefore $n(w)=1$.  Conversely, if $n(w)=1$ then the heap $H(w)$ is a chain, so no reduced word
of $w$ contains two consecutive generators that commute. The definition of FC elements then
implies that $w$ has a unique reduced word, so $\la(w)=1$ by Proposition \ref{prop:facts01}.(2). It follows
that $\la(w)=1$ if and only if $n(w)=1$.

(3) Suppose $(W,S)$ is $\la(2)$-finite. If $\la(w)=2$, then $n(w)\notin\{0,1\}$ by Parts (1)--(2)
and $n(w)\le 2$ by Proposition \ref{prop:a=n}.(1), and therefore $n(w)=2$. Now suppose $n(w)=2$. We prove that
$\la(w)=2$ by induction on the length $l(w)$ of $w$. We have $l(w)= \abs{H(w)}\ge n(w)=2$, so in the
base case we have $l(w)=2$. Since $n(w)=2$, we must have $w=st$ for two commuting generators $s,t\in
S$ in this case; therefore $\la(w)=2$ by Proposition \ref{prop:Facts}.(2), as desired.  If $l(w)>2$, then $w$
either admits or does not admit a right lower star operation. In the former case, if a right lower
star operation takes $w$ to some element $w'$ then we have $\la(w)=\la(w')=n(w')=n(w)=2$, where the
first equality holds by Corollary \ref{coro:starA}, the second equality holds by induction, and the third
equality holds because star operations do not change $n$-values of FC elements (see the proof of
Proposition 3.15 in \cite{GreenXu}). In the latter case, the element $w$ is a left stub, so
Corollary \ref{coro:firstTwoLayers} implies that the first two layers of $w$ satisfy Conditions (1)--(2) from
the corollary. As pointed out in the proof of Theorem \ref{thm:stubDescription}, this forces $w$ to be an
element in the set $\sw$ given in the theorem, and therefore $w$ indeed has $\la$-value 2.  It
follows by induction that $\la(w)=2$. We have proved that $\la(w)=2$ if and only if $n(w)=2$.
\end{proof}
    \end{remark}

\subsection{Stubs parameterize 1-cells}
\label{sec:1cells}
Let $(W,S)$ be a \ntatf Coxeter system throughout this subsection. We characterize 1-cells in $W_2$
via stubs in two ways, first in terms of star operations in Theorem \ref{thm:cellParameterization} and then
in terms of reduced words and the weak Bruhat order in Theorem \ref{thm:CellViaDecomp}. To start, we show
that distinct left stubs lie in distinct right cells in the following two results.

\begin{lemma}
    \label{lemm:stubDescents}
    Let $u,w$ be two distinct left $\la(2)$-stubs in $W$.  Then we either have $\call(u)\neq \call(w)$ or have
    $\call(u)=\call(w)=\{s\}$ for some $s\in S$. Moreover, in the latter case the elements $u':=su$ and $w':=sw$
    are also stubs, and we have $\call(u')\neq \call(w')$. 
\end{lemma}
\begin{proof} 
    We use Lemma \ref{lemm:stubRemarks}. Suppose that $\call(u)=\call(w)$.  Since $u\neq w$ by assumption and left stubs
    with two left descents can be recovered from the descent by Part (1) of the lemma, we must have
    $\call(u)=\call(w)=\{s\}$ for some $s\in S$. This proves the first claim.  For the second claim, note that in
    the notation of Theorem \ref{thm:stubDescription}, we can have $\call(u)=\call(w)=\{s\}$ for distinct stubs $u,w$ only
    in the following cases where $(W,S)$ is of type $E_{q,r}, F_n$ or $H_n$; outside these types, the claim holds
    vacuously.

    \begin{enumerate}[label=\textup{(\roman*)}]
        \item \textbf{Case 1:} $(W,S)$ is of type $E_{q,r}$, and we have either one of two subcases:
            \begin{enumerate}                
                \item $s=0$ and $u,w$ are two $y$-stubs from the set $\{y_0,y'_0,y''_0\}$ where \[y_0=0\cdot
                    x_{(-1)1}, \quad y_0'=0\cdot x_{(-1)v},\quad y''_0=0\cdot x_{1v};\]
                \item $s=i$ for some nonzero integer $i$ with $-q<i<r$, and 
                    \[ 
                        \{u,w\}=\{y_i=i\cdot
                    x_{(i-1)(i+1)},\; z_i=i\cdot z'\} \] 
                    where $z'$ is a long stub.
            \end{enumerate}
        \item\textbf{Case 2:} 
            $(W,S)$ is of type $F_n$, and there is some $1<i<n$ such that 
            \[ \{u,w\}=\{y_i=i\cdot x_{(i-1)(i+1)},\; z_i=i\cdot z'\} \]
            where $z'$ is a long stub.
        \item\textbf{Case 3:} $(W,S)$ is of type $H_n$, and there is some $1<i<n$ such that 
            \[ \{u,w\}=\{y_i=i\cdot x_{(i-1)(i+1)},\; z_i=i\cdot z'\}  \]
            where $z'$ is a long stub.
    \end{enumerate} 
   We have $\call(u')\neq \call(w')$ for the stubs  $u'=su$ and $w'=sw$ in Case 1(a) by inspection,
   and the same is true in all the other cases by Lemma \ref{lemm:stubRemarks}.(1) because one of $u',w'$
   is a short stub while the other is not.
\end{proof}

\begin{prop}
    \label{prop:distinctCells}
    Let $u,w$ be distinct left $\la(2)$-stubs of $W$. Then $u\not\sim_R w$. 
\end{prop}
\begin{proof}
    Keep the notation from Theorem \ref{thm:stubDescription}, Lemma \ref{lemm:stubDescents}, and the proof of
    the lemma.  We have $u\not\sim_R w$ if $\call(u)\neq \call(w)$ by Proposition \ref{prop:Facts}.(3), so
    Lemma \ref{lemm:stubDescents} implies that it suffices to treat Cases 1--3 from its proof, where
    $\call(u)=\call(w)=\{s\}$ for some $s\in S$. We may show $u\not\sim_R w$ by finding a generator
    $t\in S$ such that $m(s,t)=3$ and $\call(*u)\neq \call(*w)$, where $*$ denotes the simple left
    star operation with respect to $\{s,t\}$: the fact that $\call(*u)\neq \call(*w)$ implies that
    $*u\not\sim_R *w$, so $u\not\sim_R w$ by Proposition \ref{prop:SimpleStarOpAndCells}.(2). We explain how to
    find such a generator $t$ below.

    In Case 1(a) of the proof of Lemma \ref{lemm:stubDescents}, we may take $t$ to be the unique
    generator appearing in both the stubs $u'=su$ and $w'=sw$.  For example, if $u=y_0$ and
    $w=y_0'$, then $u'= x_{(-1)1}, w'=x_{(-1)v}$ and we may take $t=-1$, whence $*u=u', *w=w'$ and
    $\call(*u)\neq \call(*w)$.  In Cases 1(b), 2 and 3, we note that $s$ has a numerical label $i$
    and that $i-1,i+1$ are also generators in $S$. By our labeling of generators
    (Remark \ref{rmk:labels}), it follows that $m(i,i+1)=3$ except when $(W,S)$ is of Coxeter type $F_n$,
    $i=2$ and $\{u,w\}=\{y_2=2(13),z_2=2\cdot 3\cdot (24)\}$.  In this exceptional case we pick
    $t=i-1=1$, so that $m(s,t)=3$ and $\{*u,*w\}=\{13=x_{13},1\cdot 2\cdot 3\cdot (24)=z_1\}$ with
    respect to $\{s,t\}$. In all other cases we may pick $t=i+1$, whence $m(s,t)=3$ and
    $\{*u,*w\}=\{*y_i,*z_i\}=\{x_{(i-1)(i+1)},z_{i+1}\}$ with respect to $\{s,t\}$. Here, the fact
    that $*z_i=z_{i+1}$ in the last set can be easily seen from heaps in Tables \ref{tab:zBCFH} and
    \ref{tab:zE}, with the star operation on $z_i$ being a lower one if and only if $(W,S)$ is of
    type $E_{q,r}$ and $-q<i<0$. In all cases, one of $*u$ and $*w$ is an short stub and the other
    is a long stub, so $\call(*u)\neq \call(*w)$ by Lemma \ref{lemm:stubRemarks}.(1), as desired.
\end{proof}

We are ready to describe cells in terms of stubs and star operations.

\begin{definition}
    \label{def:closure}
    For each element $w\in W$, we define the \emph{right upper star closure} of $w$ to be the set of
    all elements $y\in W$ for which there exists a sequence $z_1=w, z_2,\dots, z_q=y$ such that
    $z_{i+1}$ can be obtained from $z_{i}$ via a right upper star operation for all $1\le i\le q-1$;
    we denote the set by $R_w$. Similarly, for each $w\in W$ we define its \emph{left upper star
    closure} to be the set $L_w$ containing all elements that can be obtained from $w$ via a
    sequence of left upper star operations.
\end{definition}

\begin{thm}
    \label{thm:cellParameterization}
    Let $W$ be a nontrivially $\la(2)$-finite Coxeter group, let $W_2=\{w\in W:\la(w)=2\}$, and let
    $\cals(W)$ be the set of left $\la(2)$-stubs in $W$.  Then the set $R_w$ forms a right \kl cell
    for every $w\in \sx$.  Moreover, we have  $R_w\cap R_{w'}=\emptyset$ for distinct
    $w,w'\in\cals(W)$, and $W_2=\sqcup_{w\in \cals(W)}R_w$. In particular, the number of right cells
    in $W_2$ equals the cardinality of $\sw$ given in Part \textup{(}7\textup{)} of
    Theorem \ref{thm:stubDescription}.
\end{thm}

\begin{proof}
    Let $R(w)$ be the right \kl cell containing $w$ for each $w\in W$. Then we have $R_w\se R(w)$ by
    Proposition \ref{prop:StarOpAndCells}.  On the other hand, by the definition of stubs we have
    $W_2=\cup_{w\in \sw} R_w$.  It follows that to prove the theorem it suffices to show that
    $w\not\sim_R w'$ whenever $w,w'$ are distinct elements in $\sw$. This holds by
    Proposition \ref{prop:distinctCells}.
\end{proof}

Using the fact that the set $R_w$ is a right cell for each $w\in\sw$, we now work towards a second
description of the cells, in terms of reduced words.

\begin{definition}
    \label{def:stubDecomp}
    Let $w\in W$. We define a \emph{left stub decomposition} of $w$ to be a reduced factorization of
    the form $w=x\cdot z$ where $x$ is a left stub and $z\in W$.  Similarly, we define a \emph{right
    stub decomposition} of $w$ to be a reduced factorization of the form $w=z\cdot y$ where $y$ is a
    right stub.  
\end{definition}

\begin{thm}
    \label{thm:CellViaDecomp}
    Let $x\in \sw$ and let $w$ be an element of $\la$-value 2 in $W$. Let $R(x)$ be the right cell containing $x$
    \textup{(}so that $R(x)=R_x$ by Theorem \ref{thm:cellParameterization}\textup{)}.
    \begin{enumerate}
        \item We have  $w\in R_x$ if and only if $w$ has a left stub decomposition of the form $w=x\cdot z$.
        \item We have $w\in R(x)$ if and only if $x\le^R w$, i.e. we have 
        \[
            R(x)=\{z\in W_2: x\le^R z\}.
        \]
        \item The element $w$ has a unique left stub decomposition in the sense that the stub $x'$ and the
             element  $z'$ in the decomposition $w=x'\cdot z'$ are both unique.
    \end{enumerate}
\end{thm}

\noindent
Note that Part (2) of the theorem, which characterizes the right cell in $W_2$ in terms of the right
weak Bruhat order, may be viewed as an analog of Proposition \ref{prop:facts01}.(3). To prove the theorem, we
will use the following lemma. 

\begin{lemma}
    \label{lemm:decompAndDescent}
    Let $(W,S)$ be an arbitrary Coxeter system and let $w\in \fc(W)$.
    \begin{enumerate}
        \item Suppose that $w$ has a reduced word of the form 
            \[
                \ul w=s_1\dots s_k (s_{k+1}s_{k+2}\dots s_{k+j})s_{k+j+1}\dots s_{q}
         \]
         where the set $A:=\{k+1,k+2,\dots,k+j\}$ forms a maximal antichain in
             the heap $H(w)$. Let $x=s_1\dots s_k(s_{k+1}\dots s_{k+j})$ and $y=(s_{k+1}\dots s_{k+j})s_{k+j+1}\dots s_q$.
             Then we have $H(x)=\cali_A$ and $H(y)=\calf_A$ as sets. 
            
        \item In the above setting, we have $\call(w)=\call(x)$ and $\calr(w)=\calr(y)$.  

        \item Let $x\in \sw$ and let $w$ be an element of $\la$-value 2 with a
    left stub decomposition of the form $w=x\cdot z$. Then $\call(w)=\call(x)$.
    \end{enumerate}
\end{lemma}
\begin{proof}
    (1) Since $A$ is a maximal antichain in $H(w)$, we have $H(\ul w)=\cali_{A}\cup\calf_{A}$ by
    Lemma \ref{lemm:maxAntichain}. For any $k+j+1\le i\le q$, we have $i\notin\cali_A$ by the definition
    of heaps; therefore $i\in \calf_A$. Similarly, we have $i\in \cali_A$ for all $k+1\le i\le q$.
    It follows that $H(x)=\cali_A$ and $H(y)=\calf_A$.

    (2) Part (1) implies that the set of minimal elements of $H(x)$ and $H(w)$ coincide, so
    $\call(w)= \call(x)$ by Remark \ref{rmk:heapsAndWords}. Similarly we have $\calr(w)=\calr(y)$.

    (3) By Lemma \ref{lemm:stubRemarks}.(2), there is a reduced word $\ul x=s_1\dots
    s_{k}(s_{k+1}s_{k+2})$ of $x$ where $(s_{k+1}s_{k+2})$ equals the first layer of $x$. Let
    $s_{k+3}\dots s_{q}$ be a reduced word of $z$, so that $\ul{w}:=s_1\dots s_k\cdot
    (s_{k+1}s_{k+2})\cdot s_{k+3}\dots s_{q}$ is a reduced word of $w$.  The set $A:=\{k+1,k+2\}$
    forms a maximal antichain in the heap $H(\ul w)$ by Corollary \ref{coro:layerSizeBound}.(1), and
    therefore $\call(w)=\call(x)$ by Parts (1) and (2).
\end{proof}

\begin{proof}[Proof of Theorem \ref{thm:CellViaDecomp}]
    Part (1) implies Part (2) by the definition of $\le^R$ and Proposition \ref{prop:Facts}.(5). Part (3) also
    follows from Part (1): if $w=x''\cdot z''$ is another left stub decomposition of $w$ then $w\in
    R(x')$ and $w\in R(x'')$, and therefore $x\sim_R w\sim_R x''$; this forces $x'=x''$, and hence
    $z'=z''$, by Proposition \ref{prop:distinctCells}.

    It remains to prove Part (1). The ``only if'' implication follows from the definition of $R_x$
    and upper star operations. To prove the ``if'' implication, suppose that $w$ has a stub
    factorization of the form $w=x\cdot z$, and let $y\in \sw$ be the unique stub such that $w\in
    R_y$. The unique existence of such a stub is guaranteed by Theorem \ref{thm:cellParameterization}, and
    we need to prove that $x=y$.

    Since $w\in R_y$, we have $w\sim_R y$ and hence  $\call(y)=\call(w)$ by Proposition \ref{prop:Facts}.(3).
    On the other hand, since $w=x\cdot z$, we have $\call(x)=\call(w)$ by
    Lemma \ref{lemm:decompAndDescent}. It follows that $\call(x)=\call(y)$, where $1\le
    \abs{\call(x)}=\abs{\call(y)}\le 2$ by the analog of Corollary \ref{coro:layerSizeBound}.(3) for left
    descents. If $\abs{\call(x)}=\abs{\call(y)}=2$, then $x=y$ by Lemma \ref{lemm:stubRemarks}.(1).  If
    $\abs{\call(x)}=\abs{\call(y)}=1$, say with $\call(x)=\call(y)=\{s\}$, then $l(x),l(y)>2$ and
    the elements $x':=sx$ and $y'=sy$ are stubs with $\call(x')\neq \call(y')$ by
    Lemma \ref{lemm:stubRemarks}.(4). The element $w':=sw$ has stub decompositions of the form
    $w'=x'\cdot z'$ and $w'=y'\cdot z''$, which forces $\call(x')=\call(w')=\call(y')$ by
    Lemma \ref{lemm:decompAndDescent}. Since $\call(x)=\call(y)$ and $\call(x')=\call(y')$,
    Lemma \ref{lemm:stubDescents} implies that $x=y$.  The proof is complete.  
\end{proof}

\begin{remark}
    \label{rmk:a2finiteAssumption}
    The assumption that $(W,S)$ be $\la(2)$-finite is crucial for the validity of most results in
    the subsection.  For example, consider the affine Weyl system $\tilde C_2$, whose Coxeter
    diagram is shown in \autoref{fig:ac2}.  The system is not $\la(2)$-finite, and the elements
    $x=acb$ and $y=acbac$ in $W$ are both left $\la(2)$-stubs by Corollary \ref{coro:firstTwoLayers}. Note
    that $\call(x)=\call(y)=\{a,c\}$, that $\la(x)=n(x)=\la(y)=n(y)=2$ by Proposition \ref{prop:a=n}.(2), that
    $x\le^R y$, and that consequently $x\sim_{R}y $ by Proposition \ref{prop:Facts}.(5). It is then easy to
    check that the conclusions of Lemma \ref{lemm:stubDescents}, Proposition \ref{prop:distinctCells},
    Theorem \ref{thm:cellParameterization} and Theorem \ref{thm:CellViaDecomp} all fail for $x$ and $y$. On the
    other hand, the assertions in Lemma \ref{lemm:decompAndDescent} do hold for all Coxeter systems, as
    indicated in the statement of the lemma. 

    \begin{figure} \begin{center} \begin{tikzpicture} \node[main node] (1) {}; \node[main node] (2) [right = 1cm of
          1] {}; \node[main node] (3) [right = 1cm of 2] {};

                \node (11) [below = 0.18cm of 1] {$a$};
                \node (22) [below = 0.1cm of 2] {$b$};
                \node (33) [below = 0.18cm of 3] {$c$};

                \path[draw]
                (1) edge node [above] {$4$} (2)
                (2) edge node [above] {$4$} (3);
            \end{tikzpicture}
        \end{center}
        \caption{The Coxeter system of type $\tilde C_2$}
        \label{fig:ac2}
    \end{figure}
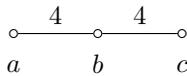
\end{remark}

As usual, Theorem \ref{thm:cellParameterization} and Theorem \ref{thm:CellViaDecomp} have obvious left-handed
analogs which hold by symmetry.  First, for every right $\la(2)$-stub $x'$, the left upper star
closure $L_{x'}$ of $x'$ must be identical with the left cell that contains $x'$; moreover, we have
$W_2=\sqcup_{x'\in\cals'(W)}L_{x'}$ where $\cals'(W)$ is the set of all right stubs in $W_2$.
Second, for each right stub $x'\in \cals'(W)$, the left cell $L_{x'}$ consists precisely of all
$\la(2)$-elements $w$ that admit a right stub decomposition of the form $w=z\cdot x'$.  It is also
possible to describe left cells in $W_2$ via left stubs: for every $w\in \sw$,
Remark \ref{rmk:stubRemarks1}.(1) and Proposition \ref{prop:Facts}.(4) imply that
\begin{equation}
    \label{eq:1cellSymmetry} L_{w\inverse}=R_w\inverse
\end{equation} where $R_w\inverse=\{x\inverse: x\in R_w\}$, because both sets can be described as
    the left cell containing $w\inverse$. We can now label the 0-cells in $W_2$ with pairs of left
    stubs: 

\begin{definition}  
For all $x,y\in \sw$, we define $I(x,y):=R_{x}\cap R_y\inverse=R_x\cap L_{y\inverse}$ and define
$N(x,y):=\abs{I(x,y)}$.
\label{def:intersection}
\end{definition} 

\begin{remark}
    By Equation \eqref{eq:1cellSymmetry}, for all $x,y\in \sw$, the 0-cell $I(y,x)=R_y\cap R_x\inverse$ is in
    bijection with the 0-cell $I(x,y)$ via the inversion map $I(x,y)\ra I(y,x), w\mapsto w\inverse$.
    In particular, we have $N(x,y)=N(y,x)$.
    \label{rmk:Isymmetry}
\end{remark}

The following proposition, where the last part follows from the first two parts, is immediate from
Corollary \ref{coro:intersection}.

\begin{prop}
    \label{prop:a2intersections}
    Let $\sw$ be the set of left $\la(2)$-stubs of $W$, and let $E$ be a two-sided $\la(2)$-cell of $W$. Set
    $\mathcal{T}=E\cap\sw$. Then
    \begin{enumerate}
        \item For each $x\in \calt$, the right cell $R_x\se E $ equals the disjoint union \[
             R_{x}=\bigsqcup_{y\in \calt} I(x,y).\] Consequently, we have \[
             \abs{R_x}=\sum_{y\in\calt} N(x,y).  \]
        \item  The decomposition of the two-sided cell $E$ into right cells is given by
             \[E=\bigsqcup_{x\in \calt}R_x.\] Consequently, we have \[ \abs{E}=\sum_{x\in
             \calt}\abs{R_x}.\] 
         \item The two-sided cell $E$ equals the disjoint union \[ E=\bigsqcup_{x,y\in \mathcal{T}}
              I(x,y).  \] Consequently, we have \[ \abs{E}=\sum_{x,y\in\mathcal{T}}N(x,y).  \]
    \end{enumerate}
\end{prop} 
\noindent
Note that in light of the proposition, to count the cells in $W_2$ it suffices to understand the
0-cells $I(x,y)$. We will study these 0-cells in \autoref{sec:0cells}. 

\subsection{Slide equivalence and 2-cells}
\label{sec:2cells} 
In this subsection, we introduce two equivalence relations on $\sw$ and use one of them to determine
how the right cells $R_w (w\in \sw$ coalesce into 2-cells in $W_2$.  Recalling that the first layer
of every stub $w\in \sw$ is itself a short stub by Lemma \ref{lemm:stubRemarks}.(4), we define the
relations in two steps, as follows.

\begin{definition}
    \label{def:slidings} We define a \emph{slide} to be a transformation taking a short stub $st\in
          \sw$ to a short stub $su\in \sw$ where $m(t,u)\ge 3$, i.e, where $\{t,u\}$ forms an edge
          in the Coxeter diagram. We say the move is \emph{along the edge $\{t,u\}$} and call it a
          \emph{simple} slide if the edge $\{t,u\}$ is simple, i.e. if $m(t,u)=3$.  
\end{definition}
\begin{definition}
    \label{def:equivalences} Let $w,w'\in \sw$ and let $x,x'$ be their respective first layers.
    \begin{enumerate}
    \item We say that $w$ and $w'$ are \emph{slide equivalent}, and write $w\approx w'$, if $x$ and
         $x'$ can be related by a (possibly empty) sequence of slides. 
    \item We say that $w$ and $w'$ are \emph{simple slide equivalent}, and write $w\sim w'$, if $x$
         and $x'$ can be related by a (possibly empty) sequence of simple slides.
    \end{enumerate}
\end{definition}

It is clear that $\approx$ and $\sim$ are equivalence relations, and that $\sim$ refines $\approx$. 

\begin{remark}
\label{rmk:slidings}
Let $w,w'$ and $x,x'$ be as in Definition \ref{def:equivalences}. 
\begin{enumerate}
\item Short stubs equal their first layers, so $w\approx x$, that is, every stub is slide equivalent
     to a short stub.  It also follows that every slide equivalence class $C$ in $\sw$ can be
     recovered from the set $Z_C$ of short stubs it contains, namely, as the set of stubs whose
     first layer lies in $Z_C$. The above facts also hold for the relation $\sim$.  For these
     reasons, we will often first focus on short stubs when studying $\approx$ and $\sim$.
\item If $w\approx w'$ then $w\approx x\approx x'\approx w'$, where we can connect $w$ to $x$ and
     connect $x'$ to $w'$ with left star operations by Lemma \ref{lemm:stubRemarks}.(4).  Since $x$ and
     $x'$ can be related by slides by Definition \ref{def:equivalences}, it follows that  the relation
     $\approx$ is generated by left star operations and slides.  Similarly, the relation $\sim$ is
     generated by left star operations and simple slides.
\end{enumerate} 
\end{remark}

For the rest of this subsection we will focus on the relation $\approx$ and its connection with
2-cells. The relation $\sim$ will be studied further in \autoref{sec:0cells} in connection with
0-cells; see Proposition \ref{prop:invariance} and \autoref{sec:cellData}.

Slide equivalent stubs lie in the same 2-cells:

\begin{lemma}
\label{lemm:sliding2cell}
    Let $w,w'\in \sw$. If $w\approx w'$, then we have $w\sim_{LR} w'$. 
\end{lemma}

\begin{proof}
    Slides can be achieved by star operations: for any two short stubs $x=st$ and $x'=su$ where
    $m(t,u)\ge 3$, we have ${}_*(x^*)={}_* (stu)=x'$ where the two star operations are performed
    with respect to $\{t,u\}$. It follows from Remark \ref{rmk:slidings}.(2) that if $w\approx w'$, then
    $w$ and $w'$ can be related by left and right star operations, hence $w\sim_{LR} w'$ by
    Proposition \ref{prop:StarOpAndCells},
\end{proof}

Lemma \ref{lemm:sliding2cell} suggests that to find the 2-cells of $W_2$ it is helpful to find the
equivalence classes of $\approx$. We do so in Lemma \ref{lemm:path}, Lemma \ref{lemm:Esliding} and
Proposition \ref{prop:desClasses} below. The lemmas focus on short stubs, as justified in
Remark \ref{rmk:slidings}.(1). For convenience, we will say a subset of $\sw$ is \emph{slide connected} if
its elements are pairwise slide equivalent.

\begin{lemma}
    \label{lemm:path}
    Let $G$ be the Coxeter diagram of an arbitrary Coxeter system $(W,S)$. Suppose a subset $S'$ of
    $S$ induces a linear subgraph in $G$ in the sense that we can label the elements of $S'$ as
    $1,2,\dots, k$ in a way such that $m(i,{j})\ge 3$ if and only if $\abs{j-i}=1$ for all $1\le
    i,j\le k$.  Let $X$ be the set of short stubs in $\sw$ whose supports lie in $S'$. 
    \begin{enumerate}
        \item We have $X= \{{i}j: 1\le i,j\le n, \abs{i-j}>1\}$.
        \item The set $X$ is slide connected.
    \end{enumerate}
\end{lemma}
\begin{proof} Part (1) follows from definitions. Part (2) holds since the elements of $X$ can be
      arranged into the array \[
        \begin{matrix} 13 & 14 & 15 & \dots & 1n\\ & 24 & 25 & \dots & 2n\\ &&\dots &\dots&\dots\\
              &&&(n-3)(n-1)& (n-3)n\\ &&&& (n-2)n
        \end{matrix} \] where every two adjacent entries in a row or in a column are short stubs
            related by a slide along an edge of the form $\{i,i+1\}$ in $G$.
\end{proof}

\begin{lemma}
\label{lemm:Esliding}
Let $(W,S)$ be a Coxeter system of type $E_{q,r}$ where $r\ge q\ge 1$. Consider the short stubs $x_1:=(-1)v,
x_2:=1v$ and $x_3:=(-1)1$ in $\sw$, and let $Z$ be the set of all short stubs in $\sw$. 
    \begin{enumerate}
        \item If $q=r=1$, then none of $x_1,x_2,x_3$ admits any slides.
        \item If $q=1$ and $r>1$, then $x_1$ admits no slides, we have $x_2\approx x_3$, and the set
            $Z\setminus\{x_1\}$ is slide connected. 
         \item If $r\ge q>1$, then $Z$ is slide connected. 
    \end{enumerate}
\end{lemma}

\begin{proof}
    We start with some observations, the key one being that if $r>1$, then we have $x_2\approx x_3$ via the
    sequence of slides (see Figure \ref{fig:sliding}) 
    \begin{equation}
    \label{eq:slidings} 
    x_2=1v=v1\ra v2 \ra 02\ra (-1)2\ra (-1)1=x_3.
    \end{equation} 
    Similarly, we have $x_1\approx x_3$ whenever $q>1$. We also consider the partition of the set
    $Z':=Z\setminus\{x_1\}$ into the sets \[ X=\{x\in Z\setminus\{x_1\}: v\in
    \supp(x)\},\, Y=\{x\in
    Z\setminus\{x_1\}: v\not\in \supp(x)\}  \] 
    where we have $x_2=1v\in X$ and $x_3=(-1)1\in Y$. The set $X$ is slide connected because its
    elements can be listed in the sequence $1v, 2v, \dots, rv$ where every two adjacent stubs are
    related by a slide. The set $Y$ is slide connected by an application of Lemma \ref{lemm:path} with
    $S'=S\setminus\{v\}$.  
 
    To prove the lemma, first note that if $q=1$ then $x_1=(-1)v$ admits no slide because  both the
    generators $-1$ and $v$ are adjacent to and only to $0$ in the Coxeter graph, so the claim about
    $x_1$ from Part (2) holds. Part (1) holds by similar arguments.  Next, suppose that $r>1$. Then
    $x_2\approx x_3$ as we observed, and the set $Z'=X\sqcup Y$ is slide connected because $x_2\in
    X, x_3\in Y$, and $X,Y$ are slide connected. This completes the proof of Part (2). Finally, if
    in addition we have $q>1$ then $x_3\approx x_1$. Since the set $Z'=Z\setminus\{x_1\}$ is slide
    connected and contains $x_3$, it follows that $Z$ is slide connected, so Part (3) holds.
\end{proof}

\begin{example}
    \label{eg:sliding}
    The sequence of slides in \eqref{eq:slidings} is depicted in Figure \ref{fig:sliding}, where in each
    copy of the Coxeter diagram we have omitted the vertices $3,4,\dots,r$, filled the generators in
    the support of the short stub, and indicated the slide to be performed with arrows. Note that
    the slides rely on the existence of the vertex $2$, and hence on the assumption that $r>1$, in a
    crucial way: sliding from $x_2$ to $x_3$ is impossible when $r=1$ but becomes possible when
    $r>1$ because the edge $\{1,2\}$ permits more slides.

    \begin{figure}
        \centering
        \begin{tikzpicture}
            \node[main node] (z) {};
            \node[main node] (0) [right=0.4cm of z] {};
            \node[main node, fill] (v) [above=0.4cm of 0] {};
            \node[main node,fill] (1) [right=0.4cm of 0] {};
            \node[main node] (2) [right=0.4cm of 1] {};
            \node (3) [right=0.2cm of 2] {\tiny{$\ra$}};

            \path[draw]
            (z)--(0)--(1)--(2)
            (0)--(v);

            \node(a) [above =0.01cm of 1] {};
            \node(b) [above =0.01cm of 2] {};

            \path[draw,->]
            (a)--(b);

            \node (c11) [below =0.1cm of z] {\tiny{$-1$}};
            \node (c22) [below =0.1cm of 0] {\tiny{$0$}};
            \node (c33) [below =0.1cm of 1] {\tiny{$1$}};
            \node (c44) [below =0.1cm of 2] {\tiny{$2$}};
            \node (c66) [above =0.01cm of v] {\tiny{$v$}};

            \node[main node,right=0.1cm of 3] (1z) {};
            \node[main node] (10) [right=0.4cm of 1z] {};
            \node[main node, fill] (1v) [above=0.4cm of 10] {};
            \node[main node] (11) [right=0.4cm of 10] {};
            \node[main node,fill] (12) [right=0.4cm of 11] {};
            \node (13) [right=0.2cm of 12] {\tiny{$\ra$}};

            \path[draw]
            (1z)--(10)--(11)--(12)
            (10)--(1v);

            \node(1a) [right =0.01cm of 1v] {};
            \node(1b) [right =0.01cm of 10] {};

            \path[draw,->]
            (1a)--(1b);

            \node (1c11) [below =0.1cm of 1z] {\tiny{$-1$}};
            \node (1c22) [below =0.1cm of 10] {\tiny{$0$}};
            \node (1c33) [below =0.1cm of 11] {\tiny{$1$}};
            \node (1c44) [below =0.1cm of 12] {\tiny{$2$}};
            \node (1c66) [above =0.01cm of 1v] {\tiny{$v$}};

            \node[main node,right=0.1cm of 13] (2z) {};
            \node[main node,fill] (20) [right=0.4cm of 2z] {};
            \node[main node] (2v) [above=0.4cm of 20] {};
            \node[main node] (21) [right=0.4cm of 20] {};
            \node[main node,fill] (22) [right=0.4cm of 21] {};
            \node (23) [right=0.2cm of 22] {\tiny{$\ra$}};

            \path[draw]
            (2z)--(20)--(21)--(22)
            (20)--(2v);

            \node(2a) [above =0.01cm of 20] {};
            \node(2b) [above =0.01cm of 2z] {};

            \path[draw,->]
            (2a)--(2b);

            \node (2c11) [below =0.1cm of 2z] {\tiny{$-1$}};
            \node (2c22) [below =0.1cm of 20] {\tiny{$0$}};
            \node (2c33) [below =0.1cm of 21] {\tiny{$1$}};
            \node (2c44) [below =0.1cm of 22] {\tiny{$2$}};
            \node (2c66) [above =0.01cm of 2v] {\tiny{$v$}};

            \node[main node,right=0.1cm of 23,fill] (3z) {};
            \node[main node] (30) [right=0.4cm of 3z] {};
            \node[main node] (3v) [above=0.4cm of 30] {};
            \node[main node] (31) [right=0.4cm of 30] {};
            \node[main node,fill] (32) [right=0.4cm of 31] {};
            \node (33) [right=0.2cm of 32] {\tiny{$\ra$}};

            \path[draw]
            (3z)--(30)--(31)--(32)
            (30)--(3v);

            \node(3a) [above =0.01cm of 32] {};
            \node(3b) [above =0.01cm of 31] {};

            \path[draw,->]
            (3a)--(3b);

            \node (3c11) [below =0.1cm of 3z] {\tiny{$-1$}};
            \node (3c22) [below =0.1cm of 30] {\tiny{$0$}};
            \node (3c33) [below =0.1cm of 31] {\tiny{$1$}};
            \node (3c44) [below =0.1cm of 32] {\tiny{$2$}};
            \node (3c66) [above =0.01cm of 3v] {\tiny{$v$}};

            \node[main node,right=0.1cm of 33,fill] (4z) {};
            \node[main node] (40) [right=0.4cm of 4z] {};
            \node[main node] (4v) [above=0.4cm of 40] {};
            \node[main node,fill] (41) [right=0.4cm of 40] {};
            \node[main node] (42) [right=0.4cm of 41] {};

            \path[draw]
            (4z)--(40)--(41)--(42)
            (40)--(4v);

            \node (4c11) [below =0.1cm of 4z] {\tiny{$-1$}};
            \node (4c22) [below =0.1cm of 40] {\tiny{$0$}};
            \node (4c33) [below =0.1cm of 41] {\tiny{$1$}};
            \node (4c44) [below =0.1cm of 42] {\tiny{$2$}};
            \node (4c66) [above =0.01cm of 4v] {\tiny{$v$}}; 
        \end{tikzpicture}
        \caption{Sliding from $x$ to $z$ when $r>1$}
        \label{fig:sliding}
    \end{figure}
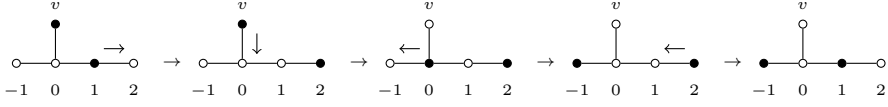 
\end{example}

    We can now describe the slide equivalence classes of $\sw$.

    \begin{prop}
        \label{prop:desClasses}
        Let $(W,S)$ be an $\la(2)$-finite Coxeter system.
\begin{enumerate}
    \item Suppose $(W,S)$ is of type $A_n (n\ge 3), B_n (n\ge 3), \tilde C_{n-1} (n\ge 5), F_n (n\ge
         4), H_n (n\ge 3)$ or of type $E_{q,r}$ where $r\ge q>1$. Then $\sw$ contains a single slide
         equivalence class.
    \item Suppose $(W,S)$ is of type $E_{q,r}$ where $r\ge q=1$. Let $x_1=(-1)v, x_2=1v,x_3=(-1)1$
         and let $C_i=\{w\in \sw: w\approx x_i\}$ for $1\le i\le 3$. 
        \begin{enumerate}
            \item[\tul a\tur] If $r=1$, then $\sw$ consists of three slide equivalence classes. The
                 classes are exactly $C_1, C_2, C_3$, and we have \[ C_1=\{(-1)v, 0(-1)v,
                 10(-1)v\},\] \[C_2= \{1v, 01v, (-1)0v\},\] \[C_3=\{(-1)1, 0(-1)1, v0(-1)1\}.\]
            \item[\tul b\tur] If $r>1$, then $C_2=C_3$ and the set $\sw$ consists of two slide
                 equivalence classes. The classes are exactly $C_1, C_2$, and we have \[
                 \hspace{2.5em} C_1=\{(-1)v, 0(-1)v, 10(-1)v,210(-1)v,\dots, r\dots 210(-1)v\}.  \]
        \end{enumerate}
\end{enumerate}
    \end{prop}

    \begin{proof}
        (1) By Remark \ref{rmk:slidings}.(1) it suffices to show that the set of all short stubs in $\sw$
        is slide connected. This holds in type $E_{q,r}$ when $r\ge q>1$ by Lemma \ref{lemm:Esliding}.(3)
        and holds in types $A_n, B_n, \tilde C_{n-1}, F_n, H_n$ by applications of
        Lemma \ref{lemm:path}.(2) with $S'=S$.  

        (2) If $r=1$, then the stubs $x_1,x_2,x_3$ are all the short stubs in $\sw$, and they lie in
        three distinct slide equivalence classes by Lemma \ref{lemm:Esliding}.(1). Thus,
        Remark \ref{rmk:slidings}.(1) implies that  $C_1, C_2, C_3$ are the slide equivalence classes in
        $\sw$ and are given by \[ C_i=\{w\in \sw: \text{the first layer of $w$ equals $x_i$}\} \]
        for each $i$. The equations in Part (a) then follow from the above equality by direct
        computation. Part (b) can be proved using Lemma \ref{lemm:Esliding}.(2) in a similar way, with
        $C_2=C_3$ since $x_2\approx x_3$ when $r>1$. 
\end{proof}

We are ready to find the 2-cells in $W_2$ in all $\la(2)$-finite Coxeter systems:

\begin{thm}
\label{thm:twoSidedCells}
Let $(W,S)$ be an irreducible \ntatf Coxeter system.
\begin{enumerate}
    \item Suppose $(W,S)$ is of type $A_n (n\ge 3), B_n (n\ge 3), \tilde C_{n-1} (n\ge 5), F_n (n\ge 4), H_n (n\ge
        3)$ or of type $E_{q,r}$ where $r\ge q>1$. Then $W_2$ contains a single 2-cell.
    \item Suppose $(W,S)$ is of type $E_{q,r}$ where $r\ge q=1$. Let $C_1, C_2,C_3$ be as in
        Proposition \ref{prop:desClasses}.(2) and let $E_i= \sqcup_{w\in C_i}R_w$ for all $1\le i\le 3$. 
        If $r=1$, then $W_2$ contains three 2-cells and they are exactly 
        $E_1, E_2$ and $E_3$. If $r>1$, then $W_2$ contains two 2-cells and they are exactly 
        $E_1$ and $E_2$.
    \item Let $\mathcal{C}$ be the set of slide equivalence classes in $\sw$ and let $\mathcal{E}$
         be the set of 2-cells in $W_2$. Then there is a bijection $\Phi: \mathcal{C}\ra
         \mathcal{E}$ given by $\Phi(C)=\cup_{w\in C}R_w$ for all $C\in \mathcal{C}$. 
\end{enumerate}
\end{thm}

\begin{proof}
Part (1) follows directly from Theorem \ref{thm:cellParameterization}, Lemma \ref{lemm:sliding2cell} and
Proposition \ref{prop:desClasses}.(1). To prove Part (2), first note that $E_i$ is a subset of a 2-cell
by Theorem \ref{thm:cellParameterization} and Lemma \ref{lemm:sliding2cell} for all $1\le i\le 3$. Also
note that by Proposition \ref{prop:desClasses}.(2), we have  $W_2=E_1\sqcup E_2\sqcup E_3$ if $r=1$ and
$W_2=E_1\sqcup E_2$ if $r>1$.  These facts, together with the symmetry in the Coxeter diagram
of $E_{1,1}$, imply that to show the 2-cells of $W_2$ in type $E_{1,r} (r\ge 1)$ are as
claimed it remains to prove that $E_1$ and $E_2$ lie in different 2-cells for all $r\ge 1$.
We do so by showing below that the stubs $x:=(-1)v\in E_1$ and $y:=1v\in E_2$ are not in the
same 2-cell.

Suppose $x\sim_{LR} y$. Then Part (9) of Proposition \ref{prop:Facts} implies that some element $z\in W$
satisfies $z\sim_L x$ and $z\sim_R y$. Parts (3) and (5) of the same proposition then imply that
$\calr(z)=\calr(x)=\{-1,v\}, \call(z)=\call(y)=\{1,v\}$ and $\la(z)=\la(x)=2$. Let $z=w_p\dots
w_2w_1$ be the Cartier--Foata form of $z$. Recall that $w_1$ must equal the product of the right
descents of $z$ and that every generator in $\supp(w_2)$ must fail to commute with some
generator in $\supp(w_1)$. It follows that $w_1=(-1)v$ and $w_2=0$. Corollary \ref{coro:firstTwoLayers}
then implies that $z\in \sw$, whence the fact that $\call(z)=\{1,v\}$ forces $z=1v$ by
Lemma \ref{lemm:stubRemarks}.(1).  This contradicts the fact that $z=w_p\dots w_3\cdot
w_2w_1=w_p\dots w_3\cdot (01v)$ has length at least 3, so $x\not\sim_{LR} y$, and the proof of
Part (2) is complete. 

Part (3) follows immediately from comparison of Parts (1)--(2) of the theorem with
Proposition \ref{prop:desClasses}.
\end{proof}
   
As mentioned in the introduction, the partition of some Weyl and affine Weyl groups such as $A_n$
and $B_n$ into cells are known in the literature. Cells in the finite Coxeter groups of types $H_3$
and $H_4$ have also been determined using computational methods in \cite{AlvisH4}. In this sense,
the descriptions of the Kazhdan--Lusztig cells of $W_2$ from Theorems
\ref{thm:cellParameterization}, \ref{thm:CellViaDecomp} and \ref{thm:twoSidedCells} are not new for
those groups.  On the other hand, we note that these three theorems offer convenient combinatorial
descriptions of the cells of $W_2$, and it does so in uniform way for all $\la(2)$-finite Coxeter
groups. For the Coxeter systems of types $F_n$ where $n\ge 6$, $H_n$ where $n\ge 5$, $\tilde
C_{n-1}$ for general values of $n$, and $E_{q,r}$ for general values of $q,r$, the descriptions of
cells from these theorems are new to our knowledge.

\section{Enumeration via 0-cells}
\label{sec:0cells}

We maintain the setting and notation in \autoref{KLcells}, with $(W,S)$ being an irreducible \ntatf
Coxeter system.  In this section we compute the cardinalities of the 1-cells and 2-cells of $W_2$
found in \autoref{sec:1cells} and \autoref{sec:2cells}. By Proposition \ref{prop:a2intersections}, to do so it
suffices to understand the 0-cells $I(x,y)$ and their cardinalities $N(x,y)$ in $W_2$ where $x,y\in
\sw$, so the study of 0-cells occupies almost the entire section. We will interpret each 0-cell as a
certain set of \emph{core elements} in \autoref{sec:anatomy}, explain how different 0-cells relate
to each other in \autoref{sec:relating0cells}, and exploit such relationships to count all 0-cells,
1-cells and 2-cells of $W_2$ in \autoref{sec:cellData}.

\subsection{Anatomy of \texorpdfstring{$\la(2)$}{a(2)}-elements}
\label{sec:anatomy}
In this subsection we establish a canonical decomposition that illuminates the structure of both the
elements and the 0-cells in $W_2$. The decomposition uses the following notions: 

\begin{definition}
\label{def:anatomy}
\begin{enumerate}
    \item We call an element $w\in W_2$ a \emph{core element}, or simply a \emph{core}, if
         $\abs{\call(w)}=\abs{\calr(w)}=2$.
    \item We say that an ordered pair $(w,w')$ of FC elements in $W$ are \emph{descent compatible}
         if the sets $\calr(w)=\call(w')$ and $\abs{\calr(w)}=\abs{\call(w')}=2$.  When this is the
         case, Remark \ref{rmk:heapsAndWords} guarantees the existence of reduced words 
        \begin{equation*} \ul w=s_1\dots s_{k}(s_{k+1}s_{k+2}), \quad \ul w'
              =(s_{k+1}s_{k+2})s_{k+3}\dots s_{k'}
        \end{equation*} of $w$ and $w'$, respectively, such that
            $\calr(w)=\call(w')=\{s_{k+1},s_{k+2}\}$, and we define the \emph{glued product} of $w$
            and $w'$ to be the element of $W$ expressed by the word \[ \ul w*\ul w':=s_1\dots
            (s_{k+1}s_{k+2})\dots s_{k'}.  \] We denote the glued product by $w*w'$.
    \item Let $x\in \sw, y'\in \cals'(W)$ and let $w\in W$ be a core. We say $w$ is \emph{compatible
         with both $x$ and $y$} if both the pairs $(x,w)$ and $(w,y')$ are descent compatible. We
         denote the set of cores compatible with both $x$ and $y'$ by $\cores(x,y')$. 
    \item We define an $\la(2)$-\emph{triple} in $W$ to be a triple $(x,w,y')$ of elements in $W_2$
         such that $x\in \sw$, $y'\in \cals'(W)$, and $w\in\cores(x,y')$.  We denote the set of
         $\la(2)$-triples in $W$ by $\triples$.
\end{enumerate}
\end{definition}

\begin{remark}
    \label{rmk:glue}
    The element $w*w'$ in Part (2) of Definition \ref{def:anatomy} is well-defined, independent of the choice
    of the reduced words $\ul w$ and $\ul w'$. Indeed, in the notation of the definition, we have
    $w*w'=w{s_{k+1}s_{k+2}}w'$. It is worth noting, however, that the element $w*w'$ need not be FC
    when $w$ and $w'$ are FC. For example, in the Coxeter system of type $A_4$ the glued product of
    the FC elements $w=124,w'=241$ equals $1241=(121)4$, which is not FC.
\end{remark}

We describe the canonical decomposition of $\la(2)$-elements below.

\begin{thm}
    \label{thm:anatomy}
    Let $(W,S)$ be an irreducible \ntatf Coxeter system. Then there is a bijection $g:\triples
    \ra W_2$ given by 
    \[
        g(x,w,y')=(x*w)*y'
    \]
    for all $(x,w,y')\in \triples$. Furthermore, for all $x\in \sw$
    and $y'\in \cals'(W)$, the map $g$ restricts to a bijection $g_{xy'}: \{x\}\times \cores(x,y')\times \{y'\}\ra
    R_x\cap L_{y'}$.
\end{thm}

\noindent
The theorem immediately relates 0-cells to sets of cores:

\begin{corollary}
    Let $(W,S)$ be an irreducible \ntatf Coxeter system. Then for all
    $x,y\in\sw$, there is a bijection $f_{xy}: \cores(x,y\inverse)\ra I(x,y)$ given by 
    \[
        f_{xy}(w)=(x*w)*y\inverse
    \]
    for every $w\in\cores(x,y\inverse)$. In particular, we have $N(x,y)=\abs{\cores(x,y\inverse)}$.
    \label{coro:intersectionsViaCores}
\end{corollary}

\begin{example}
    Suppose $(W,S)$ is of type $B_4$. Consider the left stubs $x=13, y=213$ and the core element
    $w=132413\in \cores(x,y')$ compatible with both $x$ and $y':=y\inverse$. The canonical bijection
    $g$ of Theorem \ref{thm:anatomy} sends the $\la(2)$-triple $(x,w,y)$ to the element $z:=1324132\in
    I(x,y)$, as shown in Figure \ref{fig:anatomy}.  In consequence, we have $f_{xy}(w)=z$ for the induced
    bijection $f_{xy}:\cores(x,y\inverse)\ra I(x,y)$ of Corollary \ref{coro:intersectionsViaCores}.

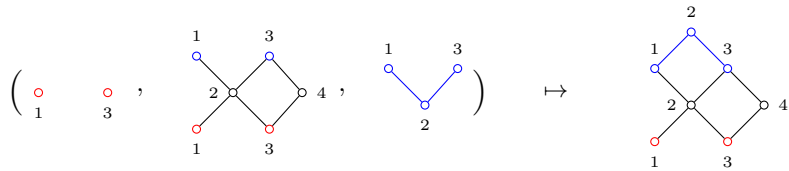
\begin{figure}
    \begin{center}
        \begin{tikzpicture}
            \node (xx) {\huge{$($}};
                \node[main node, red] (1) [right=0.01cm of xx] {};
                \node[main node, red] (2) [right=0.8cm of 1] {};
                \node (y1) [right=0.2cm of 2] {,};
                \node[main node] (5) [right = 1cm of y1] {}; 
                \node[main node,blue] (3) [above left = 0.4cm and 0.4cm of 5] {}; 
                \node[main node,red] (4) [below left = 0.4cm and 0.4cm of 5] {}; 
                \node[main node,blue] (6) [above right = 0.4cm and 0.4cm of 5] {}; 
                \node[main node,red] (7) [below right = 0.4cm and 0.4cm of 5] {}; 
                \node[main node] (8) [right = 0.8cm of 5] {}; 
                \node (y2) [right=0.3cm of 8] {,};
                \node[main node,blue] (9) [above right = 0.1cm and 0.4cm of y2] {}; 
                \node[main node,blue] (10) [below right = 0.4cm and 0.4cm of 9] {}; 
                \node[main node,blue] (11) [right = 0.8cm of 9] {}; 
            \node (xxx) [right=5.7cm of xx] {\huge{$)$}};

            \node (x) [right =0.5cm of xxx] {\small{$\mapsto$}};
            \node[main node,blue] (l) [right = 2.5cm of 11] {};
            \node[main node,blue] (n) [above right = 0.4cm and 0.4cm of l] {};
            \node[main node] (o) [below right = 0.4cm and 0.4cm of l] {};
            \node[main node,red] (m) [below left = 0.4cm and 0.4cm of o] {};
            \node[main node,blue] (p) [above right = 0.4cm and 0.4cm of o] {};
            \node[main node,red] (q) [below right= 0.4cm and 0.4cm of o] {};
            \node[main node] (r) [below right= 0.4cm and 0.4cm of p] {};

            \node (1x) [below = 0.01cm of 1] {\tiny{$1$}};

            \node (2x) [below = 0.01cm of 2] {\tiny{$3$}};
            \node (3x) [above = 0.01cm of 3] {\tiny{$1$}};
            \node (4x) [below = 0.01cm of 4] {\tiny{$1$}};
            \node (5x) [left = 0.01cm of 5] {\tiny{$2$}};
            \node (6x) [above = 0.01cm of 6] {\tiny{$3$}};
            \node (7x) [below = 0.01cm of 7] {\tiny{$3$}};
            \node (8x) [right = 0.01cm of 8] {\tiny{$4$}};
            \node (9x) [above = 0.01cm of 9] {\tiny{$1$}};
            \node (10x) [below = 0.01cm of 10] {\tiny{$2$}};
            \node (11x) [above = 0.01cm of 11] {\tiny{$3$}};

            \node (lx) [above= 0.01cm of l] {\tiny{$1$}};
            \node (mx) [below = 0.01cm of m] {\tiny{$1$}};
            \node (nx) [above = 0.01cm of n] {\tiny{$2$}};
            \node (ox) [left = 0.01cm of o] {\tiny{$2$}};
            \node (px) [above = 0.01cm of p] {\tiny{$3$}};
            \node (qx) [below = 0.01cm of q] {\tiny{$3$}};
            \node (rx) [right = 0.01cm of r] {\tiny{$4$}};

            \path[draw]
            (m)--(o)--(l)
            (p)--(r)--(q)
            (p)--(o)--(q)
            (3)--(5)--(6)--(8)--(7)--(5)--(4);

            \path[draw,blue]
            (9)--(10)--(11)
            (l)--(n)--(p);
        \end{tikzpicture}
    \end{center}
    \caption{The glued product of an $\la(2)$-triple}
    \label{fig:anatomy}
\end{figure}
\end{example}

We now begin the proof of Theorem \ref{thm:anatomy}. An important subtlety about the theorem concerns the
definition of the map $g$: in order for the supposed output $(x*w)*y'$ to be defined, we need $x*w$
to be an FC element that is descent compatible with $y$, yet as we explained in Remark \ref{rmk:glue} the
glued product of two FC elements need not be FC in general.  This subtlety is addressed in the
following proposition and reflected in its rather technical proof, where we frequently have to
invoke certain features of $\la(2)$-stubs in $\la(2)$-finite Coxeter systems highlighted in
Lemma \ref{lemm:stubRemarks}.

\begin{prop}
    \label{prop:plugin}
    Let $x\in \sw$ and let $w\in W_2$. Suppose that the pair $(x,w)$ is descent compatible, and let
    $\ul x, \ul w, \ul x*\ul w$ and $x*w$ be as in Part \textup{(2)} of Definition \ref{def:anatomy}. Then
    $\ul x*\ul w$ is a reduced word, the element $x*w$ is FC, and $\la(x*w)=2$. Moreover, we have
    $\calr(x*w)=\call(y)$, so that $x*w$ is descent compatible with $y$. 
\end{prop}

\begin{proof}
    Let $\ul x= s_1\dots s_k(s_{k+1}s_{k+2})$ and $\ul w=(s_{k+1}s_{k+2})s_{k+3}\dots s_{k'}$ be
    reduced words of $x$ and $w$, respectively, with $\calr(x)=\call(w)=\{s_{k+1},s_{k+2}\}$. Let
    $\ul z=\ul x*\ul w=x_1\dots s_{k'}$. We first use Proposition \ref{prop:FCCriterion} to show that $\ul z$
    is a reduced word of an FC element, and hence $x*w$ is FC, by examining the heap $H(\ul
    z)=\{1,\dots,k'\}$. To do so, we view the heaps $H(\ul x)=\{1,\dots, k+2\}$ and $H(\ul
    w)=\{k+1,\dots, k'\}$ of $\ul x$ and $\ul w$ naturally as sub-posets of $H(\ul z)$. Let
    $A=\{k+1,k+2\}$. Recall from Lemma \ref{lemm:stubRemarks}.(3) that for every element $i\in H(\ul x)$,
    we have both $i\preceq k+1$ and $i\preceq k+2$ in $H(\ul x)$. On the other hand, we claim that
    every element $j\in H(\ul w)$ satisfies either $k+1\preceq j$ or $k+2 \preceq j$ in $H(\ul w)$.
    To see this, note that since $\call(w)=\{s_{k+1},s_{k+2}\}$, the set $A$ consists of minimal
    elements in $H(w)$, so it is an antichain in $H(\ul w)$ by Remark \ref{rmk:heapsAndWords}. Since
    $\la(w)=2$, this antichain is maximal by Corollary \ref{coro:layerSizeBound}. Since $A$ is a maximal
    antichain consisting of all minimal elements in $H(\ul w)$, the filter generated by $A$ in
    $H(\ul w)$ must equal $H(\ul w)$ by Lemma \ref{lemm:maxAntichain}. Our claim follows.

    By Proposition \ref{prop:FCCriterion}, to prove that $\ul z$ is the reduced word of an FC element it
    suffices to show that $H(\ul z)$ contains no covering relation $i\prec j$ with $s_i=s_j$ or
    convex chains $C=(i_1\prec i_2\prec \dots \prec i_m)$ where the elements' labels alternate in
    two noncommuting generators $s,t\in S$ and $m=m(s,t)$.  Since $x$ and $w$ are both FC, such
    covering relations or convex chains cannot exist in $H(\ul x)=\{1,\dots, k+2\}$ or in $H(\ul
    w)=\{k+1,\dots,q\}$, so if they exist in $H(\ul z)$, then they must involve indices from both
    $\ul x$ and $\ul w$, that is, we must have either (a) a covering relation $i\prec j$ where
    $i<k+1, j>k+2$, and $s_i=s_j$ or (b) a convex chain of the form $C$ where $i_1<k+1,i_m>k+2$.
    Write $i=i_1$ and $j=i_m$ in Case (b).  Then in both Cases (a) and (b), the last paragraph shows
    that in $H(\ul z)$ the index $i$ is smaller than both elements of $A$ in the poset order and the
    index $j$ is larger than at least one element of $A$. It follows that that $i\preceq k+1\preceq
    j$ or $i\preceq k+2\preceq j$, so $j$ cannot cover $i$ and Case (a) cannot occur. It also
    follows that in Case (b) the convex chain $i_1,\dots, i_m$ contains three consecutive elements
    $a, b, c$ such that $a<k+1,b\in \{k+1,k+2\}, c>k+2$, and $s_a=s_c$. By
    Lemma \ref{lemm:stubRemarks}.(3), the element $a$ covered by $b$ must be $k$. Let $t=s_a=s_k$. Then
    $s_c=t$ and $t$ commutes with neither $s_{k+1}$ nor $s_{k+2}$ by Corollary \ref{coro:firstTwoLayers}.(2).
    But then we must have both $a\preceq k+1\preceq c$ and $a\preceq k+2\preceq c$ in $H(\ul z)$,
    contradicting the convexity of the chain $i_1,\dots, i_m$. We may now conclude that $x*w$ is a
    FC element with reduced word $\ul x*\ul w$. 

    Next, we show that $\la(x*w)=2$ by induction on the length $l(x)=k+2\ge 2$ of the stub $x$. In
    the base case, we have $k=2$ and $x*w=w$, and therefore $\la(x*w)=\la(w)=2$. If $k+2>2$, then
    the element $x':=s_1x=s_2\dots s_{k+1}s_{k+2}$ is also a stub and satifies $\calr(x')=\calr(x)$
    by Lemma \ref{lemm:stubRemarks}.(4), so $x'*w$ is defined and has $\la$-value 2 by induction. The
    element $1$ is the unique minimal element in the heap $H(\ul x)$, so it is also the unique
    minimal element in $H(\ul z)$ since we argued that $k+1\preceq j$ or $k+2\preceq j$ in $H(\ul
    z)$ for all $j\in H(\ul z)\setminus H(\ul x)$. Also, in the heap $H(\ul x)$ we have the covering
    relation $1\prec 2$ by Lemma \ref{lemm:stubRemarks}.(3). It then follows from the left-sided analog
    of Proposition \ref{prop:Reducibility} that $x*w$ admits a left lower star operation with respect to
    $\{s_1,s_2\}$ that takes it to the element $s_2s_3\dots s_k'=x'*w$; therefore
    $\la(x*w)=\la(x'*w)=2$ by Corollary \ref{coro:starA}, as desired. 

    It remains to show that $\calr(x*w)=\calr(w)$. Since $\la(x*w)=2$, the antichain $A=\{k+1,k+2\}$
    in $H(\ul z)=H(x*w)$ must be maximal by Corollary \ref{coro:layerSizeBound}. Since $x*w$ is FC, it
    follows from Definition \ref{def:anatomy}.(2) and Lemma \ref{lemm:decompAndDescent}.(2) that
    $\calr(x*w)=\calr(w)=\call(y)$, so $x*w$ is descent compatible with $y$.
\end{proof}

Besides Proposition \ref{prop:plugin}, we prepare one more ingredient for the proof of Theorem \ref{thm:anatomy},
namely how suitable glued products relate to stub decompositions and hence 1-cells. 

\begin{prop}
    \label{prop:glueAndCells}
      \begin{enumerate}
          \item For every stub $x\in \sw$ and every element $w\in W_2$ such that the pair $(x,w)$ is descent
              compatible, the element $z:=x*w$ is in the right cell of $x$.
        \item For every element $z\in W_2$, there exists a unique way to write $z$ as a glued product $z=x*w$ where
             $x\in \sw$ and $w$ is an $\la(2)$-element such that the pair $(x,w)$ is compatible.
    \end{enumerate}
\end{prop}

\begin{proof} 
(1) We have $\la(z)=2$ by Proposition \ref{prop:plugin}, and writing $\ul x, \ul w$ and $\ul z=\ul x*\ul w$ as
in Proposition \ref{prop:plugin} yields a left stub decomposition $z=x\cdot w'$ where $w'=s_{k+3}\dots s_{k'}$.
It follows that $z\sim_R x$ by Theorem \ref{thm:CellViaDecomp}.(1).

\noindent
(2) We can find the candidates for the elements $x$ and $w$ as follows. Let $x\in \sw$ be the unique
stub such that $z\in R_x$, so that by Theorem \ref{thm:CellViaDecomp} we have $z=x\cdot w'$ for some
element $w'$. Let $\ul w'=s_{k+3}s_{k+4}\dots s_k'$ be a reduced word of $w'$, and let $w$ be the
element expressed by the word $\ul w:=(s_{k+1}s_{k+2}) s_{k+3}\dots s_{k'}=s_{k+1}\dots s_{k'}$.

We now check that $\la(w)=2$ and $z=x*w$. To do so, note that since the factorization $z=x\cdot
w'=s_1\dots s_{k'}$ is reduced, the word $\ul w$ is reduced and $w\le_L z$. In particular, we have
$\la(w)\le \la(z)=2$ by Parts (1) and (5) of Proposition \ref{prop:Facts}. Meanwhile, the elements $s_{k+1}$
and $s_{k+2}$ are left descents of $w$, so the elements $(k+1)$ and $(k+2)$ form an antichain in the
heap $H(w)=\{k+1, k+2,\dots,k'\}$ and $\la(w)\ge n(w)\ge 2$. It follows that $\la(w)=2$ and
$\call(w)=\{s_{k+1},s_{k+2}\}$.  It further follows that $z=x*w$, as desired.

To prove the uniqueness claim, suppose $z=x_0*w_0$ is another factorization where $x_0\in \sw,
\la(w_0)=2$ and $(x_0,w_0)$ is descent compatible. Then $x_0\sim_R z$ by Part (1), so $x_0=x$ by
Theorem \ref{thm:cellParameterization}. Remark \ref{rmk:glue} now implies that
$xs_{k+1}s_{k+2}w=z=x_0s_{k+1}s_{k+2}w_0$ where $\{s_{k+1},s_{k+2}\}=\calr(x)=\calr(x_0)$, so
$w=w_0$.  The uniqueness claim follows. 
\end{proof}
  
We are ready to prove Theorem \ref{thm:anatomy}.

\begin{proof}[Proof of Theorem \ref{thm:anatomy}]
    Let $(x,w,y')\in \triples$ and let $z=g(x,w,y')$. By Proposition \ref{prop:plugin}, the pair $(x*w,y')$ is
    descent compatible since $x*w$ is an $\la(2)$-element with $\calr(x*w)=\calr(w)=\call(y')$. A
    second application of the proposition then shows that $(x*w)*y'$ has $\la$-value 2, so the
    proposed formula for $g$ does define a map from $\triples$ to $W_2$. Note that by symmetry, we
    have $z=(x*w)*y'=x*(w*y')$, and therefore $z\in R_x$ by Proposition \ref{prop:glueAndCells}.(1). Similarly
    we have $z\in L_{y'}$, and therefore $g$ restricts to a map
    $g_{xy'}:\{x\}\times\cores(x,y')\times \{y'\}\ra R_x\cap L_{y'}$ as claimed. 

    It remains to prove that $g$ is bijective. The surjectivity and injectivity of $g$ follow
    respectively from the existence and uniqueness claims in Proposition \ref{prop:glueAndCells}.(2). More
    precisely, for every $z\in W_2$, the proposition guarantees a unique factorization $z=x* w'$
    with $x\in\sw, w'\in W_2$ and $(x,w')$ descent compatible, and the obvious analog of the
    proposition, involving right stubs and left cells, guarantees a unique factorization $w'=w*y'$
    with $w\in W_2, y'\in \cals'(W)$ and $(w,y')$ descent compatible; it follows that $(x,w,y')$ is
    the unique triple in $\triples$ such that $g(x,w,y')=z$.
\end{proof}

\begin{remark}[Distinguished involutions]
    \label{rmk:distinguishedInvolutions}
    The tools developed in this subsection are well-suited for describing the distinguished
    involutions in $W_2$.  Distinguished involutions are certain important involutions in Coxeter
    groups, and it is known that every right cell $R$ contains a unique distinguished involution,
    which is a member of the 0-cell $R\cap R\inverse$. For each $\la(2)$-finite Coxeter system
    $(W,S)$ and each stub $x\in \sw$, we claim that the distinguished involution in the cell $R_x$
    is precisely the glued product $d:=x*x^{-1}$. This can be proved by showing $d$ is the unique
    involution in $R_x$ for which a certain number $\gamma_{x,x\inverse, d}$ is nonzero. Here, the
    number $\gamma_{x,x\inverse,d}$ can be obtained by directly computing the product
    $C_{x}C_{x\inverse}$ in the Hecke algebra of $(W,S)$ (see \cite[\S 14]{LusztigHecke}), but we
    omit the details. Instead, we note that $d$ is clearly an involution and we have $d\in R_x$ by
    Proposition \ref{prop:glueAndCells}.(1). Consequently we have $d\inverse=d\in R_x$, so that $d\in
    R_x^{\inverse}$ and $d\in R_x\cap R_{x}\inverse=R_x\cap L_{x\inverse}=I(x,x)$. In particular,
    when the 0-cell $I(x,x)$ is a singleton (which is sometimes the case, see
    Proposition \ref{prop:intersectionSizes}), we can conclude that $d$ is the distinguished involution in
    $R_x$ without any computation.
\end{remark}

\subsection{Relating 0-cells}
\label{sec:relating0cells}
In this subsection we use  the relations $\sim$ and $\approx$ introduced in Definition \ref{def:slidings} to
connect different 0-cells in $W_2$. Our results explain how two 0-cells $I(x,y)$ and $I(x',y')$
relate to each other for $x,y,x',y'\in \sw$ when $x\approx x', y\approx y$ or when $x\sim x',y\sim
y'$. Since $\approx$ and $\sim$ are generated by left star operations and suitable slides, we first
study how these generating transformations affect 0-cells. The effect of the left star operations
can be understood via the canonical decomposition of $\la(2)$-elements described in
\autoref{sec:anatomy}:

\begin{lemma}
    \label{lemm:staropAnd0cells}
    Let $x,x',y\in \sw$, suppose that $x$ and $x'$ can be related by left star operations, and let 
\[
f_{xy}: \cores(x,y\inverse)\ra I(x,y), \quad 
f_{x'y}: \cores(x',y\inverse)\ra I(x',y)   
\]
   be the bijections defined in Corollary \ref{coro:intersectionsViaCores}.  Then the sets $\cores(x,y\inverse)$ and
   $\cores(x',y\inverse)$ are equal and the map 
   \[
       \varphi_{xx'}:=f_{x'y}\circ f_{xy}\inverse: I(x,y)\ra I(x',y)
\]
is a bijection. In particular, we have $N(x,y)=N(x',y)$. 
\end{lemma}
\begin{proof} 
    Since $x$ and $x'$ can be related by left star operations, we have $\calr(x)=\calr(x')$ by
    Propositions \ref{prop:StarOpAndCells} and \ref{prop:Facts}.(3). By Definition \ref{def:anatomy}, for any
    two stubs  $z,z'\in\sw$ the set $\cores(z,z')$ depends only on the descent sets $\calr(z)$ and
    $\call(z')$, so $\cores(x,y)=\cores(x',y)$. The remaining claims follow immediately.
\end{proof}

\begin{remark}
    \label{rmk:swap}
    The bijection $\varphi_{xx'}$ from Lemma \ref{lemm:staropAnd0cells} has two simple descriptions.
    First, in the notation of Theorem \ref{thm:anatomy} and Corollary \ref{coro:intersectionsViaCores}, the elements
    of $I(x,y)$ and $I(x',y)$ are precisely the glued products $g(x,w,y)=x*w*y$ and
    $g(x',w,y)=x'*w*y$ where $w\in\cores(x,y)=\cores(x',y)$, so for each element in $I(x,y)$ the map
    $\varphi_{xx'}$ simply replaces the left stub $x$ with $x'$ in the canonical decomposition of
    the element into $\la(2)$-triples. Second, the above description implies that
    $\varphi_{xx'}(z)=x'x\inverse z$ by Remark \ref{rmk:glue}; that is, applying $\varphi_{xx'}$ amounts
    to left multiplication by $x'x\inverse$ in the group $W$. 
\end{remark}

Next, we study how slides on short stubs affect relevant 0-cells.

\begin{lemma}
    \label{lemm:slidingMoves}
    Let $(W,S)$ be an arbitrary Coxeter system. Suppose $s,t,u\in S$ are generators such that $m(s,t)=m(t,u)=2$ and
    $m(t,u)\ge 3$, and let $J=\{t,u\}$. 
    \begin{enumerate}
        \item Let $w\in \fc(W)$ and suppose that $\call(w)=\{s,t\}$. Then $w$ admits at least one left star
             operation $\sigma$ with respect to $J$, and for any such operation we have $\{s,u\}\se
             \call(\sigma(w))$.
         \item Now suppose that $(W,S)$ is $\la(2)$-finite and consider the short stubs $x=st, x'=su$ in $\sw$.
             For each subset $A$ of $W$, let \[ \hspace{1.5em} \Sigma(A)=\{\sigma(w): w\in A, \sigma
         \text{ is a left star operation with respect to } J\}.  \] Then we have $\Sigma(R_{x})=R_{x'}$ and
         $\Sigma(R_{x'})=R_x$. For every $y\in \sw$, we also have $\Sigma(I(x,y))=I(x',y)$ and $\Sigma(I(x',y))=I(x,y)$.
     \item Let $(W,S), x,x'$ and $y$ be as in Part (2), and suppose that $m(t,u)=3$.  Then
          $N(x',y)=N(x,y)$.
    \end{enumerate}
\end{lemma}

\begin{proof}
    (1)    Suppose $\call(w)=\{s,t\}$ and consider the right coset decomposition $w=w_J\cdot {}^J w$
    of $w$ relative to $J$. Since $w\in \fc(W)$, we have $l(w_I)<m(s,t)$ by the word criterion for
    full commutativity (\autoref{sec:FC}). Since $t\in \call(w)$, we have $l(w_I)\ge 1$.  It follows
    that $w$ admits at least one left star operation with respect to $I$. By the definition of star
    operations, any such operation must introduce $u$ as a new  left descent and keep $s$ as a
    descent, hence we have $\{s,u\}\se\call(\sigma(w))$.

    (2) Let $w\in R_x$ and let $L$ be a left cell in $W_2$.  Then $\call(w)= \call(x)=\{s,t\}$ by
    Proposition \ref{prop:Facts}.(3), so Part (1) guarantees that $w$ admits at least one left star operation,
    say $\sigma$, with respect to $J$. Let $z\in \sw$ be the unique stub such that $\sigma(w)\in
    R_z$. Then $\{s,u\}\se \call(\sigma(w))=\call(z)$ by Part (1) and Proposition \ref{prop:Facts}.(3). This
    forces $\call(z)=\{s,u\}$ by Corollary \ref{coro:layerSizeBound}, which in turn forces $z=su=x'$ by
    Lemma \ref{lemm:stubRemarks}.(1). It follows that $\sigma(w)\in R_{x'}$, and therefore
    $\Sigma(R_{x})\se R_{x'}$. A similar argument shows that $\Sigma(R_{x'})\se R_x$. Note that
    since left cells are closed under left star operations by Proposition \ref{prop:StarOpAndCells}, it further
    follows that $\Sigma(R_x\cap L)\se R_{x'}\cap L$ and $\Sigma(R_{x'}\cap L)\se R_x\cap L$.

The operation $\sigma$ can be reversed by another left star operation $\sigma'$ with respect to $J$,
with $\sigma'$ being a lower operation if $\sigma$ is upper and vice versa; therefore
$w=\sigma'\circ\sigma(w)\in \Sigma(\Sigma(R_{x}))$.  It follows that $R_x\se \Sigma(\Sigma(R_x))$.
The same argument shows that $R_x\cap L\se \Sigma(\Sigma(R_{x}\cap L))$.

We have shown that  $\Sigma(R_x)\se R_{x'}$ and $R_x\se \Sigma(\Sigma(R_{x}))$, so $R_{x}\se
\Sigma(\Sigma(R_{x}))\se \Sigma(R_x')$.  Since $\Sigma(R_{x'})\se R_x$, it follows that
$\Sigma(R_x)=R_{x'}$.  Similarly we may conclude that $\Sigma(R_{x'})=R_x, \Sigma(R_x\cap
L)=R_{x'}\cap L$ and $\Sigma(R_{x'}\cap L)=R_x\cap L$ from the previous two paragraphs. The last two
equalities imply the desired claims on $I(x,y)$ and $I(x',y)$ when we set $L=R_y\inverse$, and the
proof of (2) is complete.

(3) Let $w\in R_x$. Since $m(t,u)=3$, the element $w$ admits at most one left star operation
relative to $J$, namely, the simple left star operation $*$. On the other hand $w$ admits at least
one left star operation relative to $J=\{t,u\}$ by Part (1).  It follows that the set $\Sigma(R_x)$
defined in Part (2) equals the set $*(R_x):=\{*w:w\in R_x\}$.  Being an involution, the map $*$ is
injective on the set of elements it is defined on, so we have
$\abs{R_x'}=\abs{\Sigma(R_x)}=\abs{*(R_x)}=\abs{R_x}$, where the first equality holds by Part (2).
Similarly we have $\abs{I(x',y)}=\abs{\Sigma(I(x,y))}=\abs{*(I(x,y))}=\abs{I(x,y)}$, i.e. we have
$N(x',y)=N(x,y)$.
\end{proof}

Lemma \ref{lemm:staropAnd0cells} and Lemma \ref{lemm:slidingMoves}.(3) lead to a numerical invariance that
will be very useful for counting 1-cells and 2-cells in $W_2$ (see Proposition \ref{prop:Nij}): 

\begin{prop}
    \label{prop:invariance}
    Let $x,y,x',y'\in \sw$ and suppose $x'\sim x, y'\sim y$. Then $N(x',y')=N(x,y)$. 
\end{prop}
\begin{proof}
    The relation $\sim$ is generated by left star operations and simple slides by
    Remark \ref{rmk:slidings}.(2), we have $N(x',y)=N(x,y)$ when $x'$ and $x$ are related by left star
    operations by Lemma \ref{lemm:staropAnd0cells}, and we have $N(x',y)=N(x,y)$ if $x'$ and $x$ are
    short stubs related by a simple slide by Lemma \ref{lemm:slidingMoves}.(3), hence $N(x',y)=N(x,y)$.
    Similarly we have $N(y',x')=N(y,x')$. Remark \ref{rmk:Isymmetry} now implies that
    $N(x',y')=N(y',x')=N(y,x')=N(x',y)=N(x,y) $, as desired. 
\end{proof}

\begin{remark}
    \label{rmk:relating0cells}
In the same way that  Lemma \ref{lemm:slidingMoves}.(3) is a reflection of Lemma \ref{lemm:slidingMoves}.(2)
in the special case where the slide involved in simple, Proposition \ref{prop:invariance} is a numerical
reflection of the more general fact that in an $\la(2)$-finite Coxeter system $(W,S)$, we can always
deduce all 0-cells in the same 2-cell of $W_2$ from each other. More precisely, if the system
$(W,S)$ is not of type $E_{1,r}$ for some integer $r\ge 1$, then $W_2$ is itself a single 2-cell by
Theorem \ref{thm:twoSidedCells}, and we can deduce $I(x',y')$ from $I(x,y)$ for all stubs $x,y,x',y'\in
\sw$ as follows: \begin{enumerate}[leftmargin=2.5em]
    \item[(1.a)] If $x$ and $x'$ are short, then we may obtain $x'$ from $x$ by a sequence of slides
         because all short stubs in $\sw$ are slide equivalent by the proof of
         Theorem \ref{thm:twoSidedCells}. By Lemma \ref{lemm:slidingMoves}.(2), this implies that we may start
         from $I(x,y)$ and apply a corresponding sequence of left star operations setwise to obtain
         $I(x',y)$. 
     \item[(1.b)] More generally, let $x_1$ and $x_1'$ be the first layers of $x$ and $x'$,
          respectively. Then $x$ and $x'$ can be related to $x_1$ and $x_1'$ by left star operations
          by Lemma \ref{lemm:stubRemarks}.(4), respectively, so by Lemma \ref{lemm:staropAnd0cells} and
          Remark \ref{rmk:swap} we may obtain $I(x_1,y)$ from $I(x,y)$, and $I(x',y)$ from $I(x_1',y)$,
          via suitable left multiplications. Since $x_1$ and $x_1'$ are short stubs, we can obtain
          $I(x_1',y)$ from $I(x_1,y)$ as explained in (1.a), so we can obtain $I(x',y)$ from
          $I(x,y)$.
     \item[(2)] By symmetry, the suitable counterparts of Lemma \ref{lemm:staropAnd0cells},
          Remark \ref{rmk:swap}, and Lemma \ref{lemm:slidingMoves}.(2) allow us to obtain $I(x',y')$ from
          $I(x',y)$ via right star operations and right multiplications. It follows that we can
          obtain $I(x',y')$ from $I(x,y)$.
 \end{enumerate} When $(W,S)$ is of type $E_{1,r}$ where $r\ge 1$, we may obtain different 0-cells
     from each other in each 2-cells $E_i$ (in the notation of Theorem \ref{thm:twoSidedCells}) by the same
     reasoning because the short stubs in $E_i$ are always pairwise slide equivalent by
     Lemma \ref{lemm:Esliding}.
 \end{remark}

 \begin{example}
     \label{eg:0cellsB4}
     The Coxeter system $(W,S)$ of type $B_4$ contains six $\la(2)$-stubs, namely $x_1:=1\cdot
     2\cdot 13, x_2:=2\cdot 13, x_3=13, x_4=14, x_5=24$ and $x_6=3\cdot 24$. The six stubs are all
     slide equivalent but fall into two simple slide equivalence classes, with $x_1\sim x_2\sim
     x_3\sim x_4$ and $x_5\sim x_6$. Let $R(i)=R_{x_i}$, $L(i)=R(i)\inverse$, $I(i,j)=R(i)\cap L(j)$
     and $N(i,j)=\abs{I(i,j)}$ for all $1\le i,j\le 6$. Then the 0-cells in $W_2$ are given in
     Table \ref{tab:0cellsB4}, where $x_i$ labels the $i$-th row and $i$-th column and where $I(i,j)$
     appears in the $i$-th row and $j$-column for all $1\le i,j\le 6$.

     We may use the procedures described earlier to determine the rows of the table from each other,
     i.e. to obtain an entry from another entry in the same column. For example, since $x_3$ and
     $x_4$ differ by a sliding move along the edge $\{3,4\}$, we may obtain the third and fourth row
     from each other via left  star operations with respect to $\{3,4\}$; since $x_5$ and $x_6$ are
     related to each other by left star operations and $x_5x_6\inverse=x_6x_5\inverse=3$, we may
     obtain the fifth and sixth row from each other using left multiplication by the generator $3$.
     Similarly, we can determine the columns of the table from each other via right star operations
     and right multiplications, so we can recover the entire table from any single entry.
     
     We note that taking the cardinality $N(i,j)$ of each listed 0-cell $I(x,y)$ in
     Table \ref{tab:0cellsB4} recovers Table \ref{tab:b4example} from Example \ref{eg:b4}. Given that $x_1\sim
     x_2\sim x_3\sim x_4$ and $x_5\sim x_6$, both these tables conform to Proposition \ref{prop:invariance}.
     The tables also demonstrate the fact that the conclusion $N(x',y')=N(x,y)$ of
     Proposition \ref{prop:invariance} no longer holds when we weaken the assumption $x\sim x',y\sim y$ to the
     assumption $x\approx x', y\approx y'$. In this sense, the theorem cannot be strengthened.

\begin{table}\centering \def\arraystretch{1.5}
\resizebox{\columnwidth}{!}{
    \begin{tabular}{|c|c|c|c|c|c|c|} \hline
        & $1\cdot 2\cdot13$ & $2\cdot 13$ & $13$ & $14$ & $24$ & $3\cdot 24$\\\hline 
        $1\cdot 2\cdot 13$ & $\{121321,1213241321\}$& $\{12132,121324132\}$&  $\{1213,12132413\}$ & $\{12134,1213241\}$ & $\{121324\}$&
        $\{1213243\}$\\\hline
        $2\cdot 13$ & $\{21321,213241321\}$& $\{2132,21324132\}$&  $\{213,2132413\}$ & $\{2134,213241\}$ &
        $\{21324\}$&  $\{213243\}$\\\hline
        $13$ & $\{1321,13241321\}$& $\{132,1324132\}$& $\{13,132413\}$ & $\{134,13241\}$ & $\{1324\}$&  $\{13243\}$\\\hline
        $14$ & $\{41321,1241321\}$& $\{4132,124132\}$&  $\{413,12413\}$ & $\{14,1241\}$ & $\{124\}$&  $\{1243\}$\\\hline 
        $24$ & $\{241321\}$& $\{24132\}$&  $\{2413\}$ & $\{214\}$ & $\{24,2124\}$&  $\{243,21243\}$\\\hline
        $3\cdot 24$ & $\{3241321\}$& $\{324132\}$&  $\{32413\}$ & $\{3214\}$ & $\{324,32124\}$&
        $\{3243,321243\}$\\\hline 
        \end{tabular}
    }\vspace{1em} 
    \caption{Zero-sided cells of $\la$-value 2 in type $B_4$}
        \label{tab:0cellsB4}
    \end{table} 
 \end{example}

\subsection{Preparation for cell enumeration}
\label{sec:cellData}
Proposition \ref{prop:invariance} brings the relation $\sim$ on $\sw$ to the forefront of the enumerations of
1-cells and 2-cells in $W_2$. To be more precise, suppose $w_1,\dots,w_d$ is a complete, irredundant
list of representatives of the $\sim$-classes in $\sw$, let $n_i$ be the size of the $\sim$-class of
$w_i$, let $N_i=\abs{R_{w_i}}$ and let $I_{ij}=I(w_i,w_j)$ and $N_{ij}=N(w_i,w_j)$ for all $1\le
i,j\le d$. Then by Proposition \ref{prop:invariance} and Proposition \ref{prop:a2intersections}, we can count all cells in
$W_2$ via the following results: 

\begin{prop}
    \label{prop:Nij}
    \begin{enumerate} Maintain the notation of Theorem \ref{thm:twoSidedCells}.
           \item For all $x,y\in \sw$, if $x\not\sim_{LR}y$ then $N(x,y)=0$. 
           \item For all $x,y\in\sw$, we have $N(x,y)=N_{ij}$ where $i$ and $j$ are the unique integers such that
            $x\sim w_i$ and $y\sim w_j$.
        \item For all $x\in \sw$, we have 
            \[
                \abs{R_x}=\sum_{j=1}^d n_j N_{ij} =\abs{
                R_{w_i}}=N_i
            \]
            where $i$ is the unique integer such that $x\sim w_i$.
        \item Suppose $(W,S)$ is not of type $E_{1r}$ for any $r\ge 1$. Then the unique two-sided cell
            in $W_2$, i.e. the set $W_2$ itself, has cardinality
            \[
                \abs{W_2}=\sum_{1\le i,j\le d} n_in_jN_{ij}= \sum_{1\le i\le d} n_i N_{i}. 
            \] 
        \item Suppose $(W,S)$ is of type $E_{1r}$ for some $r\ge 1$, and let
          $1\le i,j\le d$. Then we have $N_{ij}=0$
                if $i\neq j$. Moreover, each two-sided cell $E_i$ in $W_2$ has cardinality
            \[
                \abs{E_i}=\abs{\sqcup_{y\sim w_i}R_y}=\sum_{y\sim w_i}\abs{R_y}=n_i N_i=n_i\cdot \sum_{j=1}^d
                n_jN_{ij}= n_i^2N_{ii}.
            \]
        \item We have $N_{ij}=N_{ji}$ for all $1\le i,j\le d$. 
    \end{enumerate} 
\end{prop}
\begin{proof}
Let $x,y\in \sw$. If $x\not\sim_{LR} y$, then $x\not\sim_{LR} y\inverse$ by Proposition \ref{prop:Facts}.(6);
therefore $R_x\cap L_{y\inverse}=\emptyset$ and $N(x,y)=0$. This proves Part (1).  Part (2) follows
from Proposition \ref{prop:invariance}.  To prove (3), note that by Proposition \ref{prop:a2intersections}.(1) we have
$\abs{R_x}=\sum_{y\in \mathcal {T}}N(x,y)$ where $\mathcal{T}=E\cap \sw$ and $E$ is the 2-cell
containing $x$. By Part (1), we may enlarge the set $\mathcal{T}$ to obtain $\abs{R_x}=\sum_{y\in
\sw}N(x,y)$.  Part (3) then follows from Part (2).  Part (4) follows directly from
Proposition \ref{prop:a2intersections}.(3) and Part (3). To see Part (5), first note that if $i\neq j$ then
$w_i\not\sim_{LR} w_j$ by Theorem \ref{thm:twoSidedCells}, so $N_{ij}=0$ by Part (1). This implies the last
equality in the displayed equation; the other equalities hold by the definition of $E_i$ and Part
(3).  Finally, Part (6) follows from Remark \ref{rmk:Isymmetry}.
\end{proof} 

We now present the data $w_i, n_i, N_{ij}$ defined earlier in this subsection, starting with the
description of the equivalence classes of $\sim$ in terms of the set $\{w_1:n_1, w_2:n_2,\dots, w_d:n_d\}$.

\begin{prop}
\label{prop:classes} 
Let $(W,S)$ be an irreducible \ntatf Coxeter system of type $X$, and let $\sw$ be the set of left
\twostubs in $W$, denoted in the same way as in Theorem \ref{thm:stubDescription}. Let
$\beta_n=\binom{n}{2}$ for all $n\ge 2$.  Then the equivalence classes of $\sim$ are described by
Table \ref{tab:classSizes}.

\begin{table} \centering
  \begin{tabularx}{\textwidth}{|@{\extracolsep{\fill} }c|X@{\extracolsep{\fill} }|} \hline $X$           &
      representatives and cardinalities of $\approx$-classes\\         \hline $A_n, n\ge 3$ & $\{13: \beta_n-1\}$\\
      \hline $B_3$         & $\{13:3\}$\\                                   \hline $B_n, n>3$    & $\{13:n,
      \quad 24: \beta_{n-1}-1\}$ \\         \hline $\tilde C_{n-1}, n\ge 5$ & $\{13: n-1,\; 24:
      \beta_{n-2}-1,\; (n-2)n: n-1,\; 1n: 1\}$\\\hline $E_{1,1}=D_4$ & $\{(-1)v:3, 1v: 3, \; (-1)1:3\}$\\
      \hline $E_{q,r}, r> q= 1$ & $\{(-1)v: 2+r=n-1, \; 1v: \beta_n\}$\\\hline $E_{q,r}, r\ge q\ge 2$ &
      $\{(-1)1:\beta_{q+r+3}-1=\beta_{n+1}-1\} $\\\hline $F_4$ & $\{13:9\}$\\\hline $F_n, n> 4$ &
      $\{13:3n-3,\; 35:\beta_{n-2}-1\}$\\               \hline $H_3$ & $\{13: 5\}$\\\hline $H_n, n>3$ &
      $\{13:2n-1, \; 24:\beta_{n-1}-1\}$\\\hline
  \end{tabularx} \vspace{1em} \caption{The slide equivalence classes of $\sw$, where
  $\beta_n=\binom{n
}{2}$ for each integer $n$}
  \label{tab:classSizes}
\end{table}
\end{prop}
\begin{proof}
Let $G$ be the Coxeter diagram of $(W,S)$. Recall that the number $\abs{\sw}$ is given in
Theorem \ref{thm:stubDescription}.(7), which we will use without comment from now on. 

When $(W,S)$ is of type $A_n (n\ge 3)$ or $E_{q,r} (r\ge q\ge 1)$, all edges of $G$ are simple, so
the relations $\approx$ and $\sim$ coincide. It follows from Proposition \ref{prop:desClasses} and
Theorem \ref{thm:stubDescription}.(7) that Table \ref{tab:classSizes} gives the correct information in these
types. In particular, in type $E_{1,r}$ where $r>1$, counting the set $C_1$ in
Proposition \ref{prop:desClasses}.(2).(b) yields $n_1=\abs{C_1}=r+2=n-1$, and therefore
$n_2=\abs{C_2}=\abs{\sw}-\abs{C_1}=(\beta_{n+1}-1)-(n-1)=\beta_n$.

For the Coxeter systems of types $B_n (n\ge 3), \tilde C_{n-1} (n\ge 5), F_n (n\ge 4)$ and $H_n
(n\ge 3)$, we sketch the key ideas for obtaining the set $\{w_1:n_1, \dots, w_d:n_2\}$. Let $Z$ be
the set of all short stubs of $\sw$. By Remark \ref{rmk:slidings}.(1), we may choose all the class
representatives $w_1,\dots, w_d$ from $Z$, and to understand the $\sim$-classes of $\sw$ it suffices
to understand how the short stubs in $Z$ fall into different $\sim$-classes. To this end, note that
removing the heavy edge or heavy edges in $G$ results in $r=2$ or $r=3$ linear subgraphs of $G$
whose edges are simple and whose vertex sets partition $S$: these are the sets $S_1=\{1\},
S_2=\{2,\dots, n-1\}, S_3=\{n\}$ in type $\tilde C_{n-1}$, the sets $S_1=\{1,2\}, S_2=\{3,\dots,n\}$
in type $F_n (n\ge 4)$, and the sets $S_1=\{1\}, S_2=\{2,\dots,n\}$ otherwise.  Associate a tuple
$\mathrm{dist}(z):=(\supp(z)\cap S_i)_{1\le i\le r}$ to each stub $z\in Z$ to record how its support
is distributed among these subsets, and note that by the definition of simple slides, we have $z\sim
z'$ if and only if $\mathrm{dist}(z)=\mathrm{dist}(z')$ for all $z,z'\in Z$.  Observe that in every
Coxeter type we have included in the list $w_1,\dots, w_d$ one short stub corresponding to every
possible distribution.  It follows that $w_1,\dots,w_d$ is a complete, irredundant list of class
representatives of $\sim$ in $\sw$.  

To obtain the sizes $n_1,\dots, n_d$ of the equivalence classes, start with the class represented by
the short stub whose support lies entirely in the set $S_2$ specified earlier. This stub is always
$w_2$ in our list of representatives. If $\abs{S_2}=k$ then the size of its $\sim$-class in $\sw$ is
$n_2=\beta_{k}-1$ because $S_2$ generates a parabolic subgroup of $W$ isomorphic to $A_{k}$ and
$\cals(A_k)=\beta_{k}-1$. Three $\sim$-classes remain in type $\tilde C_{n-1}$ and one class remains
otherwise. In the latter case we can obtain cardinality $n_1$ of the remaining class as
$n_1=\abs{\sw}-n_2$. In the former case, we note that the element $w_4=1n$ admits no slides and is
the only stub with first layer $1n$ by Corollary \ref{coro:firstTwoLayers}, so $\{1n\}$ is a $\sim$-class and
$n_4=1$. By the symmetry in $G$, it further follows that $n_1=n_3=(\abs{\sw}-n_2-n_4)/2=(n-1)$.
Alternatively, in both cases we may list the elements of the remaining equivalence classes and count
them without difficulty with the help of Corollary \ref{coro:firstTwoLayers} and Theorem \ref{thm:stubDescription},
but we omit the details.  
\end{proof}

Our next result describes selected 0-cells and will lead to the cardinalities $N_{ij}$ for all $1\le
i,j\le d$. We postpone its proof to \autoref{sec:rep0cells}, where we will refer to the cells given in
the theorem as \emph{representative 0-cells}.

\begin{thm}
    \label{thm:oneI}
   Let $(W,S)$ be a \ntatf Coxeter system of type $X$. Then the following set equalities hold in $W$. 
    \begin{enumerate}
        \item If $X=A_n$ where $n\ge 3$, then $I(13,13)=\{13\}$.
        \item If $X=B_3$, then $I(13,13)=\{13\}$. 
        \item If $X=B_n$ where $n>3$, then $I(24,24)=\{24,2124\}$.
        \item If $X=\tilde C_{n-1}$ where $n\ge 5$, then $I(24,24)=\{24, 2124, 2\cdot z, 212\cdot z\}$ where
             $z=45\cdots (n-1)n(n-1)\cdots54$.
         \item If $X=E_{q,r}$ where $r\ge q\ge 1$, then for all $x\in \{(-1)v,1v,(-1)1\}$ we have $I(x,x)=\{x\}$. 
        \item If $X=F_4$, then $I(24,24)=\{24\}$. 
        \item If $X=F_n$ where $n>4$, then $I(24,24)=\{24,243524\}$. 
        \item If $X=H_3$, then $I(13,13)=\{13\}$. 
        \item If $X=H_n$ where $n>3$, then $I(24,24)=\{24,2124\}$. 
    \end{enumerate}
\end{thm}

\begin{prop}
    \label{prop:intersectionSizes}
    Maintain the notation of Proposition \ref{prop:classes}. Then the numbers $N_{ij}$ where $1\le i\le j\le d$ are given by
    Table \ref{tab:intersectionSizes}. 
    \begin{table}
        \centering
        \begin{tabularx}{\textwidth}{|@{\extracolsep{\fill} }c|X@{\extracolsep{\fill} }|} \hline
            $X$                  & Values of $N_{ij}\,(1\le i<j\le d)$\\         \hline
            $A_n, n\ge 3$        & $N_{11}=1$\\                           \hline
            $B_3$                & $N_{11}=1$\\                                    \hline
            $B_n, n>3$           & $N_{11}=N_{22}=2, N_{12}=1$ \\               \hline
            $\tilde C_n, n\ge 5$ & $N_{11}=N_{22}=N_{33}=N_{44}=4,\newline N_{12}=N_{23}= N_{14}=N_{34}=2, N_{13}= N_{24}=1$\\\hline
            $E_{q,r}, r>q=1$     & $N_{ii}=1$ for all $1\le i\le d$, $N_{ij}=0$ whenever $i\neq j$\\ \hline
            $E_{q,r}, r>q\ge 2$  & $N_{11}=1$\\\hline
            $F_4$                & $N_{11}=1$\\\hline
            $F_n, n> 4$          & $N_{11}=N_{22}=2, N_{12}=1 $\\\hline
            $H_3$                & $N_{11}=1$\\\hline
            $H_n, n>3$           & $N_{ij}=2$ for all $1\le i\le j\le 2$.\\\hline
        \end{tabularx}
        \vspace{1em}
        \caption{Sizes of 0-cells of $\la$-value 2} 
        \label{tab:intersectionSizes}
    \end{table}
\end{prop}

\begin{proof}
    We sketch the proof. In Theorem \ref{thm:oneI} we have described $I_{ij}=I(w_i,w_j)$ for some
    particular $i,j$.  Taking the cardinality of this set shows that the number $N_{ij}$ given in
    Table \ref{tab:intersectionSizes} is correct. For every other pair $(k,l)$ with $1\le k\le l\le d$,
    we can obtain $I_{kl}$ from $I_{ij}$ as explained in \autoref{sec:relating0cells} and check that
    that $N_{kl}=\abs{I_{kl}}$ agrees with the value given by Table \ref{tab:intersectionSizes}.
\end{proof}

\begin{remark}
    In the above proof we mentioned computing $N_{kl}$ by computing the 0-cell $I_{kl}$ and taking
    its cardinality, but the computation of $N_{kl}$ can often be simplified. Since we are not
    interested in the set $I_{kl}=I(w_k,w_l)$ per se, we do not have to compute $I_{kl}$ from the
    set $I_{ij}$. Instead,  it is often convenient to compute a 0-cell $I'_{kl}:=I(a,b)$ from a
    0-cell $I'_{ij}:=I(c,d)$ where $a,b,c,d$ are short stubs such that $a\sim w_k, b\sim w_l, c\sim
    w_i, d\sim w_j$ and such that we can relate $c$ to $a$, and $d$ to $b$, with a small number of
    slides.  Here the set $I'_{ij}$ has $N_{ij}$ elements by Proposition \ref{prop:invariance} and $N_{ij}$ is
    at most 4 by Table \ref{tab:intersectionSizes}, so in practice it is not difficult to obtain $I_{ij}$
    by finding $N_{ij}$ elements that are obviously in it.  The point of using the stubs $a,b,c,d$
    is that if we select them carefully then we may obtain $I'_{kl}$ from $I'_{ij}$ using a smaller
    number of setwise star operations than we need for obtaining $I_{kl}$ from $I_{ij}$. For
    example, in the group $\tilde C_{n-1}$ we have $w_1=13, w_2=24, w_4=1n$ by Proposition \ref{prop:classes}
    and $N_{22}=\abs{I_{22}}=4$ by Theorem \ref{thm:oneI}.(4), and to compute $N_{14}$ it is convenient to
    use the stubs $a:=14\sim w_1, b:=1n\sim w_4, c:=24\sim w_2, d:=2(n-1)\sim w_2$ as follows: the
    set $I'_{22}=I(c,d)=I(24,2(n-1))$ should contain $N_{22}=4$ elements, so once we note that the
    set \[ I':=\{2\cdot z', 2\cdot z'\cdot n(n-1), 212\cdot z', 212\cdot z'\cdot n(n-1) \,\vert\,
    z'=45\cdots (n-1)\} \] is contained in $I'_{22}$ we can conclude that $I'_{22}=I'$. Since $a$
    can be obtained from $c$ with one slide along the edge $\{1,2\}$, applying the corresponding
    setwise left star operations to $I'_{22}=I(c,d)$ yields \[ I(a,d)=\{12\cdot z', 12\cdot z'\cdot
    n(n-1) \,\vert\, z'=45\cdots (n-1)\}.  \] Similarly, since $b$ can be obtained from $d$ with two
    slides, along the edges $\{1,2\}$ and $\{n-1,n\}$, applying the two corresponding setwise right
    star operations in succession to $I(a,d)$ yields \[ I'_{14}=I(a,b)=\{1\cdot z'', 121\cdot
    z''\,\vert\, z''=45\cdots n\}, \] so $N_{14}=2$. Note that we only need three setwise star
    operations, even if $n$ is large.  $n$. In contrast, if we computed $I_{14}$ from $I_{22}$ by
    working with $w_1=13,w_2=24$ and $w_4=(n-1)n$ directly, then we need many more setwise star
    operations to obtain $I_{14}$ when $n$ is large.
    \label{rmk:shortcut}
\end{remark}   

\subsection{Sizes of \texorpdfstring{$\la(2)$}{a(2)}-cells}
\label{sec:cellSizes}
We are ready to compute the sizes of all left, right, and two-sided Kazhdan--Lusztig cells of
$\la$-value 2 in all $\la(2)$-finite Coxeter groups:

\begin{thm}
    \label{thm:cellSizes}
    Let $(W,S)$ be an irreducible \ntatf Coxeter system of type $X$ and maintain the notation of
    \autoref{sec:cellData}. In particular, let $\beta_n =\binom{n}{2}$ for all $n\ge 2$.  Then the
    sizes of 1-cells and 2-cells in $W_2$ are given by Table \ref{tab:cellSizes}. When $W_2$ is a single
    two-sided cell of size $N$, we simply write $N$ instead of $\abs{W_2}=N$ in the corresponding
    box in the table.

    \begin{table} \centering
        \begin{tabularx}{\textwidth}{|c|X@{\hspace{-1em}}|X@{\extracolsep{\fill} }|} \hline $X$           & sizes
              of 1-cells & sizes of 2-cells \\      \hline $A_n, n\ge 3$ & $N_{1}=\beta_n-1$                 &
              $(\beta_n-1)^2$   \\ \hline $B_3$         & $N_1=3$ & $9$ \\\hline $B_n, n>3$    & $N_1=\beta_{n+1},
              N_2= n^2-2n$    & ${n^4}/{2}-2n^3+{7n^2}/{2}$ \\\hline $\tilde C_{n-1},n \ge 5$  & $N_1=N_3=n^2+1,
              \newline N_2=2n^2-6n+5,\newline N_4=n^2/2+{3}n/2 +2$ & $n^4-6n^3+20n^2-21n+10$ \\\hline $E_{11}=D_4$
              & $N_1=N_2=N_3=3$ & $\abs{C_1}=\abs{C_2}=\abs{C_3}=9$ \\\hline $E_{q,r}, r>q=1$        &
              $N_1=n-1,\newline N_2=\beta_{n}$ & $\abs{C_1}=(n-1)^2,\newline \abs{C_2}=\beta_n^2$ \\\hline
              $E_{q,r}, r\ge q\ge 2$        & $N_1=\beta_{n+1}-1$ & $N_{1}=(\beta_{n+1}-1)^2$ \\\hline $F_4$ &
              $N_1=9$ & $81$\\\hline $F_n, n> 4$        & $N_1=n^2/2+{7}n/2-4,\newline N_2=(n-1)^2$ &
              $n^4/2-2n^3+{33}n^2/2-29n+14$ \\\hline $H_3$ & $N_1=5$ & $25$\\\hline $H_n,n>3$ &
              $N_1=N_2=2(\beta_{n+1}-1)$&$2(\beta_{n+1}-1)^2$\\\hline
          \end{tabularx} \vspace{1em} \caption{Sizes of 1-cells and 2-cells of $\la$-value 2, where $\beta_n=\binom{n}{2}$ for each integer
        $n$}
        \label{tab:cellSizes}
    \end{table}
\end{thm}

\begin{proof}
  We know the numbers $n_i$ and $N_{ij}$ for all $1\le i,j\le d$ by Propositions \ref{prop:classes},
  \ref{prop:intersectionSizes} and \ref{prop:Nij}.(6), so the sizes of all 1-cells and 2-cells can
  be calculated by directly applying the formulas from Proposition \ref{prop:Nij}.(3)--(5).  The calculations
  are straightforward and similar for all $\la(2)$-finite Coxeter types, so instead of carrrying out
  the computations in detail for all types, we only illustrate below how to count the cells in
  $\tilde C_{n-1} (n\ge 5)$, where the number $d$ of $\sim$-classes in $\sw$ is the largest and the
  computations are the most complex.  Recall from Proposition \ref{prop:classes} that $\sw$ contains $d=4$
  slide equivalence classes and their sizes are $n_1=n-1, n_2=\beta_{n-2}-1,n_3=n-1$ and $n_4=1$.
  Consider the following $d\times d$ table, where the rows and columns are labeled by $n_i$ and the
  entry in the $i$-th row, $j$-th column is the number $N_{ij}$ from Table \ref{tab:intersectionSizes}. 

    \begin{table} \centering
        \begin{tabular}{|c|c|c|c|c|} \hline 
            &$n-1$ & $\beta_{n-2}-1$ & $n-1$ & $1$\\\hline
           $n-1$ & 4 & 2  & 1 & 2\\\hline
           $\beta_{n-2}-1$ & 2 & 4  & 2 & 1\\\hline
           $n-1$ & 1 & 2  & 4 & 2\\\hline
           $1$ & 2 &  1&2   & 4\\\hline
       \end{tabular} \vspace{1em} 
        \caption{}
    \end{table}
    By Proposition \ref{prop:Nij}.(2), for each $1\le i\le d$ we can compute $N_i$ as a weighted sum of the $i$-th row in the
    above table, with the column labels as weights. For example, we have
    \[
        N_1 =\sum_{j=1}^d n_jN_{ij}= 4\cdot (n-1) + 2\cdot (\beta_{n-2}-1) + 1 \cdot (n-1) + 2 \cdot 1 = n^2+1.
    \]
    Similar computations show that $N_2= 2n^2-6n+5, N_3=n^2+1$ and $N_4= n^2/2+3n/2+2$. 
    Having computed $N_1,\dots, N_d$, we can then compute the size of the 2-cell $W_2$ as the
    weighted sum $\sum_{i=1}^d n_i N_i$ by Proposition \ref{prop:Nij}.(4). Doing so yields $ \abs{W_2}= n^4-6n^3+20n^2-21n+10$.
\end{proof}

\section{Representative 0-cells}
\label{sec:rep0cells}

We assume $(W,S)$ is a \ntatf Coxeter system and prove Theorem \ref{thm:oneI} in this section. We complete
the proof in \autoref{sec:proofs} after preparing a series of technical definitions and lemmas in
\autoref{sec:heapLemmas}. 

\subsection{Lemmas on heaps}
\label{sec:heapLemmas}
Let $G$ denote the Coxeter diagram of $(W,S)$. Then $G$ has no cycles; therefore for any two
vertices $u,v$ in $G$ there is a unique path from $u$ to $v$ (see Remark \ref{rmk:chainsAndWalks}). For
pairwise distinct vertices $u,v,t$, we say that $u$ and $v$ \emph{lie on different sides of} a
vertex $t$ if the path connecting $u$ to $v$ passes through $t$. We call two vertices
\emph{neighbors} of each other if they are adjacent in $G$, define the \emph{degree} of a vertex $v$
in $G$ to be the number of its neighbors, and call a vertex an \emph{end vertex} if it has degree 1.
For each element $w\in \fc(W)$ and any two elements $i,j$ in the heap $H(w)$, we denote the label of
$i$ by $e(i)$, write $(i,j):=\{h\in H(w):i\precneqq h\precneqq j\}$, and write $[i,j]:=\{h\in
H(w):i\preceq h\preceq j\}$.

The following two lemmas allow us to deduce the existence of additional elements and structure in
$H(w)$ from certain local configurations.  The vertex whose existence is asserted in the lemmas
allows us to expand the lattice embedding of the heap vertically and horizontally, which is why we
name the lemmas vertical and horizontal completion.  

\begin{lemma}[Vertical Completion] 
    \label{lemm:vertical}
    Let $a$ be an end vertex in $G$ and let $b$ be its only neighbor. Suppose $b\in \call(w)$ and
    $e(i)=a$ for an element $i\in H(w)$. Then $H(w)$ contains an element that is covered by $i$ and
    has label $b$.
\end{lemma}
\begin{proof} 
    Since $b\in \call(w)$, some minimal element $j\in H(w)$ has label $b$ by
    Remark \ref{rmk:heapsAndWords}. The elements $i$ and $j$ are comparable since $a$ and $b$ are adjacent
    in $G$. Moreover, we must have $j\le i$ since $j$ is minimal, so there is a chain $j=j_1\prec
    \dots j_{p-1}\prec j_p=i$ of coverings connecting $j$ to $i$ in $H(w)$. Since $b$ is the only
    neighbor of $a$, we must have $e(j_{p-1})=b$. It follows that $j_{p-1}$ has the desired
    properties.
\end{proof}

\begin{lemma}[Horizontal Completion] 
    \label{lemm:horizontal}
    Let $a$ be an end vertex in $G$ and let $b$ be its only neighbor. Suppose that $C=(j\prec i\prec j')$ is
    a chain of coverings in $H(w)$ with $e(i)=a$ and $e(j)=e(j')=b$. Then $i,j,j'$ account for all
    elements in $[j,j']$ that are labeled by $a$ or $b$. Moreover, if $C$ is not convex, then  $j$
    and $j'$ are connected by a chain $D=(j=j_1\prec j_2\prec \dots \prec j_q=j')$ of coverings
    where $q\ge 3$. In this case, at least one element in $D$ is labeled by a neighbor of $b$ in $G$
    that is distinct from $a$.
\end{lemma}

\begin{proof}
    An element in $(j,j')$ with label $b$ would be comparable to $i$ and thus contradict the fact
    that $i$ covers $j$ and $j'$ covers $i$, so $j$ and $j'$ are the only elements with label $b$ in
    $[j,j']$. Since $b$ is the only neighbor of $a$, it further follows that $(j,j')$ contains no
    element with label $a$ other than $i$, because otherwise such an element would cover or be
    covered by $i$, violating Condition (1) of Proposition \ref{prop:FCCriterion}. We have proved the first
    claim.

    If $C$ is not convex, then $j$ and $j'$ must both be comparable to some element $k\in (j,j')$
    that is not in $C$.  Concatenating chains of coverings $j\prec \dots \prec k$ and $k\prec \dots
    \prec j'$ connecting $j$ to $k$ and $k$ to $j'$ yields a chain $D$ in $[j,j']$ that connects $j$
    to $j'$ and has at least three elements, as desired. The chain $D$ corresponds to a walk of
    length at least 3 from $b$ to $b$ in $G$ by Remark \ref{rmk:chainsAndWalks}. Such a walk must involve
    a neighbor of $b$, and such a neighbor cannot be $a$ by the first claim since $D$ lies in the
    interval $[j,j']$. The proof is complete.
\end{proof}

The next result proves that certain subsets of heaps are antichains.

\begin{lemma}
    Let $x,y,z$ be elements of $S$ such that $x$ and $z$ lie on different sides of $y$.  Suppose
    that $H(w)$ contains two elements $i$ and $k$ labeled by $x$ and $z$, respectively. Then the set
    $\{i,k\}$ is an antichain in $H(w)$ whenever one of the following conditions holds.
    \begin{enumerate}
        \item The heap $H(w)$ contains two elements $j,j'$ labeled by $y$ such that $j\prec i\prec
             j'$ and $k\in (j,j')$. 
        \item The heap $H(w)$ contains an element $j$ labeled by $y$ and the sets $\{i,j\}$ and
             $\{j,k\}$ are antichains in $H(w)$. 
    \end{enumerate}
    \label{lemm:3chain}
\end{lemma}

\begin{proof}
    We prove the lemma by contradiction.  Suppose $i$ and $k$ are comparable and assume, without
    loss of generality, that $i\preceq k$ in $H(w)$. Then by Remark \ref{rmk:chainsAndWalks}, there is a
    chain of coverings $i=i_1\prec i_2\prec \dots\prec i_l=k$ in $H(w)$ whose elements' labels form
    a walk on $G$ from $x$ to $z$. The walk must pass through $y$, so $e(i_p)=y$ for some $1\le p\le
    l$. Elements sharing a label are comparable in a heap, so $i_p$ is comparable to $j$ and $j'$
    under Condition (1) and is comparable to $j$ under Condition (2).

    Suppose Condition (1) holds. Then the interval $(j,j')$ contains no element with label $y$ by
    Lemma \ref{lemm:horizontal}, so we have either $i_p\preceq j$ or $j'\preceq i_p$ in $H(w)$. But in
    these cases we would have $i\preceq i_p\preceq j$ or $j'\preceq i_p\preceq k$, respectively,
    contradicting the fact that $k$ and $i$ lie in $(j,j')$.
    
    Under Condition (2), we have either $i_p\prec j$ or $j\prec i_p$, which would imply that
    $i\preceq i_p\preceq j$ or $j\preceq i_p\preceq k$ and violate the assumption that $\{i,j\}$ and
    $\{j,k\}$ are antichains.
\end{proof}

Lemma \ref{lemm:3chain} will be used in tandem with Lemma \ref{lemm:horizontal} in \autoref{sec:proofs}. It also
helps establish our next proposition, Proposition \ref{prop:trapped}, which concerns the following notions:

\begin{definition}
    We define a \emph{special quadruple in $S$} to be an ordered tuple $(a,b,c,d)$ of four distinct
    elements in $S$ with the following properties. 
    \begin{enumerate}
        \item We have $m(a,c)=m(a,d)=m(b,d)=2$, $m(a,b)\ge 3, m(b,c)\ge 3$ and $m(c,d)=3$. In
             particular, the generators $a,b,c,d$ induce a subgraph of the form $a$--$b$--$c$--$d$
             in $G$ if we ignore edge weights. 
        \item The vertices $a$ and $c$ have degrees 1 and 2 in $G$, respectively. In other words,
             the vertex $b$ is the only neighbor of $a$, and the vertices $b,d$ are the only
             neighbors of $c$.
    \end{enumerate}
    \label{def:quadruple}
\end{definition}

\begin{definition}
    Let $(a,b,c,d)$ be a special quadruple in $S$. We call an antichain $\{i,j\}$ of size 2 in $H(w)$ a
    \emph{trapped antichain relative to $(a,b,c,d)$} if the following conditions hold.
    \begin{enumerate}
        \item Neither $i$ nor $j$ is maximal or minimal in $H(w)$.
        \item The label of $i$ is $a$ and the label of $j$ is $c$. 
    \end{enumerate}
    \label{def:trapped}
\end{definition}

\begin{example}
    In each Coxeter system $(W,S)$ of type $\tilde C_{n-1}$ where $n\ge 5$, the tuple $(1,2,3,4)$ is
    a special quadruple.  In the heap of the FC element $w=2413524\in W$, the elements with labels 1
    and 3 form a trapped antichain relative to $(1,2,3,4)$. Note that $n(w)=3$, hence
    $\la(w)=n(w)>2$ by Proposition \ref{prop:a=n}. The next proposition shows that the fact that $\la(w)\neq 2$
    is to be expected.
\end{example}

\begin{prop}
    Let $(W,S)$ be a \ntatf Coxeter system, let $w\in W_2$, and suppose
    $\abs{\call(w)}=\abs{\calr(w)}=2$. Then the heap $H(w)$ cannot contain any trapped antichain
    relative to a special quadruple in $S$.
    \label{prop:trapped}
\end{prop}

\begin{proof}
Suppose that $H(w)$ contains a trapped antichain $A=\{i,j\}$ relative to a special quadruple
$(a,b,c,d)$ in $S$, with $e(i)=a$ and $e(j)=c$. We will derive a contradiction.  To start, note that
$A$ is a maximal antichain in $H(w)$ by Corollary \ref{coro:layerSizeBound}. By Lemma \ref{lemm:maxAntichain} and
Definition \ref{def:anatomy}.(2), it follows that $w$ can be written as a product $w=x*y$ such that $H(x)$ and
$H(y)$ coincide with the ideal $\mathcal{I}_A$ and filter $\mathcal{F}_A$ when viewed naturally as
subsets of $H(w)$.  Lemma \ref{lemm:decompAndDescent} then implies that $\call(x)=\call(w)$, so $x$ has
two left descents. On the other hand, Condition (1) of Definition \ref{def:trapped} implies that
$\mathcal{I}_A$ properly contains $A$, so $x\neq ac$ and $l(x)>2$.  It follows that $x\notin \sw$,
because no stub in $\sw$ with length at least 3 has two left descents by
Lemma \ref{lemm:stubRemarks}.(1).

Since $i$ is not minimal, it must cover an element $g\in H(w)$. The label of $g$ must be a neighbor
of $e(i)=a$, which has to be $b$ by Condition (2) of Definition \ref{def:quadruple}. Since $b$ is also
adjacent to $c$ in $G$, the element $g$ is comparable to $j$ as well. If $j\preceq g$ then $j\preceq
g\prec i$ and $A$ would not be antichain, so $g\preceq j$. It follows that neither the set
$H(x)\setminus\{i\}$ nor the set $H(x)\setminus\{j\}$ contains a maximal element labeled by $b$;
therefore $x$ admits no right lower star operation with respect to $\{a,b\}$ or $\{c,b\}$ by
Proposition \ref{prop:Reducibility}. On the other hand, since $x$ is not a stub it must admit a right lower
star operation with respect to some pair $\{s,t\}$ of noncommutating generators, and one of $s,t$
has to come from the set $\calr(x)=\{a,c\}$.  Condition (2) of Definition \ref{def:quadruple} now forces
$\{s,t\}=\{c,d\}$. It follows that $j$ covers an element $h\in H(x)\se H(w)$ labeled by $d$.  A
similar argument shows that $H(w)$ contains an element $h'$ that lies in $\mathcal{F}_A$, covers
$j$, and has label $d$.

Since $m(c,d)=3$ by Definition \ref{def:quadruple}, the chain $h\preceq j\preceq h'$ cannot be convex by
Proposition \ref{prop:FCCriterion}, so Lemma \ref{lemm:horizontal} implies that $h$ and $h'$ are connected by a
chain of coverings containing an element $k$ labeled by some neighbor $e$ of $d$ distinct from $c$.
The vertices $a,b,c,d,e$ induce a subgraph of $G$ of the form $a$--$b$--$c$--$d$--$e$, so we may now
apply Lemma \ref{lemm:3chain} as follows: first, since $c,e$ lie on different sides of $d$, Part (1) of
the lemma implies that $\{j,k\}$ forms an antichain; next, since $a,e$ lie on different sides of $c$
and $\{i,j\},\{j,k\}$ are both antichains, Part (2) of the lemma implies that $\{i,k\}$ forms an
antichain.  It follows that $\{i,j,k\}$ is an antichain in $H(w)$. Proposition \ref{prop:a=n} then implies that
$\la(w)\ge n(w)\ge 3$, contradicting the assumption that $\la(w)=2$.
\end{proof}

\subsection{Proof of \texorpdfstring{Theorem \ref{thm:oneI}}{Theorem 4.17}}
\label{sec:proofs}
Let $(W,S)$ be a \ntatf Coxeter system, and recall that each part of Theorem \ref{thm:oneI} claims that if
$(W,S)$ has a certain type $X$, then a representative 0-cell $I(x,y):=I(x,y)\se W$ where $x,y\in
\sw$ equals a certain set. Denoting this set by $K(x,y)$, we first note that $K(x,y)\se I(x,y)$
regardless of what $X$ is:

\begin{prop}
We have $K(x,y)\se I(x,y)$ for every Coxeter system $(W,S)$ in Theorem \ref{thm:oneI}.  
\end{prop}
\begin{proof} Let $w\in K(x,y)$. Then $w$ has reduced factorizations $w=x\cdot w'$ and $w=w''\cdot
      y\inverse$ by inspection, so we can conclude that $w\in I(x,y)$ by Theorem \ref{thm:CellViaDecomp}
      once we can show $\la(w)=2$.  If $w$ is a product $st$ of two commuting generators $s,t\in S$,
      then $\la(w)=2$ by Proposition \ref{prop:Facts}.(2). Otherwise, we observe that $w$ can be obtained from
      such a product $w=24$ (i.e. the product $w=s_2s_4$) via a sequence of right upper star
      operations; therefore $\la(w)=\la(24)=2$ by Corollary \ref{coro:starA}.  Specifically, the element
      $w=243524$ in Part (7) can be built from $24$ by applying right upper star operations with
      respect to $\{4,5\}, \{2,3\}, \{2,3\}$ and $\{4,5\}$ successively, and a suitable sequence of
      star operations taking $24$ to $w$ is straightforward to find in all other cases. 
\end{proof}

To prove Theorem \ref{thm:oneI} it remains to prove that $I(x,y)\se K(x,y)$ for each group $W$ in the
theorem. We will do so by treating one Coxeter type at a time below, and the following notions will
be used frequently. For every element $w\in \fc(W)$ and any generator $p\in S$, we denote the number
of elements labeled by $p$ in the heap by $o_w(p)$, where we drop the subscript $w$ if there is no
danger of confusion. We define a \emph{$p$-interval} in $H(w)$ to be an interval of the form
$[j,j']\se H(w)$ where $j,j'$ are consecutive elements labeled by $p$ in the sense that $j\preceq
j'$, $j\neq j'$, and $j,j'$ are the only elements with label $p$ in the interval $[j,j']$.  Finally,
when we wish to emphasize the ambient group $W$ that a 0-cell $I(x,y)$ lies in, we will write
$I_{W}(x,y)$ for $I(x,y)$. 

The following facts will also be used frequently without further comment. First, if $x,y$ are short
stubs, then every element $w\in I(x,y)$ has left descent set $\call(w)=\call(x)=\supp(x)$ and right
descent set $\calr(w)=\calr(y\inverse)=\calr(y)=\supp(y)$ by Lemma \ref{lemm:stubRemarks}.(1).  Second,
recall from Remark \ref{rmk:heapsAndWords} that the left and right descents of an FC elements are
precisely the labels of the minimal and maximal elements of its heap, respectively.  Third, consider
the situation where an $\la(2)$-finite Coxeter system $(W,S)$ is isomorphic to a parabolic subgroup
$W'_J$ of another $\la(2)$-finite Coxeter system $(W',S')$ for some subset $J$ of $S'$ via a group
isomorphism $\phi:W\ra W'_J$ induced by a bijection $f: S\ra J$ with the property that
$m(f(s),f(t))=m(s,t)$ for all $s,t\in S$.  (By ``induced'' we mean the restriction of $\phi$ to $S$
coincides with $f$.) Then the next lemma holds for any two stubs $x,y\in \sw$, where
$f(Z)=\{f(z):z\in Z\}$ for every subset $Z$ of a Coxeter group and any map $f$ whose domain contains
$Z$. The lemma allows us to deduce 0-cells of $\la$-value 2 in $W$ from such cells in $W'$ by
``parabolic restriction''.

\begin{lemma}
    \label{lemm:parabolic}
    In the above setting, let $x'=\phi(x),y'=\phi(y)$, denote the 0-cell $I_W(x,y)$ in $W$ by $I_1$,
    denote the 0-cell $I_{W'_J}(x',y')$ in $W'_J$ by $I_{2}$, and denote the 0-cell $I_{W'}(x',y')$
    in $W'$ by $I_3$.  Then we have $I_2=\phi(I_1)$ and $I_2=\{z\in I_3: \supp(z)\se J\}$. In
    particular, we have 
    \begin{equation*} I_W(x,y)=\{\phi\inverse(z):z\in I_{W'}(x',y'), \supp(z)\se J\}.
      \label{eq:parabolic}
    \end{equation*}
\end{lemma}

\begin{proof}
  We first sketch why $I_2=\phi(I_1)$: the fact that $\phi$ is induced by the bijection $f: S\ra J$
  implies that $\phi$ is in fact an {isomorphism of Coxeter systems}, so that $\phi$ in turn induces
  an isomorphism from the Hecke algebra of $(W,S)$ to the Hecke algebra of $(W'_J,J)$ that sends the
  Kazhdan--Lusztig basis element $C_{w}$ to $C_{\phi(w)}$ for all $w\in W$. It follows that
  $I_2=\phi(I_1)$. 

  Next, let $\la_{S'}$ and $\la_{J}$ denote the $\la$-functions associated to the Coxeter system $(W',S')$ and
  $(W'_J, J)$, respectively. By Proposition \ref{prop:Facts}.(7), the system $(W'_J,J)$ is $\la(2)$-finite since $(W',S')$ is,
  so Theorem \ref{thm:CellViaDecomp} and its analog for left cells imply that
  \[
    I_2=\{z\in W'_J: \la_J(z)=2, x\le^R z, y\le^L z\}  
  \]
  and
  \[
    I_3=\{z\in W': \la_{S'}(z)=2, x\le^R z, y\le^L z\}.
  \]
  An element $z\in W'$ is in the parabolic subgroup $z\in W'_J$ if and only if $\supp(w)\se J$, and for such an
  element we have $\la_{S'}(z)=\la_J(z)$ by Proposition \ref{prop:Facts}.(8). It follows that $I_2=\{z\in I_3:\supp(z)\se
  J\}$. 
\end{proof}

\subsubsection{Type $\tilde C_{n-1}$}
Suppose $X=\tilde C_{n-1} (n\ge 5)$. We complete the proof of Theorem \ref{thm:oneI}.(4) by proving the
following:
\begin{prop}    \label{prop:affineC} If $X=\tilde C_{n-1} (n\ge 5)$, then we have \[ I(24,24)\se
      \{24,2124,2\cdot z,212\cdot z\}\] where $z=45\dots (n-1)n(n-1)\dots 54$.
\end{prop} To prove the proposition, let \[ z_p=p(p-1)\dots 212\dots (p-1)p, \] \[ z_p'=p(p+1)\dots
    (n-1)n(n-1)\dots (p+1)p \] for all $1<p<n$. For example, the elements $212$ and $z$ from the
    proposition equal $z_2$ and $z_4'$, respectively.  We will take advantage of the following
    result of Ernst:

\begin{lemma}[\cite{ErnstTL}]
    \label{lemm:Dana}
    Let $n\ge 4$ and let $w\in \fc(\tilde C_{n-1})$. Let $1<p<n$ and suppose $[j,j']$ is a $p$-interval in $H(w)$
    for two elements $j\preceq j'$ in $H(w)$.
\begin{enumerate}
    \item If the interval $[j,j']$ contains no element labeled by $(p+1)$, then
    either $\la(w)=1$ or $w$ has a reduced factorization $w=u\cdot z_p\cdot u'$
  where $\supp(u)\cup\supp(u')\se\{p+1,p+2\dots, n\}$.  \item If the interval
    $[j,j']$ contains no element labeled by $(p-1)$, then either $\la(w)=1$ or
    $w$ has a reduced factorization $w=u\cdot z'_p\cdot u'$ where
    $\supp(u)\cup\supp(u')\se\{1,2\dots, p-1\}$.
\end{enumerate} \end{lemma} 
\begin{proof} Part (1) is a restatement of Lemma 3.6 of \cite{ErnstTL}. (We note that each element ``of
        Type I'' in \cite{ErnstTL} has a unique reduced word by \cite[\S3.1]{ErnstTL} and thus has $\la$-value 1 by
Proposition \ref{prop:facts01}.(2).) Part (2) follows from Part (1) by the symmetry in the Coxeter diagram of
      $\tilde C_{n-1}$.
\end{proof}

\begin{proof}[Proof of Proposition \ref{prop:affineC}]
    Let $w\in I(24,24)$.  Then $2$ labels both a maximal and a minimal element in $H(w)$, as does 4.
    It follows that $o(2),o(4)\ge 1$, and that if $3\in \supp(w)$ then every element labeled by 3 in
    $H(w)$ has to lie in a 2-interval.  Note that every 2-interval in $H(w)$ must contain an element
    labeled by $1$: otherwise, Lemma \ref{lemm:Dana}.(2) implies that $w=u\cdot z'_2\cdot u'$ where $4$
    labels no maximal element in $z_2'$ and does not appear in $u$ or $u'$, so $4$ labels no maximal
    element in $H(w)$, a contradiction.  On the other hand, since $(1,2,3,4)$ is a special quadruple
    in $S$, no 2-interval in $H(w)$ can contain both an element labeled by 1 and an element labeled
    by 3 by Proposition \ref{prop:trapped}. It follows that $3\notin\supp(w)$.  For $i\in \{2,4\}$, let $a_i$
    and $b_i$ stand for the minimal and maximal element in $H(w)$ labeled by $i$, respectively.
    Since $3\notin\supp(w)$, Remark \ref{rmk:chainsAndWalks} implies that no two elements in $H(w)$
    labeled by two generators in $S$ that lie on different sides of 3 are comparable. Since every
    element in $H(w)$ must be comparable to a maximal and to a minimal element, it follows that
    $H(w)$ equals the disjoint union of the two convex intervals $I_2:=[a_2,b_2]$ and
    $I_4:=[a_4,b_4]$ as sets, where each element in $I_2$ and $I_4$ carries a label smaller and
    larger than 3, respectively. It follows that $w=x\cdot y$ for FC elements $x,y$ such that
    $H(x)=I_1$ and $H(y)=I_2$. Applying Lemma \ref{lemm:Dana}.(1) with $p=2$ and Lemma \ref{lemm:Dana}.(2)
    with $p=4$ reveals that $x\in \{2,212\}$ and $y\in \{4,z'_4=z\}$, respectively, so we have $w\in
    \{24,2124, 2z, 212z\}$. We conclude that $I(24,24)\se \{24,2124, 2z, 212z\}$.  
\end{proof}

\subsubsection{Type $B_n$}
We complete the proof of Theorem \ref{thm:oneI}.(2)--(3) by deducing the following proposition from
Proposition \ref{prop:affineC}. 

\begin{prop}
    \label{prop:B}
    If $X=B_n (n\ge 3)$, then $I(13,13)\se \{13\}$ if $n=3$ and $I(24,24)\se \{24,2124\}$ if $n>3$.
\end{prop}
\begin{proof} 
    Let $k=n+2\ge 5$, consider the Coxeter system of type $(W',S')$ of type $\tilde C_{k-1}$, and
    let $J=\{1,2,\dots,n\}\se S'$. The group $B_n$ is naturally isomorphic to parabolic subgroup
    $W'_J$ of $W'$ via the isomorphism $\phi: W\ra W'_J$ with $\phi(i)=i$ for each generator $i\in
    S$. Thus, Lemma \ref{lemm:parabolic} implies that if $n>3$ then $I(24,24)=I_{W}(24,24)\se
    \{\phi\inverse(w'):w'\in I_{W'}(24,24),\supp(w')\se J\}$, where the last set equals
    $\{\phi\inverse(24),\phi\inverse(2124)\}=\{24,2124\}$. It follows that $I(24,24)\se
    \{24,2124\}$. Similarly, if $n=3$, then using the methods from Remark \ref{rmk:relating0cells} we can
    see that $I_{W'}(13,13)=\{13,134543,132413, 13245413\}$, whence Lemma \ref{lemm:parabolic} implies
    that $I(13,13)\se \{13\}$ since $4\notin J$.
\end{proof} 

\subsubsection{Type $A_n$}
We complete the proof of Theorem \ref{thm:oneI}.(1) by deducing the following proposition from Proposition \ref{prop:B}.
\begin{prop}
    \label{prop:A} If $X=A_n (n\ge 3)$, then $I(13,13)\se \{13\}$.
\end{prop}
\begin{proof} 
    We use parabolic restriction as in the proof of Proposition \ref{prop:B}. More precisely, let $(W,S)$ and
    $(W',S')$ be the Coxeter systems of type $A_n$ and $B_{n+1}$, respectively, let $J=\{2,3,\dots,
    n+1\}$, and consider the isomorphism $\phi:W\ra W'_{J}$ with $\phi(i)=i+1$ for all $i\in S$.  We
    have $I_{W'}(24,24)\se\{24,2124\}$ by Proposition \ref{prop:B} and $I(13,13)=I_{W}(13,13)\se
    \{\phi\inverse(z):z\in I_{W'}(24,24), \supp(z)\in J\}$ by  Lemma \ref{lemm:parabolic}; therefore
    $I(13,13)\se\{\phi\inverse(24)\}=\{13\}$.
\end{proof}

\subsubsection{Type $E_{q,r}$}
We complete the proof of Theorem \ref{thm:oneI}.(5) by deducing the following result from Theorem \ref{thm:oneI}.(1):
\begin{prop}
    \label{prop:Eqr} If $X=E_{q,r}$ where $r\ge q\ge 1$, then $I(x,x)\se \{x\}$ for all $x\in \{(-1)v,1v,(-1)1\}$.
\end{prop}
\begin{proof} 
    Let $x\in \{(-1)v, 1v,(-1)1\}$. To prove the proposition we may assume that $r\ge q> 1$ in the
    Coxeter system $W$ of type $E_{q,r}$: any system $(W',S')$ of type $E_{1,r}$ lies in $W$ as a
    parabolic subgroup, and Lemma \ref{lemm:parabolic} implies that if $I_{W}(x,x)\se\{x\}$ then
    $I_{W'}(x,x)\se \{x\}$. Under the assumption that $r\ge q> 1$, it further suffices to show that
    $I((-1)1,(-1)1)\se\{(-1)1\}$: since $\sw$ contains a single $\sim$-class by
    Proposition \ref{prop:desClasses}.(1), if $I((-1)1,(-1)1)\se\{(-1)1\}$ then  all 0-cells have size 1 in
    $W_2$ by Proposition \ref{prop:invariance}, which in turn implies that the 0-cell $I(y,y)$ has to contain
    only the element $y$ that is obviously in it for all short stubs $y\in \sw$.

  Let $x=(-1)1$ and let $w\in I(x,x)$. Suppose $v\in \supp(w)$, so that some element  $i$ in the heap
  $H(w)$ has label $v$. Since $\calr(w)=\call(w)=\{-1,1\}$, the element $i$ is neither minimal nor
  maximal in $H(w)$, so it must cover some element $j$ and be covered by some element $j'$ in $H(w)$
  by Lemma \ref{lemm:vertical}. The labels of $j$ and $j'$ must both be $0$, the only neighbor of $v$ in
  $S$. Since $m(v,0)=3$, the chain $j\prec i\prec j'$ cannot be convex by Proposition \ref{prop:FCCriterion},
  so Lemma \ref{lemm:horizontal} implies that $j$ and $j'$ are connected by a chain in $H(w)$ that
  contains an element $k$ labeled by either $-1$ or $1$.  Lemma \ref{lemm:3chain} then implies that the
  set $\{i,k\}$ forms a trapped antichain relative to either the special quadruple $(v,0,-1,-2)$ or
  the special quadruple $(v,0,1,2)$.  This contradicts Proposition \ref{prop:trapped}, so $v\notin\supp(w)$.
  But then $w$ lies in the parabolic subgroup of $W$ of type $A_{q+r+1}$ generated by the set
  $\{-q,\cdots,-1,0,1,\cdots,r\}$. It follows from Theorem \ref{thm:oneI}.(1), Lemma \ref{lemm:parabolic} and
  Proposition \ref{prop:invariance} that $I(x,x)$ is the singleton $\{x\}$. The proof is complete. 
\end{proof}

\subsubsection{Type $F_n$}
We complete the proof of Theorem \ref{thm:oneI}.(6)--(7) by deducing the following proposition from Proposition \ref{prop:B}. 

\begin{prop}
    \label{prop:F}
    If $X=F_n (n\ge 3)$, then $I(24,24)\se \{24\}$ if $n=4$ and $I(24,24)\se \{24,243524\}$ if $n>4$.
\end{prop}

\begin{proof}
  Consider the parabolic group $W_J$ of $W$ generated by the set $J=\{2,3,\cdots, n\}$. It is a
  Coxeter group of type $B_{n-1}$, so $I_{W_J}(24,24)=\{24\}$ if $n-1=3$ and $I_{W_J}(35,35)=\{35,
  3235\}$ if $n-1>3$ by Theorem \ref{thm:oneI}.(2)--(3). In the latter case, computing the 0-cell
  $I_{W_J}(24,24)$ from $I_{W_J}(35,35)$ using the methods of Remark \ref{rmk:relating0cells} gives
  $I_{W_J}(24,24)=\{24, 243524\}$. 

  Let $w\in I(24,24)$. We claim that $1\notin \supp(w)$, so that $w$ lies in the parabolic subgroup
  of type $B_{n-1}$ generated by the set $J=\{2,3,\dots,n\}$. It follows that
  $I(24,24)=I_{W}(24,24)\se I_{W_J}(24,24)$, so $I(24,24)\se \{24\}$ if $n=4$ and $I(24,24)\se
  \{24,243524\}$ if $n>4$ by the last paragraph. To prove the claim, suppose that $1\in \supp(w)$,
  so that some element  $i$ in the heap $H(w)$ has label $1$.  This leads to a contradiction to
  Proposition \ref{prop:trapped} in the same way the element $i$ does in the proof of Proposition \ref{prop:Eqr}: Lemmas
  \ref{lemm:vertical}, \ref{lemm:horizontal} and \ref{lemm:3chain} force $i$ to be part of an
  antichain $\{i,k\}$ contained in an $2$-interval $[j,j']\se H(w)$ where $j\prec i\prec j'$ and
  $e(k)=3$, and this antichain is a trapped antichain relative to the special quadruple $(1,2,3,4)$.
  It follows that $1\notin \supp(w)$, as claimed.
\end{proof}

\subsubsection{Type $H_n$}
We complete the proof of Theorem \ref{thm:oneI}.(8)--(9) by deducing the following proposition from Proposition \ref{prop:B}. 

\begin{prop}
    \label{prop:H}
    If $X=H_n (n\ge 3)$, then $I(13,13)\se \{13\}$ if $n=3$ and $I(24,24)\se
    \{24,2124\}$ if $n>3$.
\end{prop}

\begin{proof}
  If we can show $I(24,24)=\{24,2124\}$ whenever $n>3$, then we can use Remark \ref{rmk:relating0cells} to
  obtain $I(13,13)=\{13,132413\}$ when $n=4$.  Viewing $H_3$ as the parabolic subgroup of $H_4$
  generated by $\{1,2,3\}$ naturally, we can then use Lemma \ref{lemm:parabolic} to deduce that
  $I(13,13)=\{13\}$ in $H_3$. Thus, it suffices to prove that $I(24,24)\se \{24,2124\}$ in $H_n$
  whenever $n>3$. 

  Assume $n>3$ and let $w\in I(24,24)$. Let $\ul w$ be a reduced word of $w$.  We claim that the
  heap $H(w)=H(\ul w)$ does not contain a convex chain of the form $i\prec j\prec i'\prec j'$ where
  $e(i)=e(i')=1,e(j)=e(j')=2$ or where $e(i)=e(i')=2, e(j)=e(j')=1$. By Proposition \ref{prop:FCCriterion}, the
  claim implies that $\ul w$ is the reduced word of an element $w'$ in the Coxeter group of type
  $B_n$. Moreover, we have $\la(w')=n(w')=n(w)=2$ by Proposition \ref{prop:a.vs.n}.(3),  whence
  Theorem \ref{thm:CellViaDecomp}.(3) implies that $w'\in I_{B_n}(24,24)=\{24,2124\}$ by reduced word
  considerations. It follows that $w\in \{24,2124\}$ in $H_n$, so that $I(24,24)\se \{24,2124\}$ in
  $H_n$, as desired. 

  It remains to prove the claim, which we do by contradiction. Suppose $e(i)=e(i')=1$ and
  $e(j)=e(j')=2$ in the chain $C$. Since $1\notin\call(w)$, the element $i$ must cover an element
  $h$ with label $2$ in $H(w)$ by Lemma \ref{lemm:vertical}. Let $C=\{h,i,j,i',j'\}$ and let $k\in
  [h,j']$. Applying Lemma \ref{lemm:horizontal} to the chains $h\prec i\prec j$ and $j\prec i'\prec j'$,
  we note that if $e(k)\in \{1,2\}$ then $k\in C$. Also observe that $e(k)\neq 3$, for otherwise
  $\{i,k\}$ is a trapped antichain in $H(w)$ with relative to the special quadruple $(1,2,3,4)$,
  contradicting Proposition \ref{prop:trapped}.  Finally, we cannot have $e(k)>3$ either, because otherwise
  $e(k)$ and $2$ lie on different sides of the vertex 3 in $G$, so $h$ must be connected to $k$ by a
  chain in $[h,j']$ passing through an element labeled by 3, contradicting the last observation.  It
  follows that $[h,j']=C$ as sets.  But then the chain $h\prec i\prec j\prec i'\prec j'$ is convex,
  which contradicts the fact that $w\in \fc(H_n)$ by Proposition \ref{prop:FCCriterion}. A similar
  contradiction can be derived if $e(i)=e(i')=2$ and $e(j)=e(j')=1$, and the proof is complete. 
\end{proof}

\longthanks{
We thank Dana Ernst for useful discussions. We also thank the anonymous referee for
reading the paper carefully and suggesting many improvements.}

\bibliographystyle{amsplain-ac}

\bibliography{a2cells.bib}

\end{document}